\documentclass[12pt,a4paper,dvipsnames]{article}

\usepackage{hyperref}
\usepackage{pgf}
\usepackage[utf8]{inputenc}
\usepackage{amsfonts}
\usepackage{amssymb}
\usepackage{amsthm}
\usepackage{fixmath}

\usepackage{enumitem}

\usepackage{pgfplots}
\usepackage{dsfont}
\usepackage[linesnumbered]{algorithm2e}
\usepackage{setspace}
\usepackage{extarrows}
\usepackage[makeroom]{cancel}
\usepackage{graphicx}
\usepackage{multicol}
\PassOptionsToPackage{dvipsnames}{xcolor}
\usepackage{tikz,xcolor}

\usetikzlibrary{shapes,arrows}
\usetikzlibrary{hobby}
\usetikzlibrary{decorations.markings}

\usepackage[makeroom]{cancel}

\usepackage[english]{babel}		
\usepackage[T1]{fontenc}			
\usepackage{color}         			
\usepackage{listings}					

\usepackage{geometry}
\newtheorem{theorem}{Theorem}[section]
\newtheorem{lemma}[theorem]{Lemma}

\newtheorem{corollary}[theorem]{Corollary}
\newtheorem{remark}[theorem]{Remark}
\newtheorem{conjecture}[theorem]{Conjecture}
\newtheorem{question}[theorem]{Question}

\renewcommand{\P}{\mathbb{P}}
\newcommand{\TV}[1]{{\lVert #1 \rVert}_{\normalfont
\text{TV}}}
\newcommand{\N}{\mathbb{N}}
\newcommand{\R}{\mathbb{R}}
\newcommand{\E}{\mathbb{E}}
\newcommand{\Z}{\mathbb{Z}}
\newcommand{\V}{\text{\normalfont Var}}
\renewcommand{\c}{\text{\normalfont c}}
\newcommand{\diff}{\text{\normalfont d}}
\newcommand{\h}{\text{\normalfont h}}
\newcommand\abs[1]{\left| #1\right|}
\newcommand{\pdist}[2]{{\lVert #1 \rVert}_{#2}}

\title{Mixing times for the simple exclusion process with open boundaries}
\author{Nina Gantert$^{\ast}$, Evita Nestoridi$^{\ast\ast}$ and Dominik Schmid$^{\ast}$}
\date{\today}

\begin{document}
\tikzset{
    arrow/.style={postaction={decorate},
        decoration={markings,mark=at position 1 with
        {\arrow[line width=0.6mm]{>}}}} 
       }
\maketitle
\begin{abstract} We study mixing times of the symmetric and asymmetric simple exclusion process on the segment where particles are allowed to enter and exit at the endpoints. We consider different regimes depending on the entering and exiting rates as well as on the rates in the bulk, and show that the process exhibits pre-cutoff and in some cases cutoff.
Our main contribution is to study mixing times for the asymmetric simple exclusion process with open boundaries. We show that the order of the mixing time can be linear or exponential in the size of the segment depending on the choice of the boundary parameters, proving a strikingly different (and richer) behavior for the simple exclusion process with open boundaries than for the process on the closed segment. Our arguments combine coupling, second class particle and censoring techniques with current estimates. A novel idea is the use of multi-species particle arguments, where the particles only obey a partial ordering.
\end{abstract}

\phantom{.} \hspace{0.2cm}\textbf{Keywords:} Exclusion process, mixing times, coupling, second class particles \\
\phantom{.} \hspace{0.45cm} \textbf{MSC 2020:}  Primary: 60K35 Secondary: 60J27, 82C22

\let\thefootnote\relax
\footnotetext{ $^\ast$ \textit{Technical University of Munich, Germany. E-Mail}: \nolinkurl{nina.gantert@tum.de}, \nolinkurl{dominik.schmid@tum.de} \\\phantom{.} \hspace{0.3cm}
$^{\ast\ast}$ \textit{Princeton University, United States. E-Mail}: \nolinkurl{exn@princeton.edu}, \\ 
\phantom{.} \hspace{0.85cm}Partially supported by EPSRC grant EP/R022615/1}

\section{Introduction}

The simple exclusion process is an important and intensively studied interacting particle system \cite{BS:OrderCurrent,BE:Nonequilibrium, CW:TableauxCombinatorics, DEHP:ASEPCombinatorics, LL:CutoffASEP, L:CutoffSEP}.
%
%
%
%
Over the last decades, it equally raises interest of scientists from probability, statistical mechanics and combinatorics, see \cite{BE:Nonequilibrium, L:ReviewASEP, L:Book2, WBE:CombinatorialMappings} for review papers in the respective areas. Despite its simple construction, the simple exclusion process is a source for surprising phenomena such as phase transitions and formation of shocks 
\cite{DEHP:ASEPCombinatorics, F:ShockFluctuations, FF:ShockFluctuations, FKS:MicroscopicStructure,  USW:PASEPcurrent}. 
In this paper, we study the simple exclusion process with open boundaries which is given as independently moving random walks on the segment using an exclusion rule, i.e.\ when a particle tries to move to a site, which is already occupied, this move is suppressed. In addition, we allow particles to jump in and out of the system at the ends of the segment. 
We determine the order of the mixing times for this process, which quantify the speed of convergence to equilibrium, see \eqref{def:MixingTime}. 
Mixing times for simple exclusion processes have been thoroughly
studied, see \cite{BBHM:MixingBias, LL:CutoffASEP,L:CutoffSEP, L:CycleDiffusiveWindow, W:MixingLoz}. Note that in
all of the above mentioned works on mixing times, the number of particles in the segment
is preserved and the simple exclusion process is reversible (in the sense of detailed balance). 
In general, the simple exclusion process with open boundaries is no longer reversible. It is one of the most basic, however very interesting examples of a non-equilibrium particle system in statistical mechanics. \\

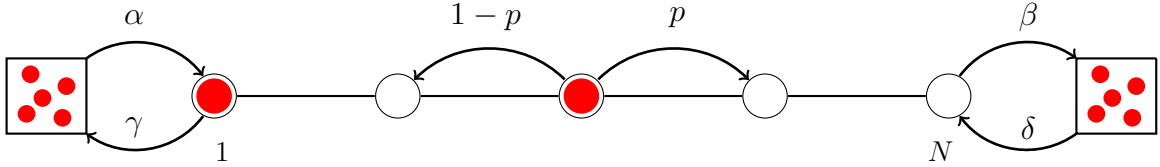
\begin{figure} \label{fig:BDEP}
\centering
\begin{tikzpicture}[scale=1.05]

\def\spiral[#1](#2)(#3:#4:#5){
\pgfmathsetmacro{\domain}{pi*#3/180+#4*2*pi}
\draw [#1,
       shift={(#2)},
       domain=0:\domain,
       variable=\t,
       smooth,
       samples=int(\domain/0.08)] plot ({\t r}: {#5*\t/\domain})
}

\def\particles(#1)(#2){

  \draw[black,thick](-3.9+#1,0.55-0.075+#2) -- (-4.9+#1,0.55-0.075+#2) -- (-4.9+#1,-0.4-0.075+#2) -- (-3.9+#1,-0.4-0.075+#2) -- (-3.9+#1,0.55-0.075+#2);
  
  	\node[shape=circle,scale=0.6,fill=red] (Y1) at (-4.15+#1,0.2-0.075+#2) {};
  	\node[shape=circle,scale=0.6,fill=red] (Y2) at (-4.6+#1,0.35-0.075+#2) {};
  	\node[shape=circle,scale=0.6,fill=red] (Y3) at (-4.2+#1,-0.2-0.075+#2) {};
   	\node[shape=circle,scale=0.6,fill=red] (Y4) at (-4.45+#1,0.05-0.075+#2) {};
  	\node[shape=circle,scale=0.6,fill=red] (Y5) at (-4.65+#1,-0.15-0.075+#2) {}; }

  \def\annhil(#1)(#2){	  \spiral[black,thick](9.0+#1,0.09+#2)(0:3:0.42);
  \draw[black,thick](8.5+#1,0.55+#2) -- (9.5+#1,0.55+#2) -- (9.5+#1,-0.4+#2) -- (8.5+#1,-0.4+#2) -- (8.5+#1,0.55+#2); }

	\node[shape=circle,scale=1.5,draw] (B) at (2.3,0){} ;
	\node[shape=circle,scale=1.5,draw] (C) at (4.6,0) {};
	\node[shape=circle,scale=1.2,fill=red] (CB) at (2.3,0) {};
    \node[shape=circle,scale=1.5,draw] (A) at (0,0){} ;
 	\node[shape=circle,scale=1.5,draw] (D) at (6.9,0){} ; 
 	 	\node[shape=circle,scale=1.5,draw] (Z) at (-2.3,0){} ;
	\node[shape=circle,scale=1.2,fill=red] (YZ) at (-2.3,0) {};
   \node[line width=0pt,shape=circle,scale=1.6] (B2) at (2.3,0){};
	\node[line width=0pt,shape=circle,scale=2.5] (D2) at (6.9,0){};
		\node[line width=0pt,shape=circle,scale=2.5] (Z2) at (-2.3,0){};
		
	\node[line width=0pt,shape=circle,scale=2.5] (X10) at (6.8,0){};
		\node[line width=0pt,shape=circle,scale=2.5] (X11) at (-2.2,0){};	
			\node[line width=0pt,shape=circle,scale=2.5] (X12) at (8.4,0){};
		\node[line width=0pt,shape=circle,scale=2.5] (X13) at (-3.8,0){};

		\draw[thick] (Z) to (A);	
	\draw[thick] (A) to (B);	
		\draw[thick] (B) to (C);	
  \draw[thick] (C) to (D);

\particles(0)(0);
\particles(6.9+4.9+1.6)(0);


\draw [->,line width=1pt]  (B2) to [bend right,in=135,out=45,arrow] (C);
  
   \draw [->,line width=1pt] (B2) to [bend right,in=-135,out=-45] (A);
    \node (text1) at (3.5,1){$p$} ;    
	\node (text2) at (1.1,1){$1-p$} ;   
	\node (text3) at (-2.3-1,1){$\alpha$}; 
	\node (text4) at (6.9+1,1){$\beta$};   
	\node (text5) at (-2.3-1,-0.4){$\gamma$}; 
	\node (text6) at (6.9+1,-0.4){$\delta$};   

    \node[scale=0.9] (text1) at (-2.2,-0.7){$1$} ;    
    \node[scale=0.9] (text1) at (6.8,-0.7){$N$} ;   
  	
  \draw [->,line width=1pt] (-3.9,0.475) to [bend right,in=135,out=45] (Z);
   \draw [->,line width=1pt] (Z) to [bend right,in=135,out=45] (-3.9,-0.475); 
   \draw [->,line width=1pt] (6.9+1.6,-0.475) to [bend right,in=135,out=45] (D);	
   \draw [->,line width=1pt] (D) to [bend right,in=135,out=45] (6.9+1.6,0.475);

	\end{tikzpicture}	
\caption{Simple exclusion process with open boundaries for parameters $(p,\alpha,\beta,\gamma,\delta)$.}
 \end{figure}
While the proofs in the symmetric cases of the simple exclusion process with open boundaries follow known routes, see Theorems \ref{thm:symmetricTwoSided} and \ref{thm:symmetricOneSide} where we adopt the arguments of \cite{L:CutoffCircle} and \cite{L:CutoffSEP}, our main contribution is to study mixing times for the asymmetric simple exclusion process with open boundaries, see Theorems \ref{thm:asymmetricOneSide} to \ref{thm:MaxCurrentTriplePoint}. We show that the order of the mixing time can be linear or exponential in $N$, depending on the choice of the boundary parameters. Our arguments combine coupling, second class particle and censoring techniques with current estimates. A novel idea is the use of multi-species particle arguments, where the particles only obey a partial ordering. 
In general, a main difficulty is to write down explicitly the stationary distribution of the simple exclusion process with open boundaries. Physicists and combinatorialists have been working hard to
acquire descriptions of the stationary measure, see Section \ref{relatedsection}. When the
stationary measure is hard to describe, a nice alternative is to
simulate it by running a Markov chain. Our results allow to determine how many steps are required  when running the specific Markov chain given by the simple exclusion
process with open boundaries. \\

We now define the simple exclusion process with drift parameters $p,q \geq 0$. Let $k \in [N]:= \lbrace 1,\dots,N\rbrace$ for some $N \in \N$. The \textbf{simple exclusion process} on a segment of size $N$ with $k$ particles is a Feller process $(\eta^{\textup{ex}}_t)_{t \geq 0}$ with state space $\Omega_{N,k}$ given by
\begin{equation}\label{def:SpaceParticles}
\Omega_{N,k} := \Big\lbrace \eta \in \lbrace 0,1 \rbrace^{N} \colon \sum_{x=1}^{N} \eta(x)= k\Big\rbrace \ . 
\end{equation}
It is generated by 
\begin{align*}
\mathcal{L}_{\textup{ex}}f(\eta) &= \sum_{x =1}^{N-1} p \ \eta(x)(1-\eta(x+1))\left[ f(\eta^{x,x+1})-f(\eta) \right] \nonumber \\
&+ \sum_{x =2}^{N} q \ \eta(x)(1-\eta(x-1))\left[ f(\eta^{x,x-1})-f(\eta) \right]  
\end{align*} where $\eta^{x,y}\in \Omega_{N,k}$ denotes the configuration in which we exchange the values at positions $x$ and $y$ in $\eta \in \Omega_{N,k}$. For an introduction to Feller processes, we refer to \cite{L:Book2}. \\

We say that site $x$ is \textbf{occupied} by a particle if $\eta(x)=1$ and \textbf{vacant} otherwise. 
A particle at a vertex $x$ is supposed to move to the right at rate $p$ and to the left at rate $q$ whenever the target is a vacant site. 
For the \textbf{simple exclusion process with open boundaries} $(\eta_t)_{t \geq 0}$, we in addition allow creating and annihilating particles at the endpoints of the segment. More precisely, for parameters $\alpha,\beta,\gamma,\delta \geq 0$, $(\eta_t)_{t \geq 0}$ is defined as the Feller process with state space $\Omega_N:=\{ 0,1\}^N$ generated by 
\begin{align*}
\mathcal{L}f(\eta) = \mathcal{L}_{\textup{ex}}f(\eta)  &+ \alpha (1-\eta(1))\ \left[ f(\eta^{1})-f(\eta) \right] \hspace{2pt} + \gamma \eta(1)\ \ \left[ f(\eta^{1})-f(\eta) \right] \\
  &+ \delta (1-\eta(N))\left[ f(\eta^{N})-f(\eta) \right]  + \beta \eta(N)\left[ f(\eta^{N})-f(\eta) \right]  
\end{align*} where $\eta^{x}\in \Omega_N$ denotes the configuration in which we flip the values at position $x$ in $\eta \in \Omega_N$.  In contrast to the simple exclusion process, the number of particles will in general no longer be preserved over time.  \\

In the remainder, we assume that the above parameters are chosen such that the corresponding simple exclusion process with open boundaries has a unique stationary distribution $\mu$ (which may also be a Dirac measure on a single configuration). Our goal is to investigate the speed of convergence towards $\mu$. For this purpose, we define the $\mathbold\varepsilon$\textbf{-mixing time} of  $(\eta_t)_{t \geq 0}$ by
\begin{equation}\label{def:MixingTime}
t^N_{\text{\normalfont mix}}(\varepsilon) := \inf\left\lbrace t\geq 0 \ \colon \max_{\eta \in \Omega_{N}} \TV{\P\left( \eta_t \in \cdot \ \right | \eta_0 = \eta) - \mu} < \varepsilon \right\rbrace
\end{equation} for all $\varepsilon \in (0,1)$. Here, $\TV{ \ \cdot \ }$ denotes the \textbf{total-variation distance}, i.e.\ for two probability measures $\mu$ and $\nu$ on $\Omega_N$, we define
\begin{equation}\label{def:TVDistance}
\TV{ \mu - \nu } := \frac{1}{2}\sum_{x \in \Omega_N} \abs{\mu(x)-\nu(x)} = \max_{A \subseteq \Omega_N} \left(\mu(A)-\nu(A)\right) \ .
\end{equation} Our goal is to study the order of $t^N_{\text{\normalfont mix}}(\varepsilon)$ when $N$ goes to infinity.

\subsection{Main results}

In the following, we investigate the mixing times for the simple exclusion process with open boundaries. Without loss of generality, we can assume that $q=1-p$ holds for some $p \in [\frac{1}{2},1]$. To see this, we rescale time by a factor of $(p+q)$ and use the symmetry in the definition of $(\eta_t)_{t \geq 0}$ with respect to the boundary parameters. Moreover, we  assume that $\max(\alpha,\beta,\gamma,\delta)>0$ holds. When all boundary parameters are zero, mixing times were investigated in \cite{BBHM:MixingBias,LL:CutoffASEP,L:CutoffSEP,W:MixingLoz} among others. 

\subsubsection{Symmetric simple exclusion process with open boundaries}

We start with the case when all transitions in the bulk are symmetric, i.e.\ $p=\frac{1}{2}$.

\begin{theorem}\label{thm:symmetricTwoSided} 
For $p=\frac{1}{2}$, the $\varepsilon$-mixing time of the simple exclusion process with open boundaries is
\begin{equation}\label{eq:SymmetricPrecutoff}
\frac{1}{\pi^2} \leq \liminf_{N \rightarrow \infty} \frac{t^N_{\textup{mix}}(\varepsilon)}{N^{2}\log N} \leq \limsup_{N \rightarrow \infty} \frac{t^N_{\textup{mix}}(\varepsilon)}{N^{2}\log N } \leq C
\end{equation} for all $\varepsilon\in (0,1)$ and some constant $C=C(\alpha,\beta,\gamma,\delta)$. 
\end{theorem}
The property that the first order of the $\varepsilon$-mixing times can be bounded within two constants which do not depend on $\varepsilon$ is called \textbf{pre-cutoff}, see \cite[Chapter 18]{LPW:markov-mixing}. When all boundary parameters are zero and particles have a density in $(0,1)$, it was shown in \cite[Theorem 2.4]{L:CutoffSEP} that the lower bound in \eqref{eq:SymmetricPrecutoff} gives the asymptotic behavior of the $\varepsilon$-mixing time for the simple exclusion process. However, the next theorem says that when particles enter and exit only at a single side, we see a different constant.

\begin{theorem}\label{thm:symmetricOneSide}
For $p=\frac{1}{2}$, suppose that $\max(\alpha,\gamma)=0$ and $\min(\beta,\delta)>0$ holds. Then for all $\varepsilon\in (0,1)$, the $\varepsilon$-mixing time of the simple exclusion process with open boundaries satisfies
\begin{equation}\label{eq:SymmetricCutoff}
\lim_{N \rightarrow \infty} \frac{t^N_{\textup{mix}}(\varepsilon)}{N^{2}\log N} = \frac{4}{\pi^2} \, .
\end{equation} By symmetry, \eqref{eq:SymmetricCutoff} holds for $p=\frac{1}{2}$, $\min(\alpha,\gamma)>0$ and $\max(\beta,\delta)=0$ as well.
\end{theorem}

The property that the leading order of the $\varepsilon$-mixing times does not depend on $\varepsilon$ is known as the \textbf{cutoff phenomenon}, see \cite[Chapter 18]{LPW:markov-mixing}.

\subsubsection{Asymmetric simple exclusion process with one blocked entry}

Next, consider the asymmetric simple exclusion process with $p>\frac{1}{2}$. When $\alpha>0$, let
\begin{equation}\label{def:a}
a=a(\alpha,\gamma,p) := \frac{1}{2 \alpha}\left( 2p-1 -\alpha+ \gamma + \sqrt{ (2p-1 -\alpha+ \gamma)^2 +4\alpha\gamma}\right)
\end{equation} and similarly, for $\beta>0$, we set
\begin{equation}\label{def:b}
b=b(\beta,\delta,p) := \frac{1}{2 \beta}\left( 2p-1 -\beta+ \delta + \sqrt{ (2p-1 -\beta+ \delta)^2 +4\beta\delta} \right) \ . 
\end{equation} 
We study the case where we have one blocked entry, i.e.\ $\min(\alpha,\beta)=0$  and $\max(\alpha,\beta)> 0$.
\begin{theorem}\label{thm:asymmetricOneSide} 
Suppose that $p>\frac{1}{2}$,  and let $\gamma,\delta \geq 0$ be arbitrary. If $\alpha=0$ and $\beta>0$, then for all $\varepsilon\in (0,1)$, the $\varepsilon$-mixing time of the simple exclusion process with open boundaries satisfies 
\begin{equation}\label{eq:OneBlockeda}
\lim_{N \rightarrow \infty} \frac{t^N_{\textup{mix}}(\varepsilon)}{N} = \frac{(\max(b,1)+1)^2}{(2p-1)\max(b,1)} \ .
\end{equation} Similarly, for $\alpha>0$ and $\beta=0$, we have that
\begin{equation}\label{eq:OneBlockedb}
\lim_{N \rightarrow \infty} \frac{t^N_{\textup{mix}}(\varepsilon)}{N} = \frac{(\max(a,1)+1)^2}{(2p-1)\max(a,1)} \ .
\end{equation} In particular, we see in both cases that cutoff occurs.
\end{theorem} 
We will see that a key ingredient for the proof of Theorem \ref{thm:asymmetricOneSide} is to understand the creation of shocks, which is typical for the asymmetric simple exclusion process. The shocks will travel at a linear speed and give rise to a sharp mixing behavior.

\subsubsection{The reverse bias phase for the simple exclusion process}

In contrast to the simple exclusion process where all boundary parameters are zero, there exists a regime of the asymmetric simple exclusion process with open boundaries with an exponentially large $\varepsilon$-mixing time. This case is known in the literature as the \textbf{reverse bias phase}, see  \cite{BE:Nonequilibrium}. This terminology can be intuitively justified since the particles are forced by the boundary conditions to move against their natural drift direction.

\begin{theorem}\label{thm:asymmetricOneSideReverse} 
Suppose that $\max(\alpha,\beta)=0$ and $p\in \left( \frac{1}{2},1 \right)$ holds. Then for all $\varepsilon \in (0,\frac{1}{2})$, we have that \begin{equation}\label{eq:LogCutoffOneSide}
\lim_{N \rightarrow \infty} \frac{\log\left(t^N_{\textup{mix}}(\varepsilon)\right)}{N}= \log\left( \frac{p}{1-p}\right) 
\end{equation} holds whenever $\min(\gamma,\delta)=0$ and $\max(\gamma,\delta)>0$.
If $\min(\gamma,\delta)>0$ holds, we have that 
\begin{equation}\label{eq:LogCutoffTwoSides}
\lim_{N \rightarrow \infty} \frac{\log\left(t^N_{\textup{mix}}(\varepsilon)\right)}{N}= \frac{1}{2}\log\left( \frac{p}{1-p}\right) \ .
\end{equation}
\end{theorem}

\subsubsection{The high and low density phase for the simple exclusion process}

Now suppose that  $\min(\alpha,\beta)>0$ and $p>\frac{1}{2}$, so the quantities $a$ and $b$ from \eqref{def:a} and \eqref{def:b} are both well-defined.
We distinguish three different regimes according to the density within the stationary distribution, see Section \ref{sec:PreliminariesBDEP} for more details. The regime $a>\max(b,1)$ is called the \textbf{low density phase} of the exclusion process, while we refer to the regime $b> \max(a,1)$ as the \textbf{high density phase}. The remaining case where $\max(a,b) \leq 1$ holds is called the \textbf{maximal current phase}. Intuitively, the invariant distribution is an interpolation between two Bernoulli-product measures with densities $\frac{1}{1+a}$ and $\frac{b}{1+b}$, respectively, and we will see a justification of this claim in Lemma \ref{lem:ComparsionStationaryMeasure}. The terminology low density phase (respectively high density phase) will be justified in Lemma \ref{lem:stationaryDensity}, 
since the density within the invariant measure stays below (respectively above) $\frac{1}{2}$. 
\begin{theorem}\label{thm:asymmetricTwoSides} For parameters $\alpha,\beta>0$ and $\gamma,\delta\geq 0$, as well as $p> \frac{1}{2}$, suppose we are in the high density phase. Then there exists a constant $C_h=C_h(a,b,p)>0$ such that the $\varepsilon$-mixing time of the simple exclusion process with open boundaries satisfies
\begin{equation}\label{eq:MixingTimeLowDensity}
\frac{1}{2p-1} \leq \liminf_{N \rightarrow \infty} \frac{t^N_{\textup{mix}}(\varepsilon)}{N} \leq \limsup_{N \rightarrow \infty} \frac{t^N_{\textup{mix}}(\varepsilon)}{N} \leq C_h
\end{equation} for all $\varepsilon \in (0,1)$. Similarly, when we are in the low density phase with parameters $\alpha,\beta>0$ and $\gamma,\delta \geq 0$, as well as $p> \frac{1}{2}$, the $\varepsilon$-mixing time of the simple exclusion process with open boundaries satisfies
\begin{equation}\label{eq:MixingTimeHighDensity}
\frac{1}{2p-1} \leq \liminf_{N \rightarrow \infty} \frac{t^N_{\textup{mix}}(\varepsilon)}{N} \leq \limsup_{N \rightarrow \infty} \frac{t^N_{\textup{mix}}(\varepsilon)}{N} \leq C_\ell
\end{equation} for some constant $C_\ell=C_\ell(a,b,p)>0$ and all $\varepsilon \in (0,1)$. In particular, pre-cutoff occurs.
\end{theorem}
For a discussion of the remaining cases and sharp constants, we refer to Section \ref{openproblems}.

\subsubsection{The triple point of the simple exclusion process}

An interesting special case of the simple exclusion process with open boundaries is the \textbf{triple point} where $p> \frac{1}{2}$ and $a=b=1$ holds (which means $\alpha-\gamma = p-\frac{1}{2}$, $\beta -\delta = p-\frac{1}{2}$.) Intuitively, the low-density phase, the high density phase and maximal current phase coexist at the triple point, and it can be shown that the process gives rise to the KPZ equation \cite{CS:OpenASEPWeakly, P:KPZlimit}. We have the following bound on the mixing time.
\begin{theorem}\label{thm:MaxCurrentTriplePoint} Suppose that $p> \frac{1}{2}$ and $a=b=1$ holds, i.e.\ we are in the triple point. For all $\varepsilon \in (0,1)$, the $\varepsilon$-mixing time of the simple exclusion process with open boundaries satisfies 
\begin{equation}
t^N_{\textup{mix}}(\varepsilon) \leq C N^{3}
\end{equation}
 for some constant $C=C(\alpha,\beta,\gamma,\delta,p)$.
\end{theorem}

\subsection{Open problems}\label{openproblems}

We saw in Theorems \ref{thm:symmetricTwoSided} and \ref{thm:symmetricOneSide} that the symmetric simple exclusion process has pre-cutoff for all non-trivial choices of boundary parameters.
\begin{conjecture} Let $p=\frac{1}{2}$ and $\alpha,\beta,\gamma,\delta\geq 0$ with $\max(\alpha,\gamma)>0$ and $\max(\beta,\delta)>0$. Then the lower bound in \eqref{eq:SymmetricPrecutoff} is sharp, and cutoff occurs.
\end{conjecture}
Intuitively, we treat the simple exclusion process with two open boundaries as a symmetric simple exclusion process on the circle of length $2N$ and $N$ particles, see Section \ref{sec:LowerBounds}. In the high density and low density phase, we have the  following conjecture.
\begin{conjecture}\label{conj:HighDens} Under the assumptions of Theorem \ref{thm:asymmetricTwoSides}, the mixing time in the high-density phase satisfies for all $\varepsilon \in (0,1)$
\begin{equation}\label{eq:ConjectureConstant}
\lim_{N \rightarrow \infty} \frac{t^N_{\textup{mix}}(\varepsilon)}{N} =  \frac{(b+1)(b-1)(\hat{a}+1)^2}{(b-\hat{a})(b\hat{a}-1)(2p-1)} 
\end{equation} where $\hat{a}:=\max(a,1)$. A similar statement holds for the low density phase.
\end{conjecture}
Let us give some heuristics on this conjecture for the high density phase. Suppose we start from the empty initial configuration, and wait until we see the equilibrium density of $\frac{b}{b+1}$ within the segment, see Lemma \ref{lem:current}. Similar to the hydrodynamic limits in \cite{LL:CutoffASEP}, we expect at time $(b+1)(b-a)^{-1}(2p-1)^{-1}n$ to see a density which is $\frac{1}{a+1}$ at $1$, $\frac{1}{b+1}$ at $n$ and linearly interpolated in between. After this time, the right boundary creates a shock wave traveling to site $1$ which supports the conjecture of cutoff. The total travel time of this shock can be computed by comparing the current at both endpoints. Note that in the maximum current phase, no such shock is created, and the particles can travel at the maximal possible speed of $\frac{1}{4}(2p-1)$ (justifying the name maximal current phase). The mixing time is governed by second class particle fluctuations, see Remark \ref{rem:IntuitionMax}, and we conjecture the following behavior.
\begin{conjecture}\label{conj:MaxCurrent} When $\max(a,b)\leq 1$ holds (including the triple point), we have that the $\varepsilon$-mixing time of the simple exclusion process with open boundaries is of order $N^{3/2}$ for all $\varepsilon \in (0,1)$. Moreover, the cutoff phenomenon does not occur.
\end{conjecture} When $p=1$ and $\gamma=\delta
=0$, the mixing time in the maximal current phase was recently determined in \cite{S:MixingTASEP} up to a logarithmic factor. For $a=b>1$ and $p >\frac{1}{2}$, called \textbf{coexistence line}, we see that the right-hand side of \eqref{eq:ConjectureConstant} in Conjecture \ref{conj:HighDens} blows up.
\begin{question} What is the order of the $\varepsilon$-mixing time of the simple exclusion process with open boundaries in the coexistence line, and does the cutoff phenomenon occur?
\end{question}

\subsection{Related work}\label{relatedsection}

The simple exclusion process can be seen from various different perspectives. Historically, the simple exclusion process is motivated in physics and biology as a model for lattice gases, but it can also be used to describe traffic flow or kinetics of protein synthesis \cite{H:TrafficFlows,MGP:Kinetics}. 
In a mathematical context, it was introduced by Spitzer \cite{S:InteractionMP}. Depending on the parameters of the simple exclusion process with open boundaries, it is found under different names, such as (totally/partially) asymmetric simple exclusion process or boundary driven simple exclusion process. \\

In this paper, we focus on investigating the speed of convergence to the stationary distribution. This is done by analyzing the total-variation mixing time, see \cite{LPW:markov-mixing} for a comprehensive introduction. 
In the case of the symmetric simple exclusion process (SSEP), i.e.\ when $p=\frac{1}{2}$ holds, the first order of the mixing time was determined using spectral techniques for the lower bound in \cite{W:MixingLoz} and a clever combination of various properties of the SSEP for the upper bound in \cite{L:CutoffSEP}. 
For the asymmetric simple exclusion process (ASEP), Benjamini et al.\ showed in \cite{BBHM:MixingBias} that the mixing time is linear in the size of the segment using the simple exclusion process on the integers and second class particle arguments, see below.
We will see that second class particle arguments play a crucial role in our analysis of the simple exclusion process with open boundaries in Sections \ref{sec:UpperBounds} to \ref{sec:HighLowDensity}. Recently, the cutoff phenomenon was established for the ASEP in \cite{LL:CutoffASEP}. More generally, mixing times for the simple exclusion process were investigated in size-dependent or random environments \cite{LL:CutoffWeakly,LP:MixingSmallBias,S:MixingBallistic} as well as on general graphs \cite{HP:EPmixing,O:MixingGeneral}. 
All these investigations have in common that the underlying simple exclusion process is reversible. Many techniques for precise bounds on the mixing time require reversibility, and can in general not be applied for the simple exclusion process with open boundaries. To our best knowledge,  mixing times for a non-reversible simple exclusion process were so far only investigated for the totally asymmetric simple exclusion process on the circle \cite{F:EVBoundsSEP}. \\

The simple exclusion process with open boundaries and a non-reversible stationary distribution is one of the  simplest examples of a non-equilibrium system. This observation is quantified by studying currents for the simple exclusion process with open boundaries, see Section \ref{sec:Current}. For the symmetric simple exclusion process, currents were investigated in \cite{LMO:StationaryBDEP}. For the simple exclusion process with general parameters, the first order of the current was determined in \cite{USW:PASEPcurrent} using Askey-Wilson polynomials, extending the results of \cite{BECE:ExactSolutionsPASEP}. Current fluctuations for the asymmetric simple exclusion process with open boundaries are investigated in \cite{GLMV:CurrentStatisticsPASEP,L:MatrixAnsatz}  
while related spectral properties are discussed in \cite{GE:BetheAnsatzPASEP,GE:ExactSpectralGap,GE:RelaxationModePASEP} among others. Furthermore, there are deep connections to the KPZ universality class, see for example  \cite{CD:ASEPline,CS:OpenASEPWeakly, P:KPZlimit}.
\\

Note that the simple exclusion process naturally extends to a Feller process on the integers. It is a classical result that the Bernoulli product measures are invariant in this case, see \cite{L:Book2}. For the asymmetric simple exclusion process on the integers, the moments of the current are closely linked to the  motion of second class particles \cite{BS:ExactConnections}. 
In particular, the fluctuations of the current at time $t\geq 0$ are given by the mean of the displacement of a single second class particle started from the origin within the Bernoulli product measure. 
Depending on the parameter of the product measure, we see either a diffusive or a super-diffusive behavior, see \cite{BS:OrderCurrent,FF:CurrentFluctuations,PS:CurrentFluctuations}. We note that currents are also studied for the simple exclusion process with open boundaries containing second class particles, see \cite{ CMRV:Integrability,U:TwoSpeciesPASEP}. 
Furthermore, second class particles can be used to identify shocks  \cite{F:ShockFluctuations,FF:ShockFluctuations,FKS:MicroscopicStructure}. More precisely, for an initial distribution with a shock, i.e.\ for two product measures with different parameters, we place a second class particle at the transition point. Under certain assumptions on the parameters of the product measures, one can show that the second class particle will stay close to the shock location for all times. In this paper, we will see a similar shock behavior for the asymmetric simple exclusion process with one blocked entry, see  Section \ref{sec:ShockwaveLemma}. \\

Another natural quantity to study is the invariant measure of the simple exclusion process with open boundaries, see \cite[Section III.3]{L:Book2}. Various beautiful representations were achieved in statistical mechanics and combinatorics. A key tool is the matrix product ansatz, which is in an implicit form already given in \cite{L:ErgodicI} and was successfully applied for the simple exclusion process with open boundaries in \cite{DEHP:ASEPCombinatorics} when particles can move only in one direction. 
Informally speaking, we assign in the matrix product ansatz to every configuration a weight which consists of a product of matrices and vectors. The matrices and vectors must satisfy certain relations, usually called the DEHP algebra, see \cite{DEHP:ASEPCombinatorics}. The matrix product ansatz allows us to study the mean current, the density profile and correlations within the stationary distribution, see \cite{S:DensityProfilePASEP, USW:PASEPcurrent,UW:Correlations}.
Representing the weights in the matrix product ansatz is a question in combinatorics which gained lots of recent attention. It led to beautiful descriptions such as (weighted) Catalan paths and staircase tableaux, see \cite{BCEPR:CombinatoricsPASEP,CW:TableauxCombinatorics, M:TASEPCombinatorics}. 
Building on the works of Sasomoto  \cite{S:PASEPpolynomials} and Uchiyama et al.\ in \cite{USW:PASEPcurrent}, the representations are closely related to Askey-Wilson polynomials. Similar representations were achieved for the simple exclusion process with second class particles using Koornwinder polynomials, see \cite{C:Koornwinder,CMW:TwoSpecies}. Recently, combinatorial representations were established for the multi-species simple exclusion process, i.e.\ for more than two different kinds of particles, see  \cite{CGGW:KoornwinderMulti,FRV:MultispeciesMatrixProduct,M:CombinatoricsMultispecies}. 

\subsection{Outline of the paper}

This paper is organized as follows. In Section \ref{sec:PreliminariesBDEP}, we state preliminaries on the simple exclusion process from different perspectives. In Sections \ref{sec:LowerBounds} and \ref{sec:UpperBounds}, we study mixing times of the symmetric simple exclusion process with open boundaries. Lower bounds will be achieved by using a continuous-time version of a
generalization of Wilson's lemma which was introduced in \cite{NN:CutoffCyclic}. A general upper bound will follow from a comparison to independent simple random walks. This bound is refined in the special case of one open boundary following closely the ideas of Lacoin in \cite{L:CutoffSEP}. 
The analysis of mixing times for the asymmetric simple exclusion process is carried out in Sections \ref{sec:ProofAsymmetricHomogenousOneSide} to \ref{sec:HighLowDensity}. In Section \ref{sec:ProofAsymmetricHomogenousOneSide}, we use second class particle and current arguments to investigate mixing times for the ASEP with one blocked entry. The reverse bias phase is considered in Section \ref{sec:ReverseBias} requiring second class particle estimates and a comparison with the simple exclusion process on the integers. Section \ref{sec:HighLowDensity} is dedicated to the study of the simple exclusion process within the low density and the high density phase using multi-species exclusion processes, stochastic orderings and the censoring inequality. The triple point for the simple exclusion process is treated in Section \ref{sec:MaxCurrent} using a symmetrization argument.

\section{Preliminaries on the simple exclusion process}\label{sec:PreliminariesBDEP}

In this section, we collect basic properties and techniques for the simple exclusion process which will be used at multiple points during the proofs. This includes couplings, second class particles, the simple exclusion process on the integers, currents, invariant measures and the censoring inequality. Motivations and applications of these techniques come from probability theory, statistical mechanics and combinatorics. For convenience, we give a brief background to the different techniques and point out where we require generalizations of the quoted results.

\subsection{The canonical coupling}\label{sec:canonicalCouplingOneSide}

A main tool in our arguments is a \textbf{grand coupling} for the simple exclusion process with open boundaries, i.e.\ a joint realization of the simple exclusion process for all initial configurations simultaneously. Couplings are a well-known technique in order to bound mixing times, see \cite[Chapter 5]{LPW:markov-mixing}. In the following, we consider a specific coupling, the \textbf{canonical coupling} of the simple exclusion process with open boundaries, sometimes called basic or standard coupling. Similar couplings are constructed in  \cite{BBHM:MixingBias} and \cite{S:MixingBallistic} for the simple exclusion process on the closed segment. For the simple exclusion process with open boundaries, the canonical coupling is given as follows: \\

We place rate $1$ Poisson clocks on all edges $e\in E$. Whenever the clock of an edge $e=\{ x,x+1\}$ rings, we sample a Uniform-$[0, 1]$-random variable $U$ independently of all previous samples and distinguish two cases.
\begin{itemize}
\item If $U\leq p$ and $\eta(x)=1-\eta(x+1)=1$ holds, we move the particle at site $x$ to site $x+1$ in configuration $\eta$. 
\item If $U> p$ and $\eta(x)=1-\eta(x+1)=0$ holds, we move the particle at site $x+1$ to site $x$  in configuration $\eta$.
\end{itemize} 
In addition, we place a rate $\alpha$ Poisson clock (a rate $\gamma$ Poisson clock) on the vertex $1$. Whenever a clock rings, we place a particle (an empty site) at site $1$, independently of the current value of $\eta(1)$. Similarly, we put a rate $\beta$ Poisson clock (a rate $\delta$ Poisson clock) on the vertex $N$. Whenever this clock rings, we place an empty site (a particle) at site $N$ independently of the current value of $\eta(N)$. 

\subsubsection{The component-wise partial order}\label{sec:ComponentwiseOrder}
The canonical coupling is constructed in such a way that it respects the partial order $\succeq_\c$ on $\Omega_N$ which is given by component-wise comparison, i.e.\ 
\begin{equation}\label{def:partialOrder}
\eta \succeq_\c \zeta \quad \Leftrightarrow \quad \eta(i) \geq \zeta(i) \text{ for all } i \in [N]
\end{equation}
for all $\eta,\zeta \in \Omega_N$. Moreover, the canonical coupling $\mathbf{P}$ can be extended such that it is monotone in $\alpha,\beta,\gamma,\delta$. These observations are formalized in the following lemma.
\begin{lemma}\label{lem:MonotoneCouplingComponentwise} Consider two exclusion processes $(\eta_t)_{t\geq 0}$ and $(\zeta_t)_{t\geq 0}$ on the segment of size $N$ with parameters $(p,\alpha,\beta,\gamma,\delta)$ and $(p,\alpha^{\prime},\beta^{\prime},\gamma^{\prime},\delta^{\prime})$, respectively. Suppose that
\begin{equation}\label{eq:LemmaMonotoneAssumptions}
\alpha \geq \alpha^{\prime} \quad \beta\leq \beta^{\prime} \quad \gamma \leq \gamma^{\prime} \quad \text{and} \quad \delta\geq \delta^{\prime}
\end{equation}
 holds, then the canonical coupling $\mathbf{P}$ can be extended such that
\begin{equation}
\mathbf{P}\left( \eta_t \succeq_\c \zeta_t  \text{ for all } t \geq 0 \mid \eta_0 \succeq_\c \zeta_0 \right) = 1 \ . 
\end{equation}
\end{lemma}
\begin{proof} We give an explicit construction of the extended canonical coupling $\mathbf{P}$. Since $p=p^{\prime}$, observe that the canonical coupling preserves the partial order $\succeq_\c$ for all transitions along edges. Hence, it remains to specify $\mathbf{P}$ at the boundary. For $(\eta_t)_{t\geq 0}$ and $(\zeta_t)_{t\geq 0}$, use the same rate $\alpha^{\prime}$ Poisson clocks to determine when a particle enters at the left-hand side boundary. In addition, when $\alpha>\alpha^{\prime}$ holds, insert particles at the leftmost site in $(\eta_t)_{t\geq 0}$ according to an independent rate $(\alpha-\alpha^{\prime})$ Poisson clock. A similar construction applies for the remaining boundary parameters.
\end{proof}
Let $\mathbf{1}$ and $\mathbf{0}$ be the configurations in $\Omega_N$ containing only particles and empty sites, respectively, and observe that these two configurations form the unique maximal and minimal elements with respect to the partial order $\succeq_\c$ on $\Omega_N$. The following lemma is an immediate consequence of Lemma \ref{lem:MonotoneCouplingComponentwise} and \cite[Corollary 5.5]{LPW:markov-mixing}.
\begin{lemma} \label{lem:HittingTime}For a simple exclusion process with open boundaries and $\varepsilon$-mixing time $t^N_{\textup{mix}}(\varepsilon)$, let $\tau$ denote the first time, at which the processes started from $\mathbf{1}$ and $\mathbf{0}$, respectively, agree within the coupling $\mathbf{P}$ given in Lemma \ref{lem:MonotoneCouplingComponentwise}. If for some $s\geq 0$ 
\begin{equation}\label{eq:couplingTimeTau}
\mathbf{P}(\tau \geq s) \leq \varepsilon
\end{equation} holds, then the $\varepsilon$-mixing time satisfies $t^N_{\textup{mix}}(\varepsilon) \leq s$.
\end{lemma}
\subsubsection{The partial order via height functions}\label{sec:HeightFunction}

When $\max(\alpha,\gamma)=0$ or $\max(\beta,\delta)=0$ holds, we define another partial order $\succeq_{\h}$ on $\Omega_N$ for the simple exclusion process with open boundaries. A similar partial order can be found in \cite{W:MixingLoz} for the simple exclusion process. For $\max(\alpha,\gamma)=0$, we let
\begin{equation}\label{def:partialorder}
\eta \succeq_{\h} \zeta  \ \ \Leftrightarrow \ \ \sum_{i=1}^j \eta(i)  \geq \sum_{i=1}^j \zeta(i) \ \text{ for all } j \in [N]
\end{equation}
for all configurations $\eta,\zeta \in \Omega_{N}$. For $\max(\beta,\delta)=0$, apply the definition \eqref{def:partialorder} to the simple exclusion process with open boundaries and parameters $(1-p,0,\gamma,0,\alpha)$, i.e. we flip the segment vertically. This partial order arises from the height function representation. 
For a given  $\eta\in \{ 0,1\}^N$, let $h_{\eta} \colon \{ 0,1,\dots, 2N \} \rightarrow \R$  be its \textbf{height function} with
\begin{equation}\label{def:HeightFunction} 
h_{\eta}(x) := \sum_{i=1}^{x}2\left[\eta(i)\mathds{1}_{\{i \leq  N\}} + (1-\eta(2N+1-i))\mathds{1}_{\{i> N\}}\right] - x 
\end{equation} for all $x\in \{ 0,1,\dots, 2N \}$. Note that we have $h_{\eta}(0) = h_{\eta}(2N) = 0$ by construction. 
For all $N \in \N$, we see that a pair of configurations satisfies $\eta \succeq_{\h} \zeta$ if and only if $h_{\eta}(x) \geq h_{\zeta}(x)$ holds for all $x \in [N]$.  A visualization of the height function in terms of \textbf{lattice paths}, which are the linear interpolations of height functions, is given in Figure \hyperref[fig:MonotoneCoupling]{2}. Again, the canonical coupling can be extended such that it is monotone in $\alpha,\beta,\gamma,\delta$ with respect to the partial order $\succeq_{\h}$. This is stated in the following lemma which uses the same coupling $\mathbf{P}$ constructed in the proof of Lemma \ref{lem:MonotoneCouplingComponentwise}.
\begin{figure}
\centering
\label{fig:MonotoneCoupling}
\begin{tikzpicture}[scale=0.65]

\def \x {11.8}

	\draw (0,0) to (4,4);
	\draw (4,4) to (4,-4);
	\draw (4,-4) to (0,0);

	\draw[Gray!50,thin](1,-1) to (4,2);
	\draw[Gray!50,thin](1,1) to (4,-2);
	\draw[Gray!50,thin](2,2) to (4,0);
	\draw[Gray!50,thin](2,-2) to (4,0);
	\draw[Gray!50,thin](3,3) to (4,2);
	\draw[Gray!50,thin](3,-3) to (4,-2);
	
\node[scale=1.2] (Lower) at (10.1,5){$\succeq_{\textup{h}}$};

		\draw (0+\x,0) to (4+\x,4);
	\draw (4+\x,4) to (4+\x,-4);
	\draw (4+\x,-4) to (0+\x,0);

	\draw[Gray!50,thin](1+\x,-1) to (4+\x,2);
	\draw[Gray!50,thin](1+\x,1) to (4+\x,-2);
	\draw[Gray!50,thin](2+\x,2) to (4+\x,0);
	\draw[Gray!50,thin](2+\x,-2) to (4+\x,0);
	\draw[Gray!50,thin](3+\x,3) to (4+\x,2);
	\draw[Gray!50,thin](3+\x,-3) to (4+\x,-2);

	\node[shape=circle,scale=1.3,draw] (A2) at (-0.2,5){} ;
	\node[shape=circle,scale=1.3,draw] (B2) at (1,5){} ;
	\node[shape=circle,scale=1.3,draw] (C2) at (2.2,5){} ;
	\node[shape=circle,scale=1.3,draw] (D2) at (3.4,5){} ;
	
	\node[shape=circle,scale=1.3,draw,black!50] (E2) at (4.6,5){} ;
	\node[shape=circle,scale=1.3,draw,black!50] (F2) at (5.8,5){} ;
	\node[shape=circle,scale=1.3,draw,black!50] (G2) at (7,5){} ;
	\node[shape=circle,scale=1.3,draw,black!50] (H2) at (8.2,5){} ;

	\node[shape=circle,scale=1.3,draw] (A3) at (-0.2+\x,5){} ;
	\node[shape=circle,scale=1.3,draw] (B3) at (1+\x,5){} ;
	\node[shape=circle,scale=1.3,draw] (C3) at (2.2+\x,5){} ;
	\node[shape=circle,scale=1.3,draw] (D3) at (3.4+\x,5){} ;

	\node[shape=circle,scale=1.3,draw,black!50] (E3) at (4.6+\x,5){} ;
	\node[shape=circle,scale=1.3,draw,black!50] (F3) at (5.8+\x,5){} ;
	\node[shape=circle,scale=1.3,draw,black!50] (G3) at (7+\x,5){} ;
	\node[shape=circle,scale=1.3,draw,black!50] (H3) at (8.2+\x,5){} ;

	\draw[thick] (A2) to (B2);		
	\draw[thick] (B2) to (C2);		
	\draw[thick] (C2) to (D2);		
	\draw[thick] (D2) to (4,5);		
	
	\draw[thick,dashed,black!50] (4,5) to (E2);		
	\draw[thick,dashed,black!50] (E2) to (F2);		
	\draw[thick,dashed,black!50] (F2) to (G2);		
	\draw[thick,dashed,black!50] (G2) to (H2);

	\draw[thick] (A3) to (B3);		
	\draw[thick] (B3) to (C3);		
	\draw[thick] (C3) to (D3);		
	\draw[thick] (D3) to (4+\x,5);		
	
	\draw[thick,dashed,black!50] (4+\x,5) to (E3);		
	\draw[thick,dashed,black!50] (E3) to (F3);		
	\draw[thick,dashed,black!50] (F3) to (G3);		
	\draw[thick,dashed,black!50] (G3) to (H3);

  \node[shape=circle,fill=red!30,scale=1] (k1) at (G2){} ;
  \node[shape=circle,fill=red!30,scale=1] (k1) at (E3){} ;
  \node[shape=circle,fill=red!30,scale=1] (k1) at (H3){} ;	
	
 \node[shape=circle,fill=red,scale=1] (k1) at (A2){} ;
 \node[shape=circle,fill=red,scale=1] (k2) at (D2){} ; 
  \node[shape=circle,fill=red,scale=1] (k3) at (C2){} ; 

 \node[shape=circle,fill=red,scale=1] (k1) at (B3){} ;
 \node[shape=circle,fill=red,scale=1] (k2) at (C3){} ; 

\draw[line width=3pt,red] (0,0) -- (1,1) -- (2,0) --(3,1)--(4,2);
\draw[line width=3pt,red] (\x,0) -- (1+\x,-1) -- (2+\x,0) --(3+\x,1)--(4+\x,0);

	\draw[dashed,black!50] (8,0) to (4,4);
	\draw[dashed,black!50] (4,-4) to (8,0);

	\draw[rotate around={180:(4,0)},Gray!50,thin,dashed](1,-1) to (4,2);
	\draw[rotate around={180:(4,0)},Gray!50,thin,dashed](1,1) to (4,-2);
	\draw[rotate around={180:(4,0)},Gray!50,thin,dashed](2,2) to (4,0);
	\draw[rotate around={180:(4,0)},Gray!50,thin,dashed](2,-2) to (4,0);
	\draw[rotate around={180:(4,0)},Gray!50,thin,dashed](3,3) to (4,2);
	\draw[rotate around={180:(4,0)},Gray!50,thin,dashed](3,-3) to (4,-2);

    \draw[rotate around={180:(4+\x,0)},Gray!50,thin,dashed](1+\x,-1) to (4+\x,2);
	\draw[rotate around={180:(4+\x,0)},Gray!50,thin,dashed](1+\x,1) to (4+\x,-2);
	\draw[rotate around={180:(4+\x,0)},Gray!50,thin,dashed](2+\x,2) to (4+\x,0);
	\draw[rotate around={180:(4+\x,0)},Gray!50,thin,dashed](2+\x,-2) to (4+\x,0);
	\draw[rotate around={180:(4+\x,0)},Gray!50,thin,dashed](3+\x,3) to (4+\x,2);
	\draw[rotate around={180:(4+\x,0)},Gray!50,thin,dashed](3+\x,-3) to (4+\x,-2);

		\draw[dashed,black!50] (8+\x,0) to (4+\x,4);
	\draw[dashed,black!50] (4+\x,-4) to (8+\x,0);

\draw[line width=3pt,red!30,dashed] (8,0) -- (7,1) -- (6,0) --(5,1)--(4,2);

\draw[line width=3pt,red!30,dashed] (8+\x,0) -- (7+\x,-1) -- (6+\x,0) --(5+\x,1)--(4+\x,0);
\end{tikzpicture}
\caption{Two ordered instances of the simple exclusion process for $N=4$ and with particles entering and exiting only at the right-hand side of the segment.}
\end{figure}
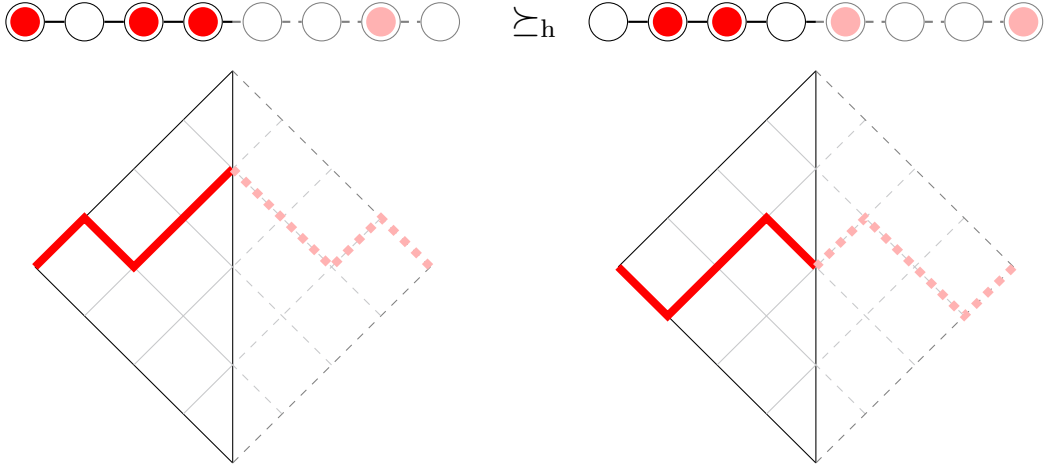
\begin{lemma}\label{lem:MonotoneCouplingHeightFunction} Consider two exclusion processes $(\eta_t)_{t\geq 0}$ and $(\zeta_t)_{t\geq 0}$ on the segment of size $N$ with parameters $(p,\alpha,\beta,\gamma,\delta)$ and $(p^{\prime},\alpha^{\prime},\beta^{\prime},\gamma^{\prime},\delta^{\prime})$, respectively. Suppose that
\begin{equation}
p\leq p^{\prime} \quad \alpha =\alpha^{\prime}=0 \quad \beta \leq \beta^{\prime} \quad \gamma=\gamma^{\prime}=0 \quad \delta \geq \delta^{\prime} 
\end{equation}
or
\begin{equation}
p\leq p^{\prime} \quad \alpha \geq \alpha^{\prime} \quad \beta=\beta^{\prime}=0  \quad \gamma \leq \gamma^{\prime} \quad \delta=\delta^{\prime}=0
\end{equation}
holds. Then there exists a coupling $\mathbf{P}$ of the two processes which satisfies
\begin{equation}
\mathbf{P}\left( \eta_t \succeq_\h \zeta_t  \text{ for all } t \geq 0 \mid \eta_0 \succeq_\h \zeta_0 \right) = 1 \ .
\end{equation}
\end{lemma}

\subsection{The simple exclusion process with second class particles} \label{sec:SecondClass}

Second class particles for the simple exclusion process are well-studied over the last decades, see  \cite[Section III.1]{L:Book2} for an introduction. The motion of a second class particle can be related to current and shock fluctuations, see \cite{BS:OrderCurrent,F:ShockFluctuations, FF:CurrentFluctuations,FF:ShockFluctuations}. 
In the context of mixing times, second class particles were used to study the simple exclusion process when all boundary parameters are zero \cite{BBHM:MixingBias,S:MixingBallistic}. In this paper, we use second class particle arguments in Sections \ref{sec:UpperBounds} to \ref{sec:HighLowDensity} in order to provide upper bounds for the mixing time of the simple exclusion process with open boundaries. \\

For a configuration $\xi \in \lbrace 0,1,2\rbrace^{N}$, we say that a vertex $x\in [N]$ is occupied by a \textbf{first class particle} whenever $\xi(x)=1$ and by a \textbf{second class particle} if $\xi(x)=2$ holds. Our main application for second class particles is to describe the difference of two exclusion processes. More precisely, for two simple exclusion processes  $(\eta_t)_{t \geq 0}$ and $(\zeta_t)_{t \geq 0}$ with open boundaries on a segment of size $N$, we define the \textbf{disagreement process} $(\xi_t)_{t\geq 0}$ between $(\eta_t)_{t \geq 0}$ and $(\zeta_t)_{t \geq 0}$ by
\begin{equation}\label{def:DisagreementProcess}
\xi_t(x)= \mathds{1}_{\{\eta_t(x)=\zeta_t(x)=1\}} + 2 \mathds{1}_{\{\eta_t(x) \neq \zeta_t(x)\}}
\end{equation}
for all $x\in [N]$ and $t \geq 0$. In words, we keep the current value if the processes $(\eta_t)_{t \geq 0}$ and $(\zeta_t)_{t \geq 0}$ agree and place a second class particle otherwise, see Figure \hyperref[fig:Disagreement]{3}.
\begin{figure} \label{fig:Disagreement}
\centering
\begin{tikzpicture}[scale=0.9]

\def\x{2};
\def\y{-1.3};

 	\node[shape=circle,scale=1.5,draw] (C1) at (0,0*\y){} ; 
 	\node[shape=circle,scale=1.5,draw] (C2) at (\x,0*\y){} ; 
 	\node[shape=circle,scale=1.5,draw] (C3) at (2*\x,0*\y){} ; 
 	\node[shape=circle,scale=1.5,draw] (C4) at (3*\x,0*\y){} ; 
 	\node[shape=circle,scale=1.5,draw] (C5) at (4*\x,0*\y){} ; 
 	\node[shape=circle,scale=1.5,draw] (C6) at (5*\x,0*\y){} ; 
 	\node[shape=circle,scale=1.5,draw] (C7) at (6*\x,-0*\y){} ; 
 	
 	\draw[thick] (C1) -- (C2);
 	\draw[thick] (C2) -- (C3);
 	\draw[thick] (C3) -- (C4);
 	\draw[thick] (C4) -- (C5);
 	\draw[thick] (C5) -- (C6);
 	\draw[thick] (C6) -- (C7);

	\node[shape=circle,scale=1.2,fill=red] (F5) at (C1) {};
	\node[shape=circle,scale=1.2,fill=red] (F5) at (C2) {}; 	
	\node[shape=circle,scale=1.2,fill=red] (F5) at (C4) {};
	\node[shape=circle,scale=1.2,fill=red] (F5) at (C5) {}; 	
	\node[shape=circle,scale=1.2,fill=red] (F5) at (C7) {};	
 
	\node (H3) at (-1.4,0*\y) {$\eta_t$};

 	\node[shape=circle,scale=1.5,draw] (C1) at (0,1*\y){} ; 
 	\node[shape=circle,scale=1.5,draw] (C2) at (\x,1*\y){} ; 
 	\node[shape=circle,scale=1.5,draw] (C3) at (2*\x,1*\y){} ; 
 	\node[shape=circle,scale=1.5,draw] (C4) at (3*\x,1*\y){} ; 
 	\node[shape=circle,scale=1.5,draw] (C5) at (4*\x,1*\y){} ; 
 	\node[shape=circle,scale=1.5,draw] (C6) at (5*\x,1*\y){} ; 
 	\node[shape=circle,scale=1.5,draw] (C7) at (6*\x,1*\y){} ;  
 
 	\draw[thick] (C1) -- (C2);		
 	\draw[thick] (C2) -- (C3);
 	\draw[thick] (C3) -- (C4);
 	\draw[thick] (C4) -- (C5);
 	\draw[thick] (C5) -- (C6);
	\draw[thick] (C6) -- (C7);

	\node[shape=circle,scale=1.2,fill=red] (F5) at (C2) {}; 	
	\node[shape=circle,scale=1.2,fill=red] (F5) at (C7) {};

	\node (H3) at (-1.4,1*\y) {$\zeta_t$};

 	\node[shape=circle,scale=1.5,draw] (C1) at (0,2*\y){} ; 
 	\node[shape=circle,scale=1.5,draw] (C2) at (\x,2*\y){} ; 
 	\node[shape=circle,scale=1.5,draw] (C3) at (2*\x,2*\y){} ; 
 	\node[shape=circle,scale=1.5,draw] (C4) at (3*\x,2*\y){} ; 
 	\node[shape=circle,scale=1.5,draw] (C5) at (4*\x,2*\y){} ; 
 	\node[shape=circle,scale=1.5,draw] (C6) at (5*\x,2*\y){} ; 
 	\node[shape=circle,scale=1.5,draw] (C7) at (6*\x,2*\y){} ;

 	\draw[thick] (C1) -- (C2); 	
 	\draw[thick] (C2) -- (C3);
 	\draw[thick] (C3) -- (C4);
 	\draw[thick] (C4) -- (C5);
 	\draw[thick] (C5) -- (C6);
 	\draw[thick] (C6) -- (C7);

	\node[shape=star,star points=5,star point ratio=2.5,fill=red,scale=0.55] (G1) at (C1) {};	
	\node[shape=star,star points=5,star point ratio=2.5,fill=red,scale=0.55] (G2) at (C4) {};	
	\node[shape=star,star points=5,star point ratio=2.5,fill=red,scale=0.55] (G4) at (C5) {};

	\node[shape=circle,scale=1.2,fill=red] (F3) at (C2) {};
	\node[shape=circle,scale=1.2,fill=red] (F5) at (C7) {}; 
 
	\node (H3) at (-1.4,2*\y) {$\xi_t$};

	\end{tikzpicture}
\caption{Two configurations $\eta_t \succeq_{\h} \zeta_t$ with disagreement process $\xi_t$ at time $t \geq 0$.}
 \end{figure}
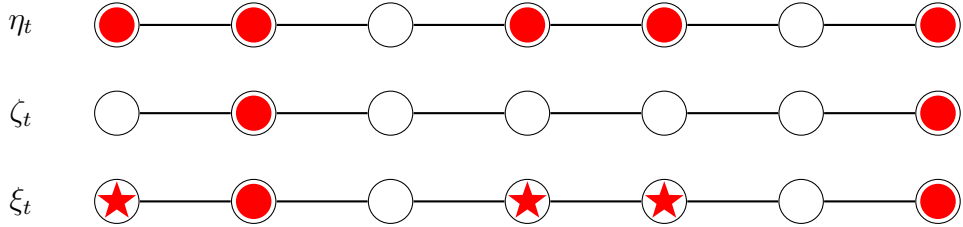
When $(\eta_t)_{t \geq 0}$ and $(\zeta_t)_{t \geq 0}$ with parameters $(p,\alpha,\beta,\gamma,\delta)$ and $(p^{\prime},\alpha^{\prime},\beta^{\prime},\gamma^{\prime},\delta^{\prime})$ satisfy the assumptions \eqref{eq:LemmaMonotoneAssumptions} of Lemma \ref{lem:MonotoneCouplingComponentwise}, and $\eta_0 \succeq_{\c} \zeta_0$, then the disagreement process with respect to the coupling $\mathbf{P}$ is a Feller process $(\xi_t)_{t\geq 0}$ on $\lbrace 0,1,2\rbrace^{N}$ according to the following description:  \\

We assign priorities to the particles and empty sites. First class particles have the highest priority, then second class particles, and then empty sites.
Suppose that a site $x$ and its neighbor $x+1$ are updated in configuration $\xi$. If $\xi(x)=\xi(x+1)$ holds, we leave the configuration unchanged. Else, we exchange the values at $x$ and $x+1$ in $\xi$ with probability $p$ if $\xi(x)$ has a higher priority than $\xi(x+1)$ and with probability $1-p$, otherwise. At the site $1$, we place a first class particle at rate $\alpha^{\prime}$ independently of the value of $\xi(1)$. In addition, if $\alpha> \alpha^{\prime}$ holds, assign a rate $(\alpha-\alpha^{\prime})$ Poisson clock to vertex $1$. When the clock rings and $\xi(1)=0$ holds, we place a second class particle at site $1$. A similar construction holds for the remaining boundary parameters. \\

In general, we define the \textbf{simple exclusion process with second class particles} (also called two-species exclusion process) to be the Feller process $(\xi_t)_{t\geq 0}$ on $\lbrace 0,1,2\rbrace^{N}$ which has the above update rules along the edges, i.e.\ the positions are exchanged according to the priorities assigned to the sites of the edge. However, we allow general transition rules for the particles to enter and exit at the boundary.

\begin{remark}\label{rem:MultiSpecies} A similar construction extends the canonical coupling to more than two different hierarchies of particles. In this case, the resulting process is usually called \textbf{multi-species exclusion process}, see \cite{CGGW:KoornwinderMulti, FRV:MultispeciesMatrixProduct} as well as Section \ref{sec:HighLowDensity}.
\end{remark}

We notice that two simple exclusion processes in the canonical coupling agree when their disagreement process contains no second class particles. Therefore, we have the following immediate consequence of Lemma \ref{lem:HittingTime}.
\begin{corollary}\label{cor:SecondClassParticles}  For a given set of parameters, let $(\eta^{\mathbf{0}}_t)_{t \geq 0}$ and $(\eta^{\mathbf{1}}_t)_{t \geq 0}$ denote the simple exclusion processes with open boundaries in the canonical coupling $\mathbf{P}$ with respect to the initial configurations $\mathbf{0}$ and $\mathbf{1}$. Let $(\xi_t)_{t \geq 0}$ be their disagreement process and denote by $\tau$ the first time at which $(\xi_t)_{t \geq 0}$ contains no second class particle. If  $\mathbf{P}(\tau>s) \leq \varepsilon$ holds for some $\varepsilon>0$ and $s\geq 0$, then we have that  $t^N_{\text{\normalfont mix}}(\varepsilon)\leq s$.
\end{corollary}

\subsection{The simple exclusion process on $\mathbb{Z}$ and blocking measures}

When we prove bounds on the mixing time, it will be convenient to compare the simple exclusion process with open boundaries to an exclusion process on the integers. The \textbf{simple exclusion process on} $\mathbb{Z}$ is given as a Feller process with state space $\{ 0,1\}^{\mathbb{Z}}$, generated by the closure of
\begin{align*}
\mathcal{L}^{\mathbb{Z}}_{\textup{ex}}f(\eta) &= \sum_{x \in \mathbb{Z}} p \ \eta(x)(1-\eta(x+1))\left[ f(\eta^{x,x+1})-f(\eta) \right] \nonumber \\
&+ \sum_{x \in \mathbb{Z}} \left(1-p\right) \ \eta(x)(1-\eta(x-1))\left[ f(\eta^{x,x-1})-f(\eta) \right]  
\end{align*} for some $p\in [0,1]$ and all cylinder functions $f$. Let now $p \in ( \frac{1}{2},1 )$. By Theorem 1.2 in \cite[Section III]{L:Book2}, the Bernoulli-product measure $\nu$ with marginals
\begin{equation}\label{eq:BernoulliProductc}
\nu(\eta \colon \eta(x)=1) = \frac{c p^x}{(1-p)^x+ c p^x} \ \text{ for all } x \in \mathbb{Z}
\end{equation}
is invariant for the simple exclusion process on $\mathbb{Z}$ for any constant $c>0$. The first Borel-Cantelli Lemma yields that
\begin{equation*}
\nu\left( \left\lbrace \eta \colon \exists\  C_{\eta}>0 \text{ s.t. } \eta(x)=1 \  \forall x > C_{\eta} \text{ and } \eta(x)=0 \ \forall x < -C_{\eta}\right\rbrace \right) =1,
\end{equation*}
i.e.\ $\nu$ is supported on the countable set of configurations $\eta$ with satisfy $\eta(x)=1$ and $\eta(-x)=0$ for all  $x>0$ sufficiently large. For $n \in \Z$, we can restrict the state space to 
\begin{equation}\label{blockingsets}
A_n := \left\lbrace \eta \in \{ 0,1\}^{\Z} \colon \sum\limits_{ x > n }  \left( 1- \eta(x) \right) = \sum\limits_{x \leq n} \eta(x) < \infty \right\rbrace
\end{equation} and define the simple exclusion process on $A_n$ as a Feller process with a countable state space. We define the \textbf{blocking measure} $\nu_{(n)}$ on $A_n$ to be given by $\nu_{(n)}(.)=\nu( \ . \mid A_n)$ for all $n \in \mathbb{Z}$.
Let us stress that this definition does not depend on the choice of $c$ in \eqref{eq:BernoulliProductc}. Further, let the \textbf{ground state} $\vartheta_n$  of $A_n$ be 
\begin{equation}\label{groundstateconfig}
\vartheta_n(x) := \begin{cases} 1 & \text{ if } x>n \\ 0 & \text{ if } x \leq n
\end{cases} \ \  \text{ for all } x \in \mathbb{Z}.
\end{equation}
Intuitively, the ground state is the state of minimal energy. Observe that $\nu(\vartheta_n) >0$ holds for all $p \in \left( \frac{1}{2},1 \right]$ and $n\in \mathbb{Z}$. Since $\vartheta_{n} \in A_n$, we have that $\nu_{(n)}(\vartheta_n) > 0$ holds. Hence,  the simple exclusion process on $A_n$ is positive recurrent for all $p \in \left( \frac{1}{2},1 \right]$ and $n \in \mathbb{Z}$. 

\begin{remark}\label{rem:PartialOrderonZ}
Note that the canonical coupling and the partial order $\succeq_{\h}$ in \eqref{def:partialorder} naturally extend to $\Z$, i.e.\ for $\eta \in A_n$ and $\zeta \in A_m$ with $n,m \in \Z$, we have that
\begin{equation}\label{def:partialorderBlocking}
\eta \succeq_{\h} \zeta  \ \ \Leftrightarrow \ \ \sum_{i=-\infty}^j \eta(i)  \geq \sum_{i=-\infty}^j \zeta(i) \ \text{ for all } j \in \Z \ .
\end{equation} Moreover, observe that the canonical coupling is monotone with respect to $\succeq_{\h}$ and that the ground state $\vartheta_n$ is the unique minimal element with respect to the partial order $\succeq_{\h}$ on $A_n$ for all $n\in \Z$. 
\end{remark}

For $\eta \in A_0$, let $L(\eta)$ and $R(\eta)$ denote the position of the leftmost particle and the rightmost empty site in $\eta$, respectively. In Sections \ref{sec:ProofAsymmetricHomogenousOneSide} and \ref{sec:ReverseBias}, we use the following lemma which gives an upper bound on the positions of the leftmost particle and the rightmost empty site when starting from the blocking measure. Its proof is deferred to the appendix.

\begin{lemma}\label{lem:HittingBlockingMeasure}  For $p \in (\frac{1}{2},1)$, let $(\eta^{\Z}_t)_{t \geq 0}$ denote the simple exclusion process in $A_0$ with initial distribution $\nu_{(0)}$. There exists a constant $C=C(p)>0$ such that for any $\varepsilon \in (0,\frac{1}{2})$ and all $x\geq 0$ sufficiently large,
\begin{equation}\label{eq:HittingBlocking}
\P_{\nu_{(0)}}\left( \max\left(R(\eta^{\Z}_t),-L(\eta^{\Z}_t)\right) \leq x \text{ for all } t \in \left[ 0, \frac{\varepsilon C}{x} \left(\frac{p}{1-p}\right)^{x}\right] \right) \geq 1-2\varepsilon\, .
\end{equation} 
\end{lemma} 
 
\subsection{Current for the simple exclusion process} \label{sec:Current}

This section is dedicated to the study of the current for the simple exclusion process with open boundaries. Currents are one of the main objects for the exclusion process in statistical mechanics with deep connections to second class particles, see \cite{BS:OrderCurrent,FF:CurrentFluctuations, USW:PASEPcurrent}. Intuitively, the current formalizes the way of counting the number of particles which pass through the segment over time. For our purposes, current arguments will be used in order to prove the upper bounds in Theorems \ref{thm:asymmetricOneSide} and \ref{thm:asymmetricTwoSides}. \\

For $p \in \left( \frac{1}{2},1\right]$, assume that $\min(\alpha,\beta)>0$ holds. On the segment of size $N$, let $J^{N+}_t$ be the number of particles which have entered at the left-hand side of the segment by time $t$ and let $J^{N-}_t$ be the number of particles which have exited at the left-hand side of the segment by time $t$.  Let $(J^N_t)_{t \geq 0}$ with
\begin{equation}\label{def:current}
J^N_t := J^{N+}_t - J^{N-}_t \ \text{ for all } t \geq 0
\end{equation}   be the \textbf{current} of the simple exclusion process with open boundaries. Similarly, one could define the current with respect to the net number of particles crossing the right-hand side of the segment, leading to the same long-term behavior. The following lemma states an asymptotic bound on the current. We obtain it from the results  in \cite[Section 6]{USW:PASEPcurrent} and the observation that under the above assumptions, the simple exclusion process with open boundaries is a positive recurrent Feller process; see also Theorem 3.6 in \cite{CW:TableauxCombinatorics}.

\begin{lemma}\label{lem:current} Recall the definition of $a$ and $b$ from \eqref{def:a} and \eqref{def:b} and set
\begin{equation}
J =J(a,b,p) := \begin{cases} (2p-1)\frac{a}{(1+a)^2} & \text{ if }  a>\max(b,1)   \\
(2p-1)\frac{b}{(1+b)^2} & \text{ if } b >\max(a,1)  \\
(2p-1)\frac{1}{4}  & \text{ if } \max(a,b)\leq  1 \ .
\end{cases}
\end{equation} Then the current $(J^N_t)_{t \geq 0}$ of the simple exclusion process with open boundaries satisfies
\begin{equation}
\lim_{t \rightarrow \infty} \frac{J^N_t}{t} = J_N
\end{equation} almost surely for some deterministic sequence $(J_N)_{N \in \N}$ with $\lim_{N \rightarrow \infty}J_N = J$. 
\end{lemma} We refer to $J_N$ as the \textbf{flux} of the simple exclusion process with open boundaries on the segment of size $N$.

\subsection{Invariant measures of the simple exclusion process} \label{sec:Invariant}

In this section, we focus on the stationary distribution $\mu$ of the simple exclusion process with open boundaries. A beautiful combinatorial description of $\mu$ is given in \cite{CW:TableauxCombinatorics} using staircase tableaux. The following result, which is adopted from \cite{BCEPR:CombinatoricsPASEP}, shows that under certain conditions on the boundary parameters, the invariant distribution has a product structure. In general, $\mu$ can not be stated in a simple closed form.
\begin{lemma}[c.f.\ \cite{BCEPR:CombinatoricsPASEP}, Proposition 2] \label{lem:stationaryDistribution} Suppose that $\min(\alpha,\beta)>0$ and $a=\frac{1}{b}$ holds for $a$ and $b$ given in \eqref{def:a} and \eqref{def:b}. Then for every configuration $\eta \in \Omega_N$, we have that
\begin{equation}
\mu(\eta)=\frac{1}{(\alpha+\beta+\gamma+\delta)^N}(\alpha+\delta)^{\abs{\eta}}(\beta+\gamma)^{N-\abs{\eta}}= \left( \frac{1}{1+a}\right)^{\abs{\eta}}\left( \frac{a}{1+a}\right)^{N-\abs{\eta}}
\end{equation} where $\abs{\eta}:=\sum_{i=1}^N \eta(i)$ denotes the number of particles in $\eta$. 
\end{lemma}
Next, we compare the stationary measure $\mu$ to the Bernoulli-$\rho$-product measures $\nu_{\rho}$ for some $\rho \in [0,1]$ on $\Omega_N$. More generally, let $\nu,\nu^{\prime}$ be two probability measures defined on a common probability space $\Omega$ which is equipped with a partial order $\succeq$. We say that $\nu$ \textbf{stochastically dominates} $\nu^{\prime}$ with respect to $\succeq$ (and write $\nu \succeq \nu^{\prime}$) if there exists a coupling $P$  with $X \sim \nu$ and $Y\sim \nu^{\prime}$ such that $P(X \succeq Y)=1$. An equivalent definition using increasing functions can be found in \cite[Theorem B.9]{L:Book2}.
\begin{lemma}\label{lem:ComparsionStationaryMeasure} Suppose that $\min(\alpha,\beta)>0$ holds. Then the stationary distribution $\mu$ of the simple exclusion process with open boundaries satisfies
\begin{equation}
\nu_{\c_{\max}} \succeq_\c \mu \succeq_\c \nu_{\c_{\min}}
\end{equation} where
\begin{equation}
c_{\min} := \min\left(\frac{1}{1+a},\frac{b}{1+b}\right) \quad \mbox{and} \quad c_{\max} := \max\left(\frac{1}{1+a},\frac{b}{1+b}\right) \ .
\end{equation}
\end{lemma}
\begin{proof} We consider only $\mu \succeq_\c \nu_{\c_{\min}}$ for $c_{\min}=\frac{b}{1+b}$ as the remaining cases are similar. In this case, we have $a\leq \frac{1}{b}$, and we recall $a=a(\alpha,\gamma,p)$ from \eqref{def:a}. Observe that $a$ is decreasing in $\alpha$ and note that we can choose  some $\alpha^{\prime} \in  (0,\alpha]$ such that $a^{\prime}:=a(\alpha^{\prime},\gamma,p)$ satisfies $a^{\prime}=\frac{1}{b}$. We conclude by  Lemma \ref{lem:MonotoneCouplingComponentwise} and Lemma \ref{lem:stationaryDistribution}.
\end{proof}
Note that Lemma \ref{lem:ComparsionStationaryMeasure} is motivated by treating the simple exclusion process with open boundaries as having reservoirs at both ends with densities $\frac{1}{1+a}$ and $\frac{b}{1+b}$, respectively, and $\mu$ interpolating between both sides. The next result characterizes how the interpolation within the stationary distribution $\mu$ is realized. Using Lemma \ref{lem:current}  and Lemma \ref{lem:ComparsionStationaryMeasure}, it follows from the same arguments as Theorem 3.29 in \cite{L:Book2}.
\begin{lemma} \label{lem:stationaryDensity} Suppose that $\min(\alpha,\beta)>0$ holds. Let $(x_N)_{N \in \N}$ be a sequence with $\min(x_N,\frac{N}{2}-x_N) \rightarrow \infty$ for $N \rightarrow \infty$. Further, let $\mu_{N}$ denote the measure on $\{0,1\}^{\N}$ given on the sites $1,\dots,N-2x_N$ by the restriction of $\mu$ to $[x_N,N-x_N]$, and by the Dirac measure on empty sites everywhere else. Then we have that 
\begin{equation}
\lim_{N \rightarrow \infty} \mu_{N} = \begin{cases} 
\nu_{\frac{1}{1+a}} & \text{if } a>\max(b,1) \\
\nu_{\frac{b}{1+b}} & \text{if } b>\max(a,1) \\ 
\nu_{\frac{1}{2}} & \text{if } \max(a,b)\leq 1,
\end{cases}
\end{equation} where the limit is with respect to weak convergence, and the product measures are defined on $\{0,1\}^\N$.
\end{lemma}
When particles are allowed to enter and exit only from one side of the segment, the measure $\mu$ is reversible and can be given explicitly.  More precisely, we say that $\mu$ is \textbf{reversible} for the simple exclusion process with open boundaries if
\begin{equation}\label{def:Reversibility}
\sum_{\eta \in \Omega_N} f(\eta)(\mathcal{L}g)(\eta)  \mu(\eta)= \sum_{\eta \in \Omega_N} (\mathcal{L}f)(\eta)g(\eta) \mu(\eta)
\end{equation} holds for all functions $f,g \colon \Omega_N \rightarrow \R$. Suppose that particles are only allowed to enter and exit at the right-hand side, i.e.\ $\max(\alpha,\gamma)=0$ holds (a similar formula will hold  when $\max(\beta,\delta)=0$). For $p \in (0,1]$ and $\min(\beta,\delta)>0$, consider $\mu$ with
\begin{equation}\label{eq:reversibleDistribution}
\mu(\eta) = \frac{1}{Z_N} \left(\frac{\delta}{\beta}\right)^{|\eta|} \cdot \prod_{i=1}^{\abs{\eta}}\left(\frac{1-p}{p}\right)^{z_i} \ \text{ for all } \eta \in \Omega_N,
\end{equation}
 where $z_i$ denotes the distance of the $i^{\textup{th}}$ particle from site $N$ and $Z_N$ is a normalization constant. Then $\mu$ is reversible for the process $(\eta_t)_{t \geq 0}$.  When $\min(\beta,\delta)=0$ holds, $\mu$ is the Dirac measure on $\mathbf{1}$ if $\beta=0$ and on $\mathbf{0}$ if $\delta=0$. 

\subsection{The censoring inequality}\label{sec:Censoring}

The censoring inequality is a very recent technique in order to give upper bounds on the mixing time.
First established by Peres and Winkler in \cite{PW:Censoring} for spin systems, it was applied to the simple exclusion process by Lacoin in \cite{L:CutoffSEP}. In words, this inequality says that leaving out transitions of the exclusion process along certain edges only increases the distance from equilibrium. Using a slightly more general definition than in \cite{PW:Censoring}, we say that a \textbf{censoring scheme} $\mathcal{C}$ for $(\eta_t)_{t \geq 0}$ is a random càdlàg function 
\begin{equation}
\mathcal{C} \colon \mathbb{R}_0^+ \rightarrow \mathcal{P}\left( E \right)
\end{equation} which does not depend on the process $(\eta_t)_{t \geq 0}$. Here, $ \mathcal{P}\left( E \right)$ denotes the power set of the edges where we treat the boundary interactions as edges to reservoirs at positions $0$ and $N+1$, respectively.
In the censored dynamics $(\eta^{\mathcal{C}}_t)_{t\geq 0}$, a transition along an edge $e$ at time $t$ is performed if and only if $e \notin \mathcal{C}(t)$. The following censoring inequality with respect to the partial order $\succeq_{\h}$ for the simple exclusion process with open boundaries is an immediate consequence of Theorem 1.1 and Lemma 2.1 in \cite{PW:Censoring}.
\begin{lemma}[c.f. \cite{L:CutoffSEP}, Proposition 6.2]\label{lem:Censoring}  Let $\mathcal{C}$ be a censoring scheme for the simple exclusion process with open boundaries. For an initial configuration $\eta$ and $t \geq 0$, let $P_{\eta}(\eta_t \in \cdot)$ and $P_{\eta}(\eta^{\mathcal{C}}_t \in \cdot)$ denote the law of $(\eta_t)_{t\geq 0}$ and its censored dynamics $(\eta^{\mathcal{C}}_t)_{t\geq 0}$ at time $t\geq 0$, respectively. Under the assumptions of Lemma \ref{lem:MonotoneCouplingHeightFunction}, we have that
\begin{equation}\label{eq:CensoringDomination}
P_{\mathbf{1}}(\eta^{\mathcal{C}}_t \in \cdot) \succeq_\h P_{\mathbf{1}}(\eta_t \in \cdot) \quad \text{and} \quad P_{\mathbf{0}}(\eta^{\mathcal{C}}_t \in \cdot) \preceq_\h P_{\mathbf{0}}(\eta_t \in \cdot) \ \text{ for all } t\geq 0.
\end{equation} Moreover, the density function $\eta \mapsto \frac{1}{\mu(\eta)}P_{\mathbf{1}}(\eta_t =\eta)$ is increasing with respect to the partial order $\succeq_\h$ and we have that for all $t\geq 0$
\begin{equation}
\TV{P_{\mathbf{1}}(\eta^{\mathcal{C}}_t \in \cdot) - \mu} \geq \TV{P_{\mathbf{1}}(\eta_t \in \cdot) - \mu}
\end{equation} and 
\begin{equation}
\TV{P_{\mathbf{0}}(\eta^{\mathcal{C}}_t \in \cdot) - \mu} \geq \TV{P_{\mathbf{0}}(\eta_t \in \cdot) - \mu} \ .
\end{equation} 
\end{lemma}

\begin{remark}\label{rem:CensoringBlocking} Using the partial order $\succeq_\h$ from \eqref{def:partialorderBlocking} for the simple exclusion process on $\Z$, the same arguments show that for all $n \in \Z$, 
\begin{equation}
P_{\vartheta_n}(\eta^{\mathcal{C}}_t \in \cdot) \preceq_\h P_{\vartheta_n}(\eta_t \in \cdot) \ \text{ for all } t\geq 0,
\end{equation}
where $\vartheta_n$ is the ground state of $A_n$, see \eqref{groundstateconfig}.
\end{remark}

\section{Lower bounds for the symmetric exclusion process} \label{sec:LowerBounds}

In this section, we prove the lower bounds in Theorems \ref{thm:symmetricTwoSided} and \ref{thm:symmetricOneSide}. A key tool will be a generalized version of Wilson's lemma, which was introduced in \cite{NN:CutoffCyclic} for discrete-time Markov chains. It transfers to our setup as follows. For a Feller process $(X_t)_{t \geq 0}$ with generator $\mathcal{A}$, we consider a function $F$ which behaves almost like an eigenfunction of $-\mathcal{A}$. Further, let $(M_t)_{t \geq 0}$ be the associated martingale given by
\begin{equation}\label{def:DynkinMartingale}
M_t := F(X_t) - F(X_0) - \int_{0}^{t} (\mathcal{A}F)(X_s) \diff s \ \text{ for all } t \geq 0.
\end{equation}
We denote its quadratic variation by $(\langle M \rangle_t)_{t \geq 0}$.  For an introduction to martingales and their quadratic variation, we refer to \cite[Chapter 3 and 5]{L:Book3}. The next lemma is similar to Lemma 2 in \cite{NN:CutoffCyclic}. Its proof is deferred to the appendix.

\begin{lemma}[Generalized Wilson's lemma]\label{lem:GeneralizedWilsonContinuousTime} Let $(X_t)_{t \geq 0}$ be an irreducible Feller process with finite state space $S$ and generator $\mathcal{A}$. Let $F \colon S \rightarrow \R$ be a function with
\begin{equation}\label{eq:Eigenfunction} 
\abs{(-\mathcal{A}F)(y) - \lambda F(y) } \leq c \ \text{ for all } y \in S,
\end{equation} with constants $\lambda > 0$ and $c \geq 0$ with $\lambda \geq c$. Moreover, we assume that the quadratic variation $(\langle M \rangle_t)_{t \geq 0}$ of the associated martingale defined in \eqref{def:DynkinMartingale} satisfies
\begin{equation} \label{eq:QVBound}
\frac{\diff}{\diff t}\E\left[ \langle M \rangle_t\right] \leq R
\end{equation}
for some $R>0$ and all $t \geq 0$. Then for all $\varepsilon \in (0,1)$, the $\varepsilon$-mixing time $t_{\text{\normalfont mix}}(\varepsilon)$ of $(X_t)_{t \geq 0}$ satisfies
\begin{equation} \label{eq:LowerBoundWilson}
t_{\text{\normalfont mix}}(1-\varepsilon) \geq \frac{1}{\lambda}\log \left( \pdist{F}{\infty} \right) -   \frac{1}{2\lambda} \log\left(  \frac{ 16 (3c  \pdist{F}{\infty}+ \max(R,c))}{\lambda\varepsilon}\right).
\end{equation} 
\end{lemma}

In order to apply Lemma \ref{lem:GeneralizedWilsonContinuousTime} for the simple exclusion process with open boundaries for $p=\frac{1}{2}$, we construct a function $F$ which satisfies \eqref{eq:Eigenfunction} and \eqref{eq:QVBound}. We call $F$ an \textbf{approximate eigenfunction}. \\

Observe that for all choices of boundary parameters and initial configurations $\eta$, we have that $(f_{\eta}(x,t))_{x \in [N], t \geq 0}$ given by 
\begin{equation*}
f_{\eta}(x,t) := \E_{\eta}[\eta_t(x)] \qquad \text{ for all } x \in [N] \text{ and } t \geq 0
\end{equation*} 
solves a discrete heat equation, where we see either discrete Neumann boundary conditions for closed endpoints, or (a variant of) discrete Dirichlet boundary conditions for open endpoints. In the following, we consider a simple exclusion process with discrete Dirichlet boundary conditions at both endpoints, and compare it to a simple exclusion process on the circle of length $2N$ with $N$ particles. On the circle, the eigenfunctions are sine and cosine waves, where the length of the circle is a multiple of the period length, see \cite[Lemma 2.2]{L:CycleDiffusiveWindow} and \cite[Section 3.4]{W:MixingLoz}. We use this intuition to construct approximate eigenfunctions as stretched and shifted eigenfunctions of the classical discrete heat equation. 
 With a slight abuse of notation, extend each $\eta \in \Omega_N$ to $\Omega_{2N,N}$ given in \eqref{def:SpaceParticles} by
\begin{equation*}
\eta(x) :=1-\eta(2N+1-x) \ \text{ for all } x \in \{N+1,\dots,2N\}.
\end{equation*} 
\begin{lemma}\label{lem:EigenfunctionTwoSides} Recall that $p=\frac{1}{2}$ and assume that $\max(\alpha,\gamma)>0$ and $\max(\beta,\delta)>0$ holds. We set
\begin{equation}\label{eq:Paramtercd}
C:=\frac{1}{2(\alpha+\gamma)}-\frac{1}{2} \quad \text{ and } \quad D:=\frac{1}{2(\beta+\delta)}-\frac{1}{2}
\end{equation}
 and define $M:=N+C+D$. Let $\phi \colon \Z/(2N)\Z \rightarrow \R$ be given by
\begin{equation}
\phi(x) :=\sin\left(\left( x+C-\frac{1}{2}\right)\frac{\pi}{M}\right) \ \text{ for all } x \in [N]
\end{equation} and set $\phi(x)=-\phi(2N+1-x)$ for all $x \in \{ N+1,\dots,2N\}$. Moreover, we let $\lambda_N:=1-\cos(\frac{\pi}{M})$ and define
\begin{equation}
\Phi_N(\eta) := \sum_{x=1}^{2N}\eta(x)\phi(x) +\frac{\phi(1)}{\lambda_N}(\gamma-\alpha)+\frac{\phi(N)}{\lambda_N}(\beta-\delta)
\end{equation} for all $\eta \in \Omega_{N}$. Then $\Phi_N$ satisfies the conditions of Lemma \ref{lem:GeneralizedWilsonContinuousTime} for $\lambda=\lambda_N$, for some $c$ of order $N^{-3}$, some $R$ of order $N^{-1}$ and $\pdist{\Phi_N}{\infty}$ of order $N$. In particular, under the above assumptions the lower bound stated in Theorem \ref{thm:symmetricTwoSided} holds.
\end{lemma}
\begin{proof} Using trigonometric identities, we have that
$(\Delta\phi)(x) = -\lambda_N \phi(x)$ holds for all $x \in \{2,\dots,N-1\}$. Here, $\Delta$ is the discrete Laplace operator on the circle of length $2N$, i.e.\ for all functions $f \colon \mathbb{Z}/(2N)\mathbb{Z} \rightarrow \R$, we set 
\begin{equation}\label{def:Laplacian}
(\Delta f)(x) := \frac{1}{2}(f(x-1)+f(x+1))-f(x) \ \text{ for all } x\in \mathbb{Z}/(2N)\mathbb{Z}.
\end{equation} 
By our choice of $C$ and $D$, observe that for all $N$ large enough
\begin{align} \label{eq:FirstTaylor}
\abs{(\Delta\phi)(1)+(1-\alpha-\gamma)\phi(1) + \lambda_N\phi(1)}  &\leq  \frac{c_1}{M^3} \ \text{ for 
 } c_1>0 \nonumber\\
\abs{(\Delta\phi)(N)+(1-\beta-\delta)\phi(N) + \lambda_N \phi(N)}  &\leq  \frac{c_2}{M^3} \ \text{ for } c_2>0
\end{align}
holds using the Taylor expansion of the sine and trigonometric identities. For fixed $x\in \mathbb{Z}/(2N)\mathbb{Z}$, we let $g_{x}(\eta)=\eta(x)$ for all $\eta\in \Omega_N$, and note that
\begin{align*}
\sum_{x=1}^{N}(\mathcal{L}g_x)(\eta) \phi(x) &= \sum_{x=1}^{N}(\Delta\eta)(x) \phi(x) +\phi(1)\Big(\eta(1)(1-\alpha-\gamma)+ \alpha-\frac{1}{2}\Big)  \\
&+  \phi(N)\Big(\eta(N)(1-\beta-\delta) +\delta-\frac{1}{2}\Big)
\end{align*} and $(\mathcal{L}g_x)(\eta)=-(\mathcal{L}g_{2N+1-x)})(\eta)$ for all $x\in [N]$ and $\eta\in \Omega_N$. Since
\begin{equation*}
\sum_{x=1}^{2N}(\Delta\eta)(x) \phi(x) = \sum_{x=1}^{2N}(\Delta\phi)(x) \eta(x)
\end{equation*} we can use $(\Delta\phi)(0)\eta(0)=-(\Delta\phi)(1)(1-\eta(1))$ and $(\Delta\phi)(N+1)=-(\Delta\phi)(N)(1-\eta(N))$ to see that
\begin{align*}
\sum_{x=1}^{2N}(\mathcal{L}g_x)(\eta) \phi(x) &= \sum_{x=2}^{N-1}(\Delta\eta)(x) \phi(x) +  \sum_{x=N+2}^{2N-1}(\Delta\eta)(x) \phi(x)   \\
&+2\phi(1)\Big(\eta(1)(1-\alpha-\gamma)+ \alpha-\frac{1}{2}\Big) +  2\phi(N)\Big(\eta(N)(1-\beta-\delta) +\delta-\frac{1}{2}\Big) \\
&+  (\Delta\phi)(1)(2\eta(1)-1) +  (\Delta\phi)(N)(2\eta(N)-1) \, .
\end{align*}
In particular, using \eqref{eq:FirstTaylor}  to rewrite $(\Delta\phi)(1)$ and $(\Delta\phi)(N)$, we get that
\begin{equation*}
\abs{   (2\eta(N)-1)\phi(N)(1-\beta-\delta) + (2\eta(1)-1)\phi(1)(1-\alpha-\gamma) + \sum_{x=1}^{2N}((\Delta\eta)(x) + \lambda_N \eta(x)) \phi(x)   } 
\end{equation*} is bounded from above by $2(c_1+c_2)M^{-3}$.
This yields
\begin{align*}
\abs{(-\mathcal{L})\Phi_N(\eta)-\lambda_N\Phi_N(\eta) } \leq  \frac{2(c_1+c_2)}{M^3}
\end{align*} and gives condition \eqref{eq:Eigenfunction} in Lemma \ref{lem:GeneralizedWilsonContinuousTime}. To verify condition \eqref{eq:QVBound}, we follow the ideas of the proof of Lemma 2.2 in \cite{L:CycleDiffusiveWindow}.
Observe that the process $(\Phi(\eta_t))_{t \geq 0}$ can change its value only when an edge or boundary vertex is updated. This happens at a rate $N^{\prime}$ Poisson clock where $N^{\prime}:=N-1+\alpha+\beta+\gamma+\delta$. For two configurations $\eta$ and $\eta^{\prime}$ which differ by at most one transition, we have that
 $$\abs{\Phi(\eta)- \Phi(\eta^{\prime})} \leq 2\max_{x \in [2N]}\abs{\phi(x)-\phi(x+1)} \leq \frac{c_3}{M} \, $$ for some constant $c_3=c_3(C,D)>0$. Combining these observations, we conclude
\begin{equation*}
\frac{\diff}{\diff t} \E\left[\langle M \rangle_t \right] \leq  N^{\prime} \left(\frac{c_3}{M} \right)^{2} \ \text{ for all } t\geq 0.
\end{equation*} This gives the desired bound on $R$ of order $N^{-1}$. Since $\max(\abs{\Phi_N(\mathbf{1})},\abs{\Phi_N(\mathbf{0})})$ is of order $N$, we see that Lemma \ref{lem:GeneralizedWilsonContinuousTime} yields the lower bound stated in Theorem \ref{thm:symmetricTwoSided}.
\end{proof}

Next, we consider the case of the simple exclusion process with open boundaries when particles are allowed to enter and exit the segment only at one side. Without loss of generality, assume that $\max(\alpha,\gamma)=0$ and $\max(\beta,\delta)>0$ holds. We will provide an argument similar to Lemma \ref{lem:EigenfunctionTwoSides}. To do so, we will use the height function representation of the simple exclusion process with open boundaries defined in Section \ref{sec:HeightFunction}.

\begin{lemma}\label{lem:EigenfunctionOneSide} Recall that $p=\frac{1}{2}$ and assume that $\max(\alpha,\gamma)=0$ and $\max(\beta,\delta)>0$ holds. For $D$ defined in \eqref{eq:Paramtercd}, let $\tilde{\phi} \colon \Z/(2N)\Z \rightarrow \R$ be
\begin{equation}
\tilde{\phi}(x) :=\sin\left(\frac{x \pi}{2(N+D)}\right) \ \text{ for all } x \in [N-1]
\end{equation} and set $\tilde{\phi}(x)=\tilde{\phi}(2N+1-x)$ for all $x \in \{ N+1,\dots,2N\}$. Moreover, we set $\tilde{\lambda}_N := 1-\cos(\frac{\pi}{2(N+D)})$ and define
\begin{equation}\label{eq:definitionPhiN}
 \tilde{\phi}(N):= \frac{1}{\beta+\delta-\tilde{\lambda}_N}\tilde{\phi}(N-1) \, .
\end{equation} Recall the height function for the simple exclusion process defined in \eqref{def:HeightFunction} and set
\begin{equation}
\tilde{\Phi}_N(\eta) := \sum_{x=1}^{2N}h_{\eta}(x)\tilde{\phi}(x) +\frac{\tilde{\phi}(N)}{\tilde{\lambda}_N}(\beta-\delta) \ \text{ for all } \eta \in \Omega_N.
\end{equation} Then $\tilde{\Phi}_N$ satisfies the conditions of Lemma \ref{lem:GeneralizedWilsonContinuousTime} for $\lambda=\tilde{\lambda}_N$, some $c$ of order $N^{-4}$, some $R$ of order $N$ and $\pdist{\tilde{\Phi}_N}{\infty}$ of order $N^2$. In particular, the lower bound in Theorem \ref{thm:symmetricOneSide}  holds for $\max(\alpha,\gamma)=0$ and $\max(\beta,\delta)>0$.
\end{lemma}
\begin{proof} 
Using trigonometric identities, we have that
\begin{equation}
(\Delta\tilde{\phi})(x) = -\tilde{\lambda}_N \tilde{\phi}(x)
\end{equation} holds for all $x \in \{1,\dots,N-2\}\cup \{N+2,\dots,2N-1\}$.
By our choice of $D$, observe that
\begin{equation} \label{eq:SecondTaylor}
\abs{\frac{1}{2}\tilde{\phi}(N-2) + \frac{\beta+\delta}{2}\tilde{\phi}(N)- (1-\tilde{\lambda}_N)\tilde{\phi}(N-1)} \leq  \frac{\tilde{c}}{N^4}
\end{equation}
 holds for some constant $\tilde{c}>0$ using the Taylor expansion for $\tilde{\phi}$ and $\tilde{\lambda}_N$, and  trigonometric identities. For fixed $x\in [2N]$, set $G_x(\eta)=h_{\eta}(x)$ for all $\eta\in \Omega_N$, and observe that
\begin{align*}
(\mathcal{L}G_x)(\eta) = (\Delta h_{\eta})(x)  + \mathds{1}_{\{ x=N\}} \left(\left(1-\delta-\beta\right)(\Delta h_{\eta})(N) + \delta-\beta\right) \, .
\end{align*} 
Together with the facts that $h_\eta(0)=0$ and $\tilde\phi(0)=0$, and \eqref{eq:SecondTaylor} we obtain
\begin{align*}
\sum_{x=1}^{2N-1}(\mathcal{L}G_x)(\eta) \tilde\phi(x) &= \sum_{x=1}^{2N-1}(\Delta h_\eta)(x) \tilde\phi(x) +\tilde\phi(N)((\Delta h_\eta)(N)(\beta+\delta-1)+ \delta-\beta) \\
&= - \tilde{\lambda}_N \sum_{x=1}^{2N-1} G_x(\eta)  \tilde\phi(x) + \tilde{\phi}(N)(\delta-\beta) \, .
\end{align*} 
Hence, we see that
\begin{align*}
|(-\mathcal{L})\tilde{\Phi}_N(\eta)-\tilde{\lambda}_N\tilde{\Phi}_N(\eta)| \leq \frac{\tilde{c}}{N^4}
\end{align*} holds, which gives condition \eqref{eq:Eigenfunction} of Lemma \ref{lem:GeneralizedWilsonContinuousTime} for $c$ of order $N^{-4}$. For condition \eqref{eq:QVBound}, we again follow the ideas of the proof of Lemma 2.2 in \cite{L:CycleDiffusiveWindow}. Note that the process $(\tilde{\Phi}(\eta_t))_{t \geq 0}$ can change its value only when an edge or boundary vertex is updated. This happens at a rate $N^{\prime}$ Poisson clock for $N^{\prime}=N-1+\beta+\delta$. For two configurations $\eta$ and $\eta^{\prime}$ which differ by at most one transition, observe that $\tilde{\Phi}(\eta)$ and $\tilde{\Phi}(\eta^{\prime})$ differ by at most $2$. Hence, we conclude that
\begin{equation*}
 \frac{\diff}{\diff t}\E\left[ \langle M \rangle_t\right] \leq  4 N^{\prime} \ \text{ for all } t\geq 0.
\end{equation*} This gives the desired bound on $R$ of order $N$. Since we have that $\max(|\tilde{\Phi}(\mathbf{1})|,|\tilde{\Phi}(\mathbf{0})|)$ is of order $N^2$, Lemma \ref{lem:GeneralizedWilsonContinuousTime} yields the desired lower bound.
\end{proof}
Combining Lemma \ref{lem:EigenfunctionTwoSides} and Lemma \ref{lem:EigenfunctionOneSide}, this finishes the proof of the lower bounds in Theorem \ref{thm:symmetricTwoSided} and Theorem \ref{thm:symmetricOneSide}.

\section{Upper bounds for the SSEP with open boundaries}\label{sec:UpperBounds}

In this section, we prove the upper bounds in Theorem \ref{thm:symmetricTwoSided} and Theorem \ref{thm:symmetricOneSide}. We start with a general upper bound for the simple exclusion process with open boundaries for $p=\frac{1}{2}$ and arbitrary boundary rates with $\max(\alpha,\beta,\gamma,\delta)>0$. This bound is refined in Section \ref{sec:UpperBoundCutoff} when particles enter and exit only at one side of the segment. 

\subsection{A general upper bound}\label{sec:GeneralUpperBoundSymmetric}

We now prove the upper bound in Theorem \ref{thm:symmetricTwoSided}.
Without loss of generality, assume that $\max(\alpha,\beta,\gamma,\delta)=\alpha$ holds as we can flip the segment and use the particle -- empty site symmetry, otherwise. Let $\tau$ be the first time at which all second class particles have left in the disagreement process $(\xi_t)_{t \geq 0}$. By Corollary \ref{cor:SecondClassParticles} (taking $t = \log\frac{N}{\varepsilon}$) it suffices to show that for some constant $c>0$, and all $t \geq 0$,
\begin{equation}\label{eq:TauSymmetric}
\mathbf{P}\left( \tau > c t N^2 \right) \leq N e^{-t}\, .
\end{equation} Since $p=\frac{1}{2}$ and objects of the same type are indistinguishable, we can also describe the dynamics along the edges such that the values of the endpoints are swapped at rate $\frac{1}{2}$, independently. From this perspective, the second class particles perform continuous-time simple random walks with absorption at (at least one of) the boundaries. Using a comparison to the Gambler's ruin problem on $[N]$ with reflection at the right-hand side, we see that with probability at least $\frac{1}{2}$, a given second class particle gets either absorbed or reaches site $1$ by time $2N^2$. Note that this bound does not depend on the starting point of the particle. 
Moreover, for a second class particle at site $1$ at time $t$, with probability at least $(1- e^{-\alpha})e^{-1/2}$ the particle gets absorbed at the boundary by time $t+1$. Thus, we have that
\begin{equation}\label{eq:SymmetricExitingSecondClass1}  
\mathbf{P}\left(  \tau_{\ast} >  2N^2+1  \right) \leq  1-
\frac{1- e^{-\alpha}}{2e^{1/2}}
\end{equation} 
holds, where $\tau_{\ast}$ denotes the absorption time of a fixed second class particle in the above dynamics. Using \eqref{eq:SymmetricExitingSecondClass1} and the Markov property, we see that
\begin{equation}\label{eq:SymmetricExitingSecondClass2}
\mathbf{P}\left(  \tau_{\ast} > t \frac{2e^{1/2}}{1- e^{-\alpha}}
(2N^2+1)  \right) \leq 
\left(1- \frac{1- e^{-\alpha}}{2e^{1/2}}\right)^{t\frac{2e^{1/2}}{1- e^{-\alpha}}} \leq 
e^{-t} \ .
\end{equation} for all $t\in \N$. The inequality \eqref{eq:TauSymmetric}, and hence the upper bound in Theorem \ref{thm:symmetricTwoSided} follow using a union bound on the events in \eqref{eq:SymmetricExitingSecondClass2}, and choosing $c$ accordingly. 
\begin{remark}\label{rem:relaxSymmetric} Note that by a standard argument, the bound in \eqref{eq:TauSymmetric} implies that any eigenvalue $\lambda$ of the generator of the symmetric simple exclusion process with open boundaries must satisfy $|\lambda|^{-1} \leq cN^2$, see Corollary 12.7 in \cite{LPW:markov-mixing} for a similar statement for reversible, discrete-time Markov chains.
\end{remark}

\subsection{Cutoff for the SSEP with one open boundary}\label{sec:UpperBoundCutoff}

In this section, we prove the upper bound in Theorem \ref{thm:symmetricOneSide} using the ideas and results of \cite{L:CutoffSEP}. Since large parts of the proof will follow verbatim from the arguments in Section 8 of \cite{L:CutoffSEP} for the simple exclusion process, we will focus on presenting the required adjustments in the proof rather than giving full details. In Sections \ref{sec:CorrelationProperty} to \ref{sec:ScalingLimitHeightFunction}, we collect some technical results on the simple exclusion process with open boundaries. Together with the results presented in Section \ref{sec:PreliminariesBDEP}, this will cover the corresponding preliminaries on the simple exclusion process in Section 6 of \cite{L:CutoffSEP}. In Section \ref{sec:ProofSymmetric}, we highlight how these results are used if one adapts the arguments of \cite{L:CutoffSEP} for the simple exclusion process with one open boundary.

\subsubsection{Correlation properties of the SSEP with one open boundary}\label{sec:CorrelationProperty}
Our first preliminary result is the FKG-inequality as well as a corollary of Holley's inequality for the simple exclusion process with $p=\frac{1}{2}$ and one open boundary. For any two configurations $\eta,\zeta \in \Omega_N$, we let $\min(\eta,\zeta)$ and $\max(\eta,\zeta)$  be the configurations in $\Omega_N$ which satisfy
\begin{equation}
h_{\min(\eta,\zeta)}(x) := \min\left( h_{\eta}(x),h_{\zeta}(x)\right) \quad \text{ and } \quad h_{\max(\eta,\zeta)}(x) := \max\left( h_{\eta}(x),h_{\zeta}(x)\right)
\end{equation} for all $x \in [N]$, respectively. Note that $\min(\eta,\zeta)$ and $\max(\eta,\zeta)$ are indeed elements of $\Omega_N$ and $\Omega_N$ equipped with these operations is a distributive lattice. 
By \eqref{eq:reversibleDistribution}
\begin{equation*}
\mu(\min(\eta,\zeta)) \mu(\max(\eta,\zeta)) = \mu(\eta)\mu(\zeta)
\end{equation*} holds when $\delta\geq \beta$, and similarly for $\delta< \beta$. With these insights, the next result follows from the same arguments as Proposition 6.1 in \cite{L:CutoffSEP}.
\begin{lemma}[c.f. \cite{L:CutoffSEP}, Proposition 6.1] \label{lem:Correlations} For any two functions $f$ and $g$ on $\Omega_N$ which are increasing with respect to the partial order $\succeq_\h$ on $\Omega_N$, we have that
\begin{equation}\label{eq:PositiveCorrelation}
\int fg d\mu \geq \int f d\mu \int g d\mu \, .
\end{equation} 
Moreover, we have for any two increasing subsets $A \subseteq B$ of $\Omega_N$ with 
\begin{equation}\label{eq:Holley}
\left\{ \min(\eta,\zeta) | \eta \in A, \zeta \in B\right\} \subseteq B
\end{equation} that $\frac{1}{\mu(A)}\int\limits_A f d\mu \geq \frac{1}{\mu(B)}\int\limits_B f d\mu $ holds for any increasing function $f$.
\end{lemma}

\subsubsection{Mean of the height function of the SSEP with one open boundary}\label{sec:MeanHeightFunction}

Next, we give an estimate on the mean of the height function of the simple exclusion process with $p=\frac{1}{2}$ and one open boundary. For a given $\eta \in \Omega_N$, we define, recalling \eqref{def:HeightFunction}
\begin{equation}\label{def:HeightStarFunction}
h^{\ast}_{\eta}(x) := h_{\eta}(x)-  \min(x,2N+1-x)\frac{\delta-\beta}{\delta+\beta} \ \text{ for all } x \in [2N].
\end{equation} Intuitively, $h^{\ast}_{\eta}$ is the height function of $\eta$ after subtracting the mean height according to equilibrium.

\begin{lemma}[c.f. \cite{L:CutoffSEP}, Lemma 6.4] \label{lem:MeanHeightFunction} For all $N$ large enough, we have that
\begin{equation}
\max_{x \in \{ 0,\dots,2N\}}\abs{\E_{\eta}\left[ h^{\ast}_{\eta_t}(x)\right] }\leq 3N e^{-\lambda t} 
\end{equation} holds for all $t \geq 0$ and initial states $\eta \in \Omega_N$, where $\lambda=1-\cos(\frac{\pi}{2N+(\beta+\delta)^{-1}})$. 
\end{lemma}
\begin{proof} 
Observe that the function $f_{\eta} \colon \{ 0,\dots,2N\} \times \R_0^+ \rightarrow \R$ with $f_{\eta}(x,t):=\E_{\eta}[h^{\ast}_{\eta_t}(x)]$
for some initial state $\eta \in \Omega_N$ is a solution to the system of equations
\begin{equation}\label{eq:HeatEquationModified}
\begin{cases} \partial_tf_{\eta} = (\mathds{1}_{\{x\neq N\}}+ \mathds{1}_{\{x=N\}}(\beta+\delta))\Delta f_{\eta} \\
f_\eta(0,t)=f_\eta(2N,t)=0\end{cases}
 \end{equation} for all $t\geq 0$ and $x \in \{ 0,\dots,2N\}$, with initial condition $f_\eta(x,0)=h^{\ast}_{\eta}(x)$. Here,
 $\Delta$ denotes the discrete Laplace operator which is defined in \eqref{def:Laplacian}.
Using Taylor expansion and a continuity argument, we see that there exists some $c_N \in [\frac{1}{2(\beta+\delta)}-1, \frac{1}{2(\beta+\delta)}]$ such that for all $N$  large enough, the function $g\colon \{ 0, 1, \dots, 2N\} \rightarrow \R$ with 
\begin{equation*}
g(x,t) :=  \left(\mathds{1}_{\{x \leq N\}} \sin\left( \frac{x\pi}{2(N+c_N)} \right)+ \mathds{1}_{\{x > N\}} \sin\left( \frac{(2N-x)\pi}{2(N+c_N)}\right) \right)e^{-\lambda_N t}
\end{equation*} for all $x \in \{ 0, 1, \dots, 2N\}$, $t \geq 0$ and $\lambda_N := 1-\cos\left( \frac{\pi}{2(N+c_N)}\right)$, is a solution to \eqref{eq:HeatEquationModified}, with initial condition $g(x,0)$. Note that $\sin(z\pi/2) \geq \min(z,2-z)$ holds for all $z \in [0,2]$, and $c_N \geq -1$. Hence,  we have for sufficiently large $N$ that $$|h^{\ast}_{\eta}(x)| \leq 2\min(x,2N-x) \leq 3N g(x,0)$$ for all $x \in \{ 0, 1, \dots, 2N\}$ and $\eta \in \Omega_N$. Since this relation is preserved in \eqref{eq:HeatEquationModified} over time, we conclude.
\end{proof}

\subsubsection{Scaling limits for the SSEP with one open boundary}\label{sec:ScalingLimitHeightFunction}

We now study the law of the height function in equilibrium. 
\begin{lemma}[c.f. \cite{L:CutoffSEP}, Lemma 8.5] \label{lem:ScalingLimit}Let $\eta$ be a configuration sampled according to the stationary distribution $\mu$ of the simple exclusion process with $p=\frac{1}{2}$ and one open boundary. Then 
\begin{equation}
\left(\frac{\beta+\delta}{\sqrt{N \beta \delta}} h^{\ast}_\eta(xN) \right)_{x\in [0,1]}
\end{equation}
converges for $N \rightarrow \infty$ in law to a standard Brownian motion on the interval $[0,1]$. 
\end{lemma}
\begin{proof} Using the explicit form of the invariant distribution $\mu$ in \eqref{eq:reversibleDistribution} for $p=\frac{1}{2}$ and the Binomial theorem, we see that 
the total number of particles $|\eta|$ in a configuration $\eta$ according to $\mu$ is Binomial-$(N,\frac{\delta}{\beta+\delta})$-distributed. Conditioning on the number of particles in the segment, observe that
the number of particles in $\eta$ until position $y$ is Binomial-$(y,\frac{\delta}{\beta+\delta})$-distributed. The convergence for all finite marginals follows from the De Moivre-Laplace theorem. Together with a tightness argument, we obtain the convergence in law to a standard Brownian motion on $[0,1]$.
\end{proof}

\subsubsection{Proof of the upper bound in Theorem \ref{thm:symmetricOneSide}}\label{sec:ProofSymmetric}

The upper bound in Theorem \ref{thm:symmetricOneSide} is shown in two steps. First, we give an upper bound on the time it takes to reach equilibrium when starting from the two extremal configurations $\mathbf{1}$ and $\mathbf{0}$. In the next step, we consider a suitable coupling such that the exclusion processes started from  $\mathbf{1}$ and $\mathbf{0}$ agree with high probability. This will be formalized in Lemma \ref{lem:mainTheoremUpperBound1} and Lemma \ref{lem:mainTheoremUpperBound2}. 
\begin{lemma}[c.f. \cite{L:CutoffSEP}, Propositions 8.2] \label{lem:mainTheoremUpperBound1}Let $(\eta^{\mathbf{1}}_t)_{t\geq 0}$ and $(\eta^{\mathbf{0}}_t)_{t\geq 0}$ denote the simple exclusion processes with one open boundary and $p=\frac{1}{2}$ started from the configurations $\mathbf{1}$ and $\mathbf{0}$, respectively. For a given $\varepsilon>0$, we set
\begin{equation} t_0:=\frac{4}{\pi^2}N^2\log N \left(1+\frac{\varepsilon}{2}\right)  \ .
\end{equation} Then we have that 
\begin{align} \label{eq:ConvergenceToEquilibrium}
\lim_{N \rightarrow \infty} \TV{P(\eta^{\mathbf{1}}_{t_0} \in \cdot ) - \mu} =0 \quad \text{and} \quad
\lim_{N \rightarrow \infty} \TV{P(\eta^{\mathbf{0}}_{t_0} \in \cdot ) - \mu} =0 
\end{align} holds for all $\varepsilon>0$. 
\end{lemma}

\begin{proof}[Sketch of the proof] The proof of Lemma \ref{lem:mainTheoremUpperBound1} is divided into two main steps. First, we consider the simple exclusion process $(\eta_t)_{t \geq 0}$ with open boundaries for initial states $\mathbf{1}$ and $\mathbf{0}$ up to time $t_2$, where
\begin{equation}
t_2:=\frac{4}{\pi^2}N^2\log N \left(1+\frac{\varepsilon}{4}\right) \ .
\end{equation} We study the functions $(h^{\ast}_{\eta_t})_{t\geq 0}$, defined in \eqref{def:HeightStarFunction}, and evaluate them at $x_i :=\lfloor 2iN/K \rfloor$ for $K:=\varepsilon^{-1}$ and $i \in \{ 0,\dots,K \}$. Following  \cite{L:CutoffSEP}, we call the dynamics restricted to $(x_i)_{i \in [K]}$ the \textbf{skeleton}. Our goal is to argue that when the mean of $(h^{\ast}_{\eta_t})_{t\geq 0}$ at time $t_2$ has at most the order of the typical fluctuations within the stationary distribution $\mu$, the law of the skeleton at time $t_2$ is in total-variation distance close to equilibrium. This follows by applying the same arguments as for the proof of Lemma 8.4 in \cite{L:CutoffSEP}, replacing Proposition 6.1 and Lemma 8.5 in \cite{L:CutoffSEP} by Lemma \ref{lem:Correlations} and Lemma \ref{lem:ScalingLimit}, respectively.
In order to conclude the first step, use Lemma \ref{lem:MeanHeightFunction} to see that for initial state $\eta \in \Omega_N$, we have that $\E_{\eta}[h^{\ast}_{\eta_t}(x)]$ is at most of order $\sqrt{N}$ at time $t=t_2$ for all $x \in [2N]$. 
In a second step, we apply the censoring inequality in Lemma \ref{lem:Censoring} for the censoring scheme 
\begin{equation*}
\mathcal{C}(t) = \left\{ \{x_i,x_i+1\} \colon i \in [K]\right\} \ , 
\end{equation*} where $t \in [t_2,t_0]$, in order to show that the dynamics mixes locally. In words, this censoring scheme ensures that the number of particles in the interval $[x_{i-1},x_{i}]$ for all $i \in [K]$ remains almost surely constant between $t_2$ and $t_0$. Thus, we have $K$ independent simple exclusion processes on a closed segment during this period. Together with the above bounds at time $t_2$, the remainder of the argument is analogous to the proof of Proposition 8.2 in \cite{L:CutoffSEP}.
\end{proof}
Note that Lemma \ref{lem:mainTheoremUpperBound1} does not immediately imply Theorem \ref{thm:symmetricOneSide} since there could be an initial state other than $\mathbf{1}$ or $\mathbf{0}$, which maximizes the distance from equilibrium.
However, using Lemma \ref{lem:mainTheoremUpperBound1}, we obtain the following result which allows us to conclude the upper bound in Theorem \ref{thm:symmetricOneSide} using Lemma \ref{lem:HittingTime}.
\begin{lemma}[c.f. \cite{L:CutoffSEP}, Propositions 8.1] \label{lem:mainTheoremUpperBound2} For a given $\varepsilon>0$, we set
\begin{equation}  t_1:=\frac{4}{\pi^2}N^2\log N(1+\varepsilon) \ .
\end{equation} Then there exists a coupling $\tilde{P}$ which respects $\succeq_\h$ such that
\begin{equation}
\lim_{N \rightarrow \infty}\tilde{P}(\eta^{\mathbf{1}}_{t_1} \neq \eta^{\mathbf{0}}_{t_1}) = 0
\end{equation} is satisfied for all $\varepsilon>0$.
\end{lemma}
\begin{proof}[Sketch of the proof]  In order to show Lemma \ref{lem:mainTheoremUpperBound2} using Lemma \ref{lem:mainTheoremUpperBound1}, we consider a coupling which is monotone with respect to $\succeq_\h$ and maximizes the fluctuations of $(h^{\ast}_{\eta_t})_{t\geq 0}$. We use the construction of the alternative coupling defined in \cite[Section 8.4]{L:CutoffSEP}. However, for all transitions where particles enter and exit the segment, we apply the update rule of the canonical coupling for the simple exclusion process with open boundaries, i.e.\ we use the same rate $\beta$ and rate $\delta$ Poisson clocks in both simple exclusion processes to determine when a boundary vertex is updated. The proof of Lemma \ref{lem:mainTheoremUpperBound2} follows the same arguments as the proof of Proposition 8.1 given in \cite[Section 8.4]{L:CutoffSEP}, replacing Lemma 8.5 in \cite{L:CutoffSEP} by Lemma \ref{lem:ScalingLimit}.
\end{proof}

\section{Mixing times for ASEP with one blocked entry}\label{sec:ProofAsymmetricHomogenousOneSide}

In this section, we prove Theorem \ref{thm:asymmetricOneSide} for the asymmetric simple exclusion process with one blocked entry. We start by defining the simple exclusion process on the half-line as an auxiliary process, and investigate its current. By a coupling to the original dynamics, this allows us to deduce the lower bound on the mixing time. For the upper bound, we again use the simple exclusion process on the half-line to estimate the hitting time with respect to the extremal states $\mathbf{0}$ and $\mathbf{1}$, and conclude by a shock wave argument. In the following, we only consider the case $\alpha=0$ and $\beta>0$, and prove \eqref{eq:OneBlockeda}, since \eqref{eq:OneBlockedb} follows from the same arguments using the symmetry between particles and empty sites. 

\subsection{The simple exclusion process on the half-line} \label{sec:SEPhalfLine}

In the following, a key tool in our investigations will be the simple exclusion process $(\sigma_t)_{t\geq 0}$ on the half-line $\N=\{1,2,\dots \}$ with drift $p\in [\frac{1}{2},1]$, where particles enter at rate $\tilde{\alpha}$ and exit at rate $\tilde{\gamma}$ at site $1$. Formally, $(\sigma_t)_{t\geq 0}$ is defined as the Feller process on $\{0,1\}^{\N}$ generated by
\begin{align}\label{def:generatorHalfLine}
(\mathcal{L}^{\N}f)(\eta)&= \sum_{x =1}^{\infty} \big(p \eta(x)(1-\eta(x+1))+ (1-p)  \eta(x+1)(1-\eta(x)) \big)\left[ f(\eta^{x,x+1})-f(\eta) \right] \nonumber \\
&+  \big(\tilde{\alpha} (1-\eta(1)) +\tilde{\gamma} \eta(1) \big) \left[ f(\eta^{1})-f(\eta) \right] 
\end{align} for all cylinder functions $f$. Recall the notion of the current  from \eqref{def:current}, and let $(J_t^{\N})_{t \geq0}$ be the current of $(\sigma_t)_{t\geq 0}$, i.e.\ the net number of particles entering at the left-hand side boundary. Moreover, recall \eqref{def:partialOrder} and the  stochastic domination for probability measures from Section \ref{sec:Invariant}. We  have the following bound on the current of the simple exclusion process on the half-line, which extends  results from \cite{L:ErgodicI} for general boundary parameters.

\begin{lemma}\label{lem:CurrentHalfLine} Let $\tilde{\alpha}>0$ and $p>\frac{1}{2}$, and recall $a=a(\tilde{\alpha},\tilde{\gamma},p)$ from \eqref{def:a}. Then the simple exclusion process $(\sigma_t)_{t\geq 0}$ started from the empty initial configuration $\mathbf{0}$ satisfies
\begin{equation}\label{eq:ComparisonHalfLineProduct}
\P_{\mathbf{0}}\left( \sigma_t \in \cdot \right) \preceq_{\c} \nu_{\frac{1}{1+a}} \ ,
\end{equation} where $\nu_{\frac{1}{1+a}}$ denotes the Bernoulli-$\frac{1}{1+a}$-product measure on $\N$. Furthermore,
\begin{equation}\label{eq:CurrentHalfLineSEP}
\lim_{t \rightarrow \infty} \frac{J_t^{\N}}{t} = (2p-1)\frac{\max(a,1)}{(\max(a,1)+1)^{2}}
\end{equation} holds almost surely.
\end{lemma}
\begin{proof} Note that the measure  $\nu_{\frac{1}{1+a}}$ is invariant for the simple exclusion process on the half-line. This can be seen by a direct calculation using the generator in \eqref{def:generatorHalfLine}, or alternatively, by Lemma \ref{lem:stationaryDistribution} when taking $b = 1/a$ and letting the size of the segment go to infinity. Hence, \eqref{eq:ComparisonHalfLineProduct} follows using the canonical coupling and monotonicity for the simple exclusion process on the half-line when starting from  $\nu_{\frac{1}{1+a}}$. \\
For $\leq$ in \eqref{eq:CurrentHalfLineSEP}, we compare $(\sigma_t)_{t\geq 0}$ to a simple exclusion process $(\eta_t)_{t \geq 0}$ on the segment of size $N$ with drift $p$ and boundary parameters $\alpha=\tilde{\alpha}$, $\beta=p$, $\gamma=\tilde{\gamma}$, $\delta=0$ (implying $b=0$) which is started from the empty configuration. We adjust now the canonical coupling $\mathbf{P}$ such that we use the same Poisson clocks in both processes on the sites $[N-1]$, and try to remove a particle in $(\eta_t)_{t \geq 0}$ at site $N$ whenever the clock for performing a jump from site $N$ to $N+1$ in $(\sigma_t)_{t \geq 0}$ rings. In particular, the coupling ensures that when $\eta_t(x)=1$ holds for some $x \in [N]$, then $\sigma_t(x)=1$, provided that both processes agreed initially on $[N]$. Therefore, the current $(J^{N}_t)_{t \geq0}$ of $(\eta_t)_{t \geq 0}$ satisfies $J^{N}_t \geq J^{\N}_t$ for all $t\geq0$ and $N\in \N$, $\mathbf{P}$-almost surely, and we conclude
$\leq$ in \eqref{eq:CurrentHalfLineSEP} by Lemma \ref{lem:current} and taking $N \rightarrow \infty$. \\
For $\geq$ in \eqref{eq:CurrentHalfLineSEP}, we assume without loss of generality that $a\geq 1$. This is due to the fact that for $a=a(\tilde{\alpha},\tilde{\gamma},p) < 1$, we can decrease $\tilde{\alpha}$ continuously until we reach $a=1$, and apply (a half-line version of) Lemma \ref{lem:MonotoneCouplingComponentwise} to see that this will only decrease the current. Using Lemma \ref{lem:MonotoneCouplingComponentwise} again, and the canonical coupling $\mathbf{P}$, we see that the current in \eqref{eq:CurrentHalfLineSEP} is bounded from below by the current in $(\sigma_t)_{t\geq 0}$ when starting initially from $\nu_{\frac{1}{1+a}}$. In order to conclude \eqref{eq:CurrentHalfLineSEP}, it suffices to show that  $\nu_{\frac{1}{1+a}}$ is extremal invariant, i.e.\ the measure $\nu_{\frac{1}{1+a}}$ is an extremal point in the (convex) set of invariant measures of $(\sigma_t)_{t\geq 0}$, see Theorem B.52 of  \cite{L:Book2}. Following the arguments of Theorem 1.17 in Part III of \cite{L:Book2}, which relates the extremal invariant measures of the ASEP on $\Z$ to those of the SSEP on $\Z$, we have to show that $\nu_{\frac{1}{1+a}}$ is extremal invariant for the simple exclusion process on the half-line with $p=\frac{1}{2}$, where particles enter at rate $\frac{1}{2}(\tilde{\alpha}+a^{-1}\tilde{\gamma})$ and exit at rate $\frac{1}{2}(\tilde{\gamma}+a\tilde{\alpha})$, respectively. As observed in Section 2 of \cite{J:Extremal}, a sufficient condition for some distribution $\nu$ to be extremal invariant is that for any finite set $A \subseteq \N$, we have
\begin{equation}\label{eq:VanishingParticles}
\lim_{t \rightarrow \infty}  \P_{\nu^{A}_{1}}\left( \sigma_t \in B \right) =  \lim_{t \rightarrow \infty} \P_{\nu^{A}_{0}}\left( \sigma_t \in B \right) \quad \text{ for all finite }  B \subseteq \N,
\end{equation}
where $\nu^{A}_{1}(\ \cdot \ ):= \nu(\ \cdot \mid \eta(x)=1 \  \forall x\in A)$ and $\nu^{A}_{0}(\ \cdot \ ):= \nu(\ \cdot \mid \eta(x)=0 \ \forall x\in A)$. 
The process  
$(\sigma_t)_{t\geq 0}$ can be realized as a  disagreement process within the canonical coupling, where we start with initial laws  $\nu^{A}_{1}$ and $\nu^{A}_{0}$. Since $\nu_{\frac{1}{1+a}}$ is a product measure, we have initially second class particles on $A$, and a Bernoulli-$\frac{1}{1+a}$-product measure everywhere else. Since $p=\frac{1}{2}$,  we can view the dynamics as an interchange process, where all second class particles perform symmetric simple random walk on $\N$, with absorption at site $1$ at rate at least $(\tilde{\alpha}+a^{-1}\tilde{\gamma}+\tilde{\gamma}+a\tilde{\alpha})/2>0$. This allows us to conclude \eqref{eq:VanishingParticles}.
\end{proof}

\subsection{Lower bound for the ASEP with one blocked entry}\label{sec:AsymmetricLowerBound}
We will now show $\geq$ in \eqref{eq:OneBlockeda}.
First, we assume $\alpha =\gamma=0$ as well as $\beta>0$. Using \eqref{eq:reversibleDistribution}, the stationary distributions $\mu= \mu_N$ of $(\eta_t)_{t\geq0}$ satisfy
\begin{equation}\label{eq:SmallProbabilitySet}
\lim_{N \rightarrow \infty} \mu\left(  \exists x \in  \lbrace 1,\dots, N-\sqrt{N} \rbrace \colon \eta(x)=1 \right) = 0.
\end{equation} Suppose we start from the configuration $\mathbf{1}$ with all sites being initially occupied. Using \eqref{def:TVDistance}, we see that in order to prove an asymptotic lower bound $t_N$ on $t^{N}_{\text{mix}}(\varepsilon)$ for all $\varepsilon \in (0,1)$, it suffices to show that with probability tending to $1$, no more than $N-\sqrt{N}$ particles have exited the segment by time $t_N$. Using the symmetry between particles and empty sites, we see that the number of particles which have exited the segment by time $t_N$ is dominated by the current of the simple exclusion process on the half-line with drift $p$ and boundary parameters $\tilde{\alpha}=\beta$, $\tilde{\gamma}=\delta$, evaluated at time $t_N$. Hence, we can conclude the lower bounds on the mixing time in Theorem \ref{thm:asymmetricOneSide} due to Lemma \ref{lem:CurrentHalfLine}. Note that for $\gamma>0$, the statement  \eqref{eq:SmallProbabilitySet} holds as well, due to Lemma \ref{lem:MonotoneCouplingComponentwise}. Consider the initial state $\eta^{N} \in \Omega_N$ given by
\begin{equation}
\eta^N(x) = \mathds{1}_{x \geq \sqrt{N}}
\end{equation} for all $x \in [N]$. Comparing the process started from $\eta^{N}$ to the blocking measure on $\Z$, we see that almost surely no particle reaches site $1$ by time $N^2$ for all $N$ sufficiently large, due to Lemma \ref{lem:HittingBlockingMeasure}. Hence, we can now use again the previous bound  via the current for the simple exclusion process on the half-line to conclude.

\subsection{An a priori upper bound on the hitting time}

In order to show $\leq$ in \eqref{eq:OneBlockeda},
we will bound the hitting time $\tau_{\mathbf{0}}$, i.e.\ the first time the process reaches $\mathbf{0}$, when all sites are initially occupied. We start with an a priori bound when the starting configuration contains a small number of particles and the particles are concentrated on the right-hand side. 

\begin{lemma} \label{lem:AprioriBound} Let $\alpha= \gamma =0$ and $\beta > 0$. For $k \in [N-1]$, assume that $\eta \in \Omega_N$ satisfies $\eta(i)=0$ for all $i \in [N-k]$. There exists $c=c(\beta,\delta,p)>0$ such that for all $k$ and $N$
\begin{equation}
\E_{\eta}[\tau_{\mathbf{0}}] \leq \exp(c k^3)\, .
\end{equation}
\end{lemma}
\begin{proof} Suppose that $\delta>0$ and $p<1$ holds. We define the first \textbf{return time} $\tau^{+}_{\mathbf{0}}$
\begin{equation}\label{def:returnTime}
\tau^{+}_{\mathbf{0}} := \inf \lbrace t \geq \tau_{\Omega_N \setminus \{ \mathbf{0}\}} \colon \eta_t = \mathbf{0}\rbrace
\end{equation} for the simple exclusion process with open boundaries, where for a set $A\subseteq \Omega_N$, we let $\tau_{A}$ denote the hitting time of $A$. Note that for all $\eta \in \Omega_N \setminus \{ \mathbf{0} \}$ 
\begin{equation*}
\E_{\eta}[\tau_{\mathbf{0}}]  \leq \E_{\mathbf{0}}[\tau^{+}_{\mathbf{0}}] \left( \P_{\mathbf{0}}(\tau_{\eta}< \tau_{\mathbf{0}}^{+}) \right)^{-1} \, .
\end{equation*} Further, by Kac's lemma $\E_{\mathbf{0}}[\tau^{+}_{\mathbf{0}}]=(\mu(\mathbf{0}))^{-1}$ holds, and thus $\E_{\mathbf{0}}[\tau^{+}_{\mathbf{0}}]$ is bounded uniformly in $N$ due to \eqref{eq:reversibleDistribution} (note that $Z_N$ in \eqref{eq:reversibleDistribution} is bounded uniformly in $N$).
Starting from $\mathbf{0}$, there exists a sequence of at most $k^2$ updates to reach $\eta$ involving only the rightmost $
k+1$ edges and the right-hand side boundary. Moreover, this sequence can be chosen in such a way that all other updates do not affect the evolution of the process. Thus, forcing the rate $1$ Poisson clocks along these edges to ring according to a given order, we see that
\begin{equation*}
\mu(\mathbf{0}) \P_{\mathbf{0}}(\tau_{\eta}< \tau_{\mathbf{0}}^{+})  \geq \mu(\mathbf{0}) \left( \frac{\min(1-p,\delta)}{k+\beta+\delta} \right)^{k^2} \geq \exp(-ck^3) 
\end{equation*} holds for some $c>0$. For $\delta=0$ or $p=1$, use Lemma \ref{lem:MonotoneCouplingHeightFunction} to bound $\E_{\eta}[\tau_{\mathbf{0}}]$ by the expected hitting time for a simple exclusion process with the same parameters, except for some different choices of $\delta>0$ and $p<1$.
\end{proof}  

\subsection{Upper bound for the ASEP with one blocked entry}

We now prove $\leq$ in \eqref{eq:OneBlockeda} 
for the asymmetric simple exclusion process $(\eta_t)_{t\geq 0}$ with one blocked entry. We will bound the hitting time $\tau_{\mathbf{0}}$ of the configuration $\mathbf{0}$, starting from configuration $\mathbf{1}$ where all sites are occupied, and conclude by Lemma \ref{lem:HittingTime} since $\tau_{\mathbf{0}}$ stochastically dominates the coupling time $\tau$ in \eqref{eq:couplingTimeTau} between the states $\mathbf{0}$ and $\mathbf{1}$. Using Lemma \ref{lem:MonotoneCouplingComponentwise}, we can assume without loss of generality that $\gamma=0$.  For all $k\in [N]$, let $\theta_k$ be the configuration in $\Omega_N$ where the rightmost $k$ sites are occupied and all other sites are empty, and recall that $L(\eta)$ denotes the position of the leftmost particle in $\eta$. We set  \begin{equation}
c_{b,p}:=(\max(b,1)+1)^2/((2p-1)\max(b,1)) \label{def:ConstantCBP}.
\end{equation} Our key tool is the following lemma, showing that the particles  travel like a \textit{shock wave} at linear speed until a time $\tilde{\tau}$ (defined later on). In particular,  $\tilde{\tau}$ will be a stopping time which describes that enough particles exited.

\begin{lemma}[Shock wave phenomenon]\label{lem:Shockwave} Assume $\alpha = \gamma = 0$ and $\beta > 0$. Further, let $p > \frac{1}{2}$ and $\delta\geq 0$. Let $\varepsilon,\tilde{\varepsilon}>0$. Then there exist $N_0,k_0 \in \N$ such that for all $N \geq N_0$ and $k\geq k_0$ with $k=k(N) \in [N]$, we find a stopping time $\tilde{\tau}$ such that $(\eta_t)_{t \geq 0}$ on $\Omega_N$ started from $\theta_k$ satisfies 
\begin{equation}\label{eq:ShockwaveStatement1}
\P\left( \big| L\big(\eta_{\tilde{\tau}}\big) - N \big| \leq \log^3 k  \text{ \normalfont and }\tilde{\tau} \leq  (1+\varepsilon)c_{b,p}k \mid \eta_0 = \theta_k \right) \geq  1 -\tilde{\varepsilon} \ .
\end{equation}  Moreover, for all $t\geq0$, we have that
\begin{equation}\label{eq:ShockwaveStatement2}
 \P\left( \tau_{\mathbf{0}} \geq (1+\varepsilon) c_{b,p} k + t \mid \eta_0 = \theta_k \right) \leq  \P\left( \tau_{\mathbf{0}} \geq  t  \mid \eta_0 = \theta_{\lfloor \log^3 k \rfloor} \right)  +  \tilde{\varepsilon} \ .
\end{equation}
\end{lemma}

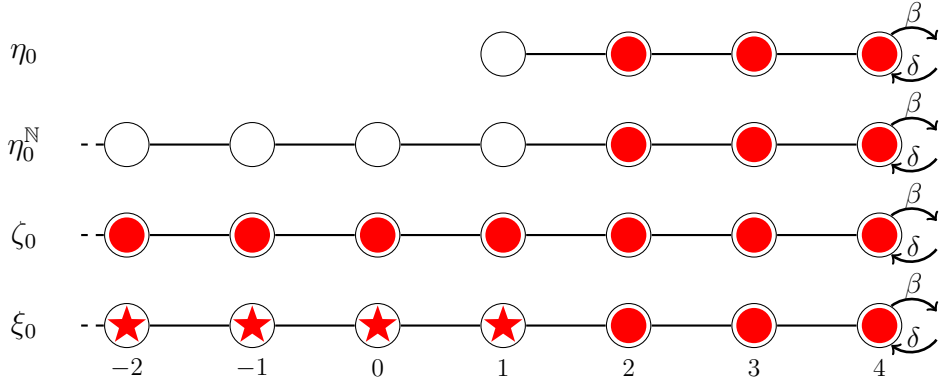
\begin{figure} 
\centering
\begin{tikzpicture}[scale=0.75]

\def\x{2.2};
\def\y{-1.6};

	\def\t{0.9};	

 	\node[shape=circle,scale=1.5,draw] (C3) at (2*\x,-1*\y){} ; 
 	\node[shape=circle,scale=1.5,draw] (C4) at (3*\x,-1*\y){} ; 
 	\node[shape=circle,scale=1.5,draw] (C5) at (4*\x,-1*\y){} ; 
 	\node[shape=circle,scale=1.5,draw] (C6) at (5*\x,-1*\y){} ; 

 	\draw[thick] (C3) -- (C4); 	
 	\draw[thick] (C4) -- (C5); 	
 	\draw[thick] (C5) -- (C6);
 	
	\draw [->,line width=1pt] (C6) to [bend right,in=135,out=45] (5*\x+1,-1*\y+0.25); 	
	\draw [->,line width=1pt] (5*\x+1,-1*\y-0.25) to [bend right,in=135,out=45] (C6); 	
	
	\node[shape=circle,scale=1.2,fill=red] (F4) at (C4) {};	
	\node[shape=circle,scale=1.2,fill=red] (F5) at (C5) {};
	\node[shape=circle,scale=1.2,fill=red] (F6) at (C6) {};

	\node (H3) at (-4,-1*\y) {$\eta_0$};	 
	\node[scale=\t] (J3) at (5*\x+0.6,-1*\y+0.7) {$\beta$};	 		 	
	\node[scale=\t] (K3) at (5*\x+0.6,-1*\y-0.2) {$\delta$};

 	\node[shape=circle,scale=1.5,draw] (C0) at (-\x,0*\y){} ; 	
 	\node[shape=circle,scale=1.5,draw] (C1) at (0,0*\y){} ; 
 	\node[shape=circle,scale=1.5,draw] (C2) at (\x,0*\y){} ; 
 	\node[shape=circle,scale=1.5,draw] (C3) at (2*\x,0*\y){} ; 
 	\node[shape=circle,scale=1.5,draw] (C4) at (3*\x,0*\y){} ; 
 	\node[shape=circle,scale=1.5,draw] (C5) at (4*\x,0*\y){} ; 
 	\node[shape=circle,scale=1.5,draw] (C6) at (5*\x,0*\y){} ;

    \draw[thick,dashed] (C0) -- (-\x-0.8,0*\y);
  	\draw[thick] (C0) -- (C1);	
 	\draw[thick] (C1) -- (C2);		
 	\draw[thick] (C2) -- (C3);
 	\draw[thick] (C3) -- (C4);
 	\draw[thick] (C4) -- (C5);
 	\draw[thick] (C5) -- (C6);

	\draw [->,line width=1pt] (C6) to [bend right,in=135,out=45] (5*\x+1,0*\y+0.25); 	
	\draw [->,line width=1pt] (5*\x+1,0*\y-0.25) to [bend right,in=135,out=45] (C6);

	\node[shape=circle,scale=1.2,fill=red] (G2) at (3*\x,0*\y) {};	
	\node[shape=circle,scale=1.2,fill=red] (G4) at (4*\x,0*\y) {};	
	\node[shape=circle,scale=1.2,fill=red] (G6) at (5*\x,0*\y) {};

	\node (H4) at (-4,0*\y) {$\eta^{\N}_0$};	 

	\node[scale=\t] (J4) at (5*\x+0.6,0*\y+0.7) {$\beta$};	
	\node[scale=\t] (K4) at (5*\x+0.6,0*\y-0.2) {$\delta$};


 	\node[shape=circle,scale=1.5,draw] (C0) at (-\x,1*\y){} ; 	
 	\node[shape=circle,scale=1.5,draw] (C1) at (0,1*\y){} ; 
 	\node[shape=circle,scale=1.5,draw] (C2) at (\x,1*\y){} ; 
 	\node[shape=circle,scale=1.5,draw] (C3) at (2*\x,1*\y){} ; 
 	\node[shape=circle,scale=1.5,draw] (C4) at (3*\x,1*\y){} ; 
 	\node[shape=circle,scale=1.5,draw] (C5) at (4*\x,1*\y){} ; 
 	\node[shape=circle,scale=1.5,draw] (C6) at (5*\x,1*\y){} ; 
 
     \draw[thick,dashed] (C0) -- (-\x-0.8,1*\y);
 	\draw[thick] (C0) -- (C1);	     
 	\draw[thick] (C1) -- (C2);		
 	\draw[thick] (C2) -- (C3);
 	\draw[thick] (C3) -- (C4);
 	\draw[thick] (C4) -- (C5);
 	\draw[thick] (C5) -- (C6);

	\draw [->,line width=1pt] (C6) to [bend right,in=135,out=45] (5*\x+1,1*\y+0.25); 	
	\draw [->,line width=1pt] (5*\x+1,1*\y-0.25) to [bend right,in=135,out=45] (C6); 	

	\node[shape=circle,scale=1.2,fill=red] (G0) at (-1*\x,1*\y) {};		
	\node[shape=circle,scale=1.2,fill=red] (G1) at (0,1*\y) {};	
	\node[shape=circle,scale=1.2,fill=red] (G2) at (1*\x,1*\y) {};	
	\node[shape=circle,scale=1.2,fill=red] (G2) at (2*\x,1*\y) {};	
	\node[shape=circle,scale=1.2,fill=red] (G4) at (3*\x,1*\y) {};	
    \node[shape=circle,scale=1.2,fill=red] (G5) at (4*\x,1*\y) {};	
	\node[shape=circle,scale=1.2,fill=red] (G6) at (5*\x,1*\y) {};

	\node (H4) at (-4,1*\y) {$\zeta_0$};	 
	
	\node[scale=\t] (J4) at (5*\x+0.6,1*\y+0.7) {$\beta$};	
	\node[scale=\t] (K4) at (5*\x+0.6,1*\y-0.2) {$\delta$};

 	\node[shape=circle,scale=1.5,draw] (C0) at (-\x,2*\y){} ; 		
 	\node[shape=circle,scale=1.5,draw] (C1) at (0,2*\y){} ; 
 	\node[shape=circle,scale=1.5,draw] (C2) at (\x,2*\y){} ; 
 	\node[shape=circle,scale=1.5,draw] (C3) at (2*\x,2*\y){} ; 
 	\node[shape=circle,scale=1.5,draw] (C4) at (3*\x,2*\y){} ; 
 	\node[shape=circle,scale=1.5,draw] (C5) at (4*\x,2*\y){} ; 
 	\node[shape=circle,scale=1.5,draw] (C6) at (5*\x,2*\y){} ; 
 
     \draw[thick,dashed] (C0) -- (-\x-0.8,2*\y);
 	\draw[thick] (C0) -- (C1);	 
 	\draw[thick] (C1) -- (C2); 	
 	\draw[thick] (C2) -- (C3);
 	\draw[thick] (C3) -- (C4);
 	\draw[thick] (C4) -- (C5);
 	\draw[thick] (C5) -- (C6);
 	
 	\draw [->,line width=1pt] (C6) to [bend right,in=135,out=45] (5*\x+1,2*\y+0.25); 	
	\draw [->,line width=1pt] (5*\x+1,2*\y-0.25) to [bend right,in=135,out=45] (C6); 	

 	\node[shape=star,star points=5,star point ratio=2.5,fill=red,scale=0.55] (F0) at (-1*\x,2*\y) {};
	\node[shape=star,star points=5,star point ratio=2.5,fill=red,scale=0.55] (F1) at (0*\x,2*\y) {};	
 	\node[shape=star,star points=5,star point ratio=2.5,fill=red,scale=0.55] (F2) at (1*\x,2*\y) {};
	\node[shape=star,star points=5,star point ratio=2.5,fill=red,scale=0.55] (F3) at (2*\x,2*\y) {};
	\node[shape=circle,scale=1.2,fill=red] (F4) at (3*\x,2*\y) {}; 
	\node[shape=circle,scale=1.2,fill=red] (F5) at (4*\x,2*\y) {}; 
	\node[shape=circle,scale=1.2,fill=red] (F6) at (5*\x,2*\y) {};  
 
	\node (H5) at (-4,2*\y) {$\xi_0$};	 
	\node[scale=\t] (J5) at (5*\x+0.6,2*\y+0.7) {$\beta$};	
	\node[scale=\t] (K5) at (5*\x+0.6,2*\y-0.2) {$\delta$};	 

	\def\z{-0.75};			
	\def\s{0.8};	

	\node[scale=\s]  (H3) at (-1*\x,2*\y+\z) {$-2$};	 
	\node[scale=\s]  (H3) at (0*\x,2*\y+\z) {$-1$};	 
	\node[scale=\s]  (H3) at (1*\x,2*\y+\z) {$0$};	 
	\node[scale=\s]  (H3) at (2*\x,2*\y+\z) {$1$};	 
	\node[scale=\s]  (H3) at (3*\x,2*\y+\z) {$2$};	 
	\node[scale=\s]  (H3) at (4*\x,2*\y+\z) {$3$};	 
	\node[scale=\s]  (H3) at (5*\x,2*\y+\z) {$4$};

	\def\particles(#1)(#2){

  \draw[black,thick](-3.9+#1,0.55+#2) -- (-4.9+#1,0.55+#2) -- (-4.9+#1,-0.4+#2) -- (-3.9+#1,-0.4+#2) -- (-3.9+#1,0.55+#2);
  
  	\node[shape=circle,scale=0.6,fill=red] (Y1) at (-4.15+#1,0.2+#2) {};
  	\node[shape=circle,scale=0.6,fill=red] (Y2) at (-4.6+#1,0.35+#2) {};
  	\node[shape=circle,scale=0.6,fill=red] (Y3) at (-4.2+#1,-0.2+#2) {};
   	\node[shape=circle,scale=0.6,fill=red] (Y4) at (-4.45+#1,0.05+#2) {};
  	\node[shape=circle,scale=0.6,fill=red] (Y5) at (-4.65+#1,-0.15+#2) {}; }

%
%
%
%
%
%
%
%

	\end{tikzpicture}	
\caption{\label{fig:ThreeProcessesCoupling}Visualization of the initial configurations of the different processes used in the proof of the upper bound in Theorem \ref{thm:asymmetricOneSide} for $N=4$ and $k=3$.}
 \end{figure}

The proof of Lemma \ref{lem:Shockwave} is deferred to the upcoming Section \ref{sec:ShockwaveLemma}. We conclude this paragraph by showing Theorem \ref{thm:asymmetricOneSide} under the assumption that Lemma \ref{lem:Shockwave} holds.

\begin{proof}[Proof of Theorem \ref{thm:asymmetricOneSide} using Lemma \ref{lem:Shockwave}] Fix $\varepsilon,\tilde{\varepsilon}>0$. For $N$ sufficiently large, we apply Lemma \ref{lem:Shockwave} twice, for $k=N$ with  $t=2\varepsilon c_{b,p}  N$ as well as for $k=\log^3 N \leq ({1+\varepsilon})^{-1}\varepsilon N
$ with $t=\varepsilon c_{b,p}  N$, respectively, to see that 
\begin{align}\label{eq:Shock2x}
 \P\left( \tau_{\mathbf{0}} \geq (1+3\varepsilon)  c_{b,p}  N \mid \eta_0 =\mathbf{1} \right) &\leq   \P\left( \tau_{\mathbf{0}} \geq   2\varepsilon c_{b,p} N \mid \eta_0 = \theta_{\lfloor \log^3N \rfloor} \right) + \tilde{\varepsilon}  \nonumber \\ &\leq  \P\left( \tau_{\mathbf{0}} \geq   \varepsilon c_{b,p} N \mid \eta_0 = \theta_{\lfloor \log^3(\lfloor \log^{3}N\rfloor ) \rfloor} \right) + 2\tilde{\varepsilon} 
\end{align} holds. Using Lemma \ref{lem:AprioriBound} and Markov's inequality, the right-hand side of  \eqref{eq:Shock2x} is bounded by $3\tilde{\varepsilon}$ for all $N$ large enough. Since $\varepsilon$ and $\tilde{\varepsilon}$ were arbitrary, we apply Lemma \ref{lem:HittingTime} to conclude.
\end{proof}

\subsection{Proof of the shock wave phenomenon}\label{sec:ShockwaveLemma}
In order to prove Lemma \ref{lem:Shockwave}, we will now introduce three auxiliary exclusion processes, which will be intertwined by the canonical coupling $\mathbf{P}$. A visualization of their construction can be found in Figure \ref{fig:ThreeProcessesCoupling}. In a first step, we define $(\eta^{\N}_t)_{t\geq0}$ by extending the simple exclusion process $(\eta_t)_{t\geq0}$ on the segment of size $N$ to the half-line $(-\infty, N]$. More precisely, we let $(\eta^{\N}_t)_{t\geq0}$ be the simple exclusion process on the half-line $(-\infty, N] \cap \Z$ with drift $p$, and particles exiting and entering at site $N$ at rates $\beta$ and $\delta$, respectively. On all positive integers, we use the same clocks for the processes $(\eta_t)_{t\geq0}$ and $(\eta^{\N}_t)_{t\geq0}$. In particular, when both processes agree initially on the sites in $[N]$, this construction will ensure that the position $L(\cdot )$ of the leftmost particle 
satisfies $L(\eta^{\N}_t) \leq L(\eta_t)$ almost surely for all $t\geq 0$. In the following, we will assume that  $(\eta_t)_{t \geq0}$ and $(\eta_{t}^{\N})_{t \geq0}$ are started from the configurations, where exactly the rightmost $k$ sites are occupied. \\
Next, we let $(\zeta_t)_{t\geq0}$ be the exclusion process on $(-\infty, N] \cap \Z$ with the same transition rules as $(\eta^{\N}_t)_{t\geq0}$, but started from the all full configuration. Under the canonical coupling $\mathbf{P}$, let $(\xi_t)_{t\geq0}$ denote the disagreement process between $(\eta^{\N}_t)_{t\geq0}$ and  $(\zeta_t)_{t\geq0}$ (recall \eqref{def:DisagreementProcess}). Note that $(\xi_t)_{t\geq0}$ is again an exclusion process on $(-\infty, N ] \cap \Z$ where the rightmost $k$ sites are initially occupied by first class particles, and all other sites by second class 
particles. 
If $\xi_t$ has a positive number of first class particles, let $L_1(\xi_t)$  be the position of its leftmost first class particle. For all $x \in [k]$, let $\tilde{\tau}(x)$ be the first time at which $k-x$ particles have exited in $(\zeta_t)_{t \geq0}$. The next lemma shows that  $L_1(\xi_t)$ is close to the boundary at time $t=\tilde{\tau}(\lceil\log^2 k
\rceil)$. Indeed $\tilde{\tau}(\lceil\log^2 k\rceil)$ will be the stopping time $\tilde{\tau}$ whose existence we claim in Lemma \ref{lem:Shockwave}.

\begin{lemma} \label{lem:MaximumInCensoredDynamics} Let $\tilde{\varepsilon}>0$. Then there exist $k_0 \in \N$ such that for all $k\geq k_0$ and all $N \in \N$, we have that under the above canonical coupling $\mathbf{P}$
\begin{equation}\label{eq:MaximumDynamics}
\mathbf{P}\left( |  L_1(\xi_{\tilde{\tau}(\lceil\log^2 k \rceil)}) - N | \leq \log^3 k \right) \geq 1 - \tilde{\varepsilon} \, .
\end{equation} 
\end{lemma}
\begin{proof} Note that by construction, $(\xi_t)_{t \geq0}$ must contain at least 
$\log^2 k$ first class particles until time $\tilde{\tau}(\lceil\log^2 k \rceil)$, and hence $L_1(\xi_{\tilde{\tau}(\lceil\log^2 k \rceil)})$ is well-defined. In order to show \eqref{eq:MaximumDynamics}, we use similar ideas as Benjamini et al.\ in \cite{BBHM:MixingBias} for the closed segment. Recall \eqref{blockingsets} and define a process $(\xi^{\ast}_t)_{t \geq 0}$ on $A_0$ from $(\xi_t)_{t \geq 0}$ as follows: For every $t\geq 0$, consider the sequence which we obtain by first deleting all sites which are empty in $\xi_t$ (here certain edges are merged) and then replacing all second class particles by empty sites. We let $\xi^{\ast}_t$ be the unique configuration in $A_0$ which contains this sequence, and has only empty sites to the left and only first class particles to its right, see Figure \hyperref[fig:ProjectionBerger]{5}. Note that $\xi^{\ast}_0= \vartheta_0$ holds by construction. \\
We claim that up to the first time $\tau^{\ast}$ at which a second class particle exits at the right-hand side boundary in $(\xi_t)_{t \geq 0}$, the process $(\xi^{\ast}_t)_{t \geq 0}$ has the law of a simple exclusion process on $A_0$ with censoring. More precisely, an edge $e$ is censored in $\xi^{\ast}_t$ at time $t$ if 
\begin{itemize}
\item  one of its endpoints is $> N$ 
\item the edge is merged in the first step of the construction 
from two edges which are adjacent to an empty site, 
\end{itemize}
 see Figure~\hyperref[fig:ProjectionBerger]{5}. To see this,  consider an update in the configuration $\xi_t$ along an edge $e=\{x,x+1\}$. When $\xi_t(x)=\xi_t(x+1)$, or one of the sites in $e$ contains an empty site, this will not change the configuration $\xi^{\ast}_t$. 
Similarly, whenever an empty site or a particle exits in $\xi_t$, this does not change the configuration $\xi^{\ast}_t$. However, each update in $\xi_t$ along the edge $e$ in which we swap a first and a second class particle will in the same way be applied in $\xi^{\ast}_t$, i.e.\ we will swap the particle -- empty site pair along the corresponding edge in $\xi^{\ast}_t$. As these are all possible updates in $\xi_t$, this gives the claim. \\
\begin{figure}
\centering
\begin{tikzpicture}[scale=0.75]

\def\x{2};
\def\y{-1.4};

 	\node[shape=circle,scale=1.5,draw] (C0) at (-\x,0*\y){} ; 
 	\node[shape=circle,scale=1.5,draw] (C1) at (0,0*\y){} ; 
 	\node[shape=circle,scale=1.5,draw] (C2) at (\x,0*\y){} ; 
 	\node[shape=circle,scale=1.5,draw] (C3) at (2*\x,0*\y){} ; 
 	\node[shape=circle,scale=1.5,draw] (C4) at (3*\x,0*\y){} ; 
 	\node[shape=circle,scale=1.5,draw] (C5) at (4*\x,0*\y){} ; 
 	\node[shape=circle,scale=1.5,draw] (C6) at (5*\x,0*\y){} ; 
 	\node[shape=circle,scale=1.5,draw] (C7) at (6*\x,-0*\y){} ; 
 	
 	\draw[thick] (C0) -- (-0.8-\x,0*\y) ; 	
 	\draw[thick] (C0) -- (C1); 
 	\draw[thick] (C1) -- (C2);
 	\draw[thick] (C2) -- (C3);
 	\draw[thick] (C3) -- (C4);
 	\draw[thick] (C4) -- (C5);
 	\draw[thick] (C5) -- (C6);
 	\draw[thick] (C6) -- (C7);
 	
	\node[shape=circle,scale=1.2,fill=red] (F5) at (C2) {};	
	\node[shape=circle,scale=1.2,fill=red] (F5) at (C5) {};
	\node[shape=circle,scale=1.2,fill=red] (F5) at (C7) {};
	
    \node[shape=star,star points=5,star point ratio=2.5,fill=red,scale=0.55] (G0) at (C0) {};		
	\node[shape=star,star points=5,star point ratio=2.5,fill=red,scale=0.55] (G1) at (C1) {};	
	\node[shape=star,star points=5,star point ratio=2.5,fill=red,scale=0.55] (G2) at (C4) {};	
 
	\node (H3) at (-1.6-\x,0*\y) {$\xi$};

 	\node[shape=circle,scale=1.5,draw] (C1) at (0,1*\y){} ; 
 	\node[shape=circle,scale=1.5,draw] (C2) at (\x,1*\y){} ; 
 	\node[shape=circle,scale=1.5,draw] (C3) at (2*\x,1*\y){} ; 
 	\node[shape=circle,scale=1.5,draw] (C4) at (3*\x,1*\y){} ; 
 	\node[shape=circle,scale=1.5,draw] (C5) at (4*\x,1*\y){} ; 
 	\node[shape=circle,scale=1.5,draw] (C6) at (5*\x,1*\y){} ; 
 
 	\draw[thick] (C1) -- (-0.8,1*\y) ; 	
 	\draw[thick] (C1) -- (C2);
 	\draw[thick] (C2) -- (C3);
 	\draw[thick, dashed] (C3) -- (C4);
 	\draw[thick] (C4) -- (C5);
 	\draw[thick, dashed] (C5) -- (C6);
	
	\node[shape=circle,scale=1.2,fill=red] (F5) at (C3) {};	
	\node[shape=circle,scale=1.2,fill=red] (F5) at (C5) {};
	\node[shape=circle,scale=1.2,fill=red] (F5) at (C6) {};

	\node[shape=star,star points=5,star point ratio=2.5,fill=red,scale=0.55] (G0) at (C1) {};			
	\node[shape=star,star points=5,star point ratio=2.5,fill=red,scale=0.55] (G1) at (C2) {};	
	\node[shape=star,star points=5,star point ratio=2.5,fill=red,scale=0.55] (G2) at (C4) {};


 	\node[shape=circle,scale=1.5,draw] (C1) at (0,2*\y){} ; 
 	\node[shape=circle,scale=1.5,draw] (C2) at (\x,2*\y){} ; 
 	\node[shape=circle,scale=1.5,draw] (C3) at (2*\x,2*\y){} ; 
 	\node[shape=circle,scale=1.5,draw] (C4) at (3*\x,2*\y){} ; 
 	\node[shape=circle,scale=1.5,draw] (C5) at (4*\x,2*\y){} ; 
 	\node[shape=circle,scale=1.5,draw] (C6) at (5*\x,2*\y){} ; 
 
 	\draw[thick] (C1) -- (-0.8,2*\y) ; 	
 	\draw[thick] (C1) -- (C2);
 	\draw[thick] (C2) -- (C3);
 	\draw[thick, dashed] (C3) -- (C4);
 	\draw[thick] (C4) -- (C5);
 	\draw[thick, dashed] (C5) -- (C6);
	
	\node[shape=circle,scale=1.2,fill=red] (F5) at (C3) {};	
	\node[shape=circle,scale=1.2,fill=red] (F5) at (C5) {};
	\node[shape=circle,scale=1.2,fill=red] (F5) at (C6) {};

 	\node[shape=circle,scale=1.5,draw] (C1) at (0,3*\y){} ; 
 	\node[shape=circle,scale=1.5,draw] (C2) at (\x,3*\y){} ; 
 	\node[shape=circle,scale=1.5,draw] (C3) at (2*\x,3*\y){} ; 
 	\node[shape=circle,scale=1.5,draw] (C4) at (3*\x,3*\y){} ; 
 	\node[shape=circle,scale=1.5,draw] (C5) at (4*\x,3*\y){} ; 
 	\node[shape=circle,scale=1.5,draw] (C6) at (5*\x,3*\y){} ; 
 	\node[shape=circle,scale=1.5,draw] (C7) at (6*\x,3*\y){} ;  
 	\node[shape=circle,scale=1.5,draw] (C8) at (7*\x,3*\y){} ; 
 
 	\draw[thick] (C1) -- (-0.8,3*\y) ; 	
 	\draw[thick] (C1) -- (C2); 	
 	\draw[thick] (C2) -- (C3);
 	\draw[thick, dashed] (C3) -- (C4);
 	\draw[thick] (C4) -- (C5);
 	\draw[thick, dashed] (C5) -- (C6);
 	\draw[thick, dashed] (C6) -- (C7);
 	\draw[thick, dashed] (C7) -- (C8);
 	\draw[thick, dashed] (C8) -- (7*\x+0.8,3*\y);

	\node[shape=circle,scale=1.2,fill=red] (F5) at (C3) {};	
	\node[shape=circle,scale=1.2,fill=red] (F5) at (C5) {};
	\node[shape=circle,scale=1.2,fill=red] (F5) at (C6) {};
	\node[shape=circle,scale=1.2,fill=red] (F5) at (C7) {};
	\node[shape=circle,scale=1.2,fill=red] (F5) at (C8) {};
 
	\node (H3) at (-1.8-\x,3*\y) {$\xi^{\ast}$};	 

	\node (H3) at (-1.8-\x,2*\y) {Step 2};	 
	\node (H3) at (-1.8-\x,1*\y) {Step 1};	 
	
\def\z{-0.75};	
\def\s{0.8};	

	\node[scale=\s]  (H3) at (0*\x,3*\y+\z) {$-2$};	 
	\node[scale=\s]  (H3) at (1*\x,3*\y+\z) {$-1$};	 
	\node[scale=\s]  (H3) at (2*\x,3*\y+\z) {$0$};	 
	\node[scale=\s]  (H3) at (3*\x,3*\y+\z) {$1$};	 
	\node[scale=\s]  (H3) at (4*\x,3*\y+\z) {$2$};	 
	\node[scale=\s]  (H3) at (5*\x,3*\y+\z) {$3$};	 
	\node[scale=\s]  (H3) at (6*\x,3*\y+\z) {$4$};	 
	\node[scale=\s]  (H3) at (7*\x,3*\y+\z) {$5$};

	\end{tikzpicture}
\caption{\label{fig:ProjectionBerger}Construction of $\xi^{\ast} \in A_0$ from $\xi$. All censored edges are drawn dashed.}
 \end{figure}
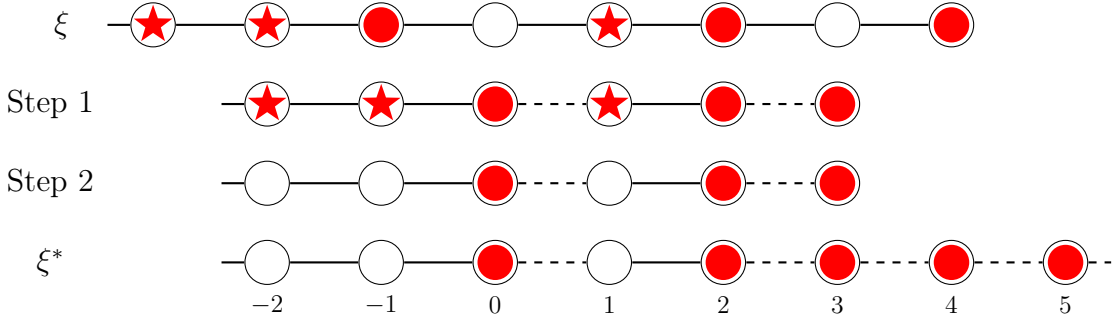Equipped with the process $(\xi_t^{\ast})_{t \geq 0}$, we will now show Lemma \ref{lem:MaximumInCensoredDynamics}  using four families of events $\{B^{i}_k\}_{i \in [4], k \in \N}$. These events will allow us to control
\begin{itemize}
\item the total number of particles which exited in $(\xi_t)_{t \geq 0}$ (events $B^{1}_k$),
\item the number of second class particle which exited in $(\xi_t)_{t \geq 0}$ (events $B^{2}_k$),
\item the interaction between first and second class particles in $(\xi_t)_{t \geq 0}$ (events $B^{3}_k$),
\item the density of empty sites in $(\xi_t)_{t \geq 0}$ (events $B^{4}_k$).
\end{itemize} 
Recall the constant $c_{b,p}$ from \eqref{def:ConstantCBP}. For all $k\in \N$ and $\varepsilon>0$, we define 
\begin{align}\label{def:B12k}
B^{1}_k :=& \left\{  \tilde{\tau}(\lceil\log^2 k \rceil)  \leq (1+\varepsilon)c_{b,p}k \right\} \nonumber\\
B^{2}_k :=& \left\{  \tilde{\tau}(\lceil\log^2 k \rceil)  \leq \tau^{\ast} \right\} 
\end{align} to be the events that there exists some time before $(1+\varepsilon)c_{b,p}k$ at which at least $k -\lceil\log^2 k \rceil$ particles in $(\zeta_t)_{t\geq 0}$ exited, and that no second class particle exited in $(\xi_t)_{t \geq 0}$ before that time, respectively. Observe that the total number of particles which have left in $(\zeta_t)_{t\geq0}$ at time $t$ has the same law as the current of a simple exclusion process on the half-line at time $t$ with 
boundary parameters $\tilde{\alpha}=\beta$ and $\tilde{\gamma}=\delta$ (recall \eqref{def:generatorHalfLine}). Hence, Lemma \ref{lem:CurrentHalfLine} implies that $ \mathbf{P}(B^{1}_k)=1-\tilde{\varepsilon}/4$ for all $k$ sufficiently large.  Moreover,  when
\begin{equation*}\label{eq:KeyEquationXiStar}
\left\{  \tau^{\ast} \leq  \tilde{\tau}(\lceil\log^2 k \rceil) \leq   (1+\varepsilon)c_{b,p}k \right\} 
\end{equation*} holds, there must be an empty site in $(\xi^{\ast}_t)_{t \geq 0}$ at position $\log^2 k $ until time $\tau^{\ast}$. Note that the censoring scheme for the process $(\xi^{\ast}_t)_{t \geq 0}$ does not depend on the motion of the particles in $(\xi^{\ast}_t)_{t \geq 0}$, since in order to determine the positions of the empty sites in $(\xi_t)_{t \geq 0}$, we do not need to distinguish between first and second class particles.  Thus, we can conclude that $\mathbf{P}(B^{2}_k \mid B^{1}_k)\geq1-\tilde{\varepsilon}/4 $ by combining Lemma \ref{lem:HittingBlockingMeasure} and Remark \ref{rem:CensoringBlocking}, using that the process $(\xi^{\ast}_t)_{t \geq 0}$ can be seen as a simple exclusion process with censoring. In particular, this yields $\mathbf{P}(B^{2}_k) \geq 1 - \tilde{\varepsilon}/2$. Recall that $R(\cdot)$ denotes the position of the rightmost empty site. Again, by Lemma  \ref{lem:HittingBlockingMeasure}, we see that the events
\begin{equation*}
B^{3}_k := \left\{ |L(\xi^{\ast}_t)-R(\xi^{\ast}_t)| \leq \log^2 k \text{ for } t=\tilde{\tau}_{\lceil\log^2 k \rceil}\right\} 
\end{equation*}  satisfy $\mathbf{P}(B^{3}_k \mid B^{1}_k \cap B^{2}_k)\geq 1-\tilde{\varepsilon}/4$ for $k$ sufficiently large. 
Note that whenever the events $B^{1}_k,B^{2}_k$ and $B^{3}_k$ occur, a sufficient condition for the statement in Lemma \ref{lem:MaximumInCensoredDynamics} to hold is that the event 
\begin{equation*}
B^{4}_k := \left\{  |  x \in  [N-\log^3 k, N] \colon \xi_t(x) \neq 0 | \geq 2\log^2 k \text{ for all } t \in [0,(1+\varepsilon)c_{b,p}k]  \right\} 
\end{equation*} occurs. Using the particle -- empty site symmetry, we see by \eqref{eq:ComparisonHalfLineProduct}  in Lemma \ref{lem:CurrentHalfLine} that the law of $\zeta_t$ dominates a Bernoulli-$\frac{b}{b+1}$-product measure for all $t \in [0,(1+\varepsilon)c_{b,p}k]$, and hence, we can conclude that $\mathbf{P}(B^{4}_k \mid B^{1}_k\cap B^{2}_k\cap B^{3}_k) \geq 1- \tilde{\varepsilon}/4$ for all $k$ large enough.
\end{proof}

\begin{proof}[Proof of Lemma \ref{lem:Shockwave}] Set $\tilde{\tau} = \tilde{\tau}(\lceil\log^2 k\rceil)$ and recall  that $L(\eta^{\N}_t) \leq L(\eta_t)$ holds almost surely for all $t\geq 0$. By Lemma \ref{lem:CurrentHalfLine}, we get that $\tilde{\tau} \leq (1+\varepsilon) c_{b,p} k$ holds with probability tending to $1$ when $k \rightarrow \infty$. The first statement  \eqref{eq:ShockwaveStatement1} is now immediate from Lemma \ref{lem:MaximumInCensoredDynamics}.  For the second statement \eqref{eq:ShockwaveStatement2}, we apply the strong Markov property for $(\eta_t)_{t\geq 0}$ with respect to the stopping time $\tilde{\tau}(\lceil\log^2 k \rceil)$. Note that by adding additional particles to the process $(\eta_t)_{t\geq 0}$ at some time $t \leq \tau_{\mathbf{0}}$, we will only increase the hitting time $\tau_{\mathbf{0}}$ of the state $\mathbf{0}$. Hence, whenever the event in \eqref{eq:ShockwaveStatement1} holds with respect to $\tilde{\tau}(\lceil\log^2 k \rceil)$, we see that the hitting time of $\mathbf{0}$ starting from $\eta_{\tilde{\tau}(\lceil\log^2 k \rceil)}$ is stochastically dominated by the hitting time when starting from  
$\theta_{\lceil\log^3 k
\rceil}$. This yields \eqref{eq:ShockwaveStatement2}. 
\end{proof}

\section{Mixing times for the reverse bias phase}\label{sec:ReverseBias}

In this section, we prove upper and lower bounds on the mixing time of the simple exclusion process in the reverse bias phase. Recall that $\frac{1}{2}<p<1$ and $\alpha=\beta=0$ holds, i.e.\ the particles have a drift to the right-hand side, but can neither exit at the right-hand side nor enter at the left-hand side boundary. Intuitively, the particles have to move against their natural drift direction. We will see that this  results in an exponentially large mixing time. For the lower bound, we consider two exclusion processes with different initial states and show that with high probability, they have a disjoint support even at exponentially large times. For the upper bound, we compare the disagreement process with respect to initial states $\mathbf{0}$ and $\mathbf{1}$ to a birth-and-death chain.
\subsection{Lower bounds for the reverse bias phase}\label{sec:ReverseLowerBound}

We start with the lower bound when $\min(\gamma,\delta)>0$ holds. Recall the total-variation distance from \eqref{def:TVDistance} and note that by the triangle inequality
\begin{equation}\label{eq:MixingTimeViaConfigurations}
\max_{\zeta \in \{\theta,\theta^{\prime}\}}\TV{P_{\zeta}(\eta_t \in \cdot) - \mu} \geq \frac{1}{2}\TV{P_{\theta}(\eta_t \in \cdot) - P_{\theta^{\prime}}(\eta_t \in \cdot)}
\end{equation} holds for any two initial states $\theta,\theta^{\prime} \in \Omega_N$ of the simple exclusion process $(\eta_t)_{t \geq 0}$ with open boundaries. We define
\begin{equation}
\theta(x):=\mathds{1}_{\{x\geq \left\lfloor\frac{N}{2}\right\rfloor\}}\  \text{ and } \ \theta^{\prime}(x):=\mathds{1}_{\{x\geq \left\lfloor\frac{N}{2}\right\rfloor+1\}} \ \text{ for all } x \in [N].
\end{equation} Note that the total-variation distance of two distributions is $1$ if they have disjoint support. Hence, we see that the right-hand side of \eqref{eq:MixingTimeViaConfigurations} is bounded from below by $1$ minus the probability that at least one particle enters or exits in at least one of the exclusion processes started from $\theta$ and $\theta^{\prime}$. We estimate this probability by comparing the simple exclusion processes started from $\theta$ and $\theta^{\prime}$, respectively, to the simple exclusion processes on $\mathbb{Z}$ via the embedding 
\begin{equation}
\tilde{\eta}(x) := \begin{cases} \eta(x) & \text{ if } x \in [N]
 \\
 0 & \text{ if } x \leq 0 \\
 1 & \text{ if } x > N \end{cases}
\end{equation} for all $x \in \mathbb{Z}$ and all configurations $\eta \in \Omega_N$. In particular, note that $\tilde{\theta}$ and $\tilde{\theta}^{\prime}$ are the ground states in $A_{n}$ and $A_{n+1}$ for $n=\lfloor N/2\rfloor$, respectively. Moreover, 
using the usual extension of the canonical coupling on the segment to the integers via the censoring inequality, we obtain that the simple exclusion processes started from $\tilde{\theta}$ and $\tilde{\theta}^{\prime}$ are stochastically dominated by the respective exclusion processes started from the blocking measures on $A_{n}$ and $A_{n+1}$, see Remark \ref{rem:PartialOrderonZ}. Thus, we obtain the lower bound in \eqref{eq:LogCutoffTwoSides} of Theorem \ref{thm:asymmetricOneSideReverse} by applying Lemma \ref{lem:HittingBlockingMeasure} for the simple exclusion process on $\Z$ with $x=\lfloor N/2 \rfloor-1$ and $\varepsilon=N^{-1}$. \\

In the case where particles can exit only from one side of the segment, a similar argument holds. More precisely, using the particle -- empty site symmetry, it suffices to consider $\gamma>0$ and $\alpha,\beta,\delta=0$. The stationary distribution $\mu$ is then the Dirac measure on $\mathbf{0}$. Consider the initial state $\zeta$ with $\zeta(x)=\mathds{1}_{\{x = N\}}$ for all $x \in [N]$ and note that $\tilde{\zeta}$ is the ground state on $A_{N-1}$. Similarly to the previous case, we obtain the lower bound in \eqref{eq:LogCutoffOneSide} by applying Lemma \ref{lem:HittingBlockingMeasure} for the simple exclusion process on $\Z$ with $x=N-2$ and $\varepsilon=N^{-1}$. This concludes the proof of the lower bounds in Theorem \ref{thm:asymmetricOneSideReverse}.

\subsection{Upper bounds for the reverse bias phase}\label{sec:ReverseBiasTwoSides}

We now show the upper bounds in Theorem \ref{thm:asymmetricOneSideReverse}.
By Corollary \ref{cor:SecondClassParticles}, it suffices to consider the disagreement process for states $\mathbf{1}$ and $\mathbf{0}$, and study the time it takes until all second class particles have left the segment. 
In the following, we enumerate the second class particles from left to right, and let $(X_t^{(i)})_{t \geq 0}$ for $i \in [N]$ denote the trajectory of the $i^{\text{th}}$ second class particle in the disagreement process. Moreover,  denote its exit time of the segment by $\tau^{\textup{ex}}_{i}$. In order to bound these exit times, we compare $(X_t^{(i)})_{t \geq 0}$ to a certain continuous-time birth-and-death chain $(B_t)_{t \geq 0}$ with state space $[n]$ for some $n \in \N$ which will be determined later on. Similar to  \eqref{def:returnTime}, we let for all $j \in [n]$
\begin{equation}
\tau^{+}_{j} := \inf\left\{ t \geq \tau_{[n] \setminus \{j\}} \colon B_t=j \right\}
\end{equation} be the first \textbf{return time} of $(B_t)_{t \geq 0}$ to the state $j$.
\begin{lemma} \label{lem:GamblersRuin} Consider a birth-and-death chain $(B_t)_{t \geq 0}$ on $[n]$ for some $n \in \N$ with  transition rates $1-p$ to the right and $p$ to the left (and $1-p$ to the right at $1$, $p$ to the left at $N$). Then we have that the return time $\tau^{+}_{n}$ to the site $n$ satisfies
\begin{equation}\label{eq:ExpectedBirthDeath}
\E_{k}\left[ \tau^{+}_{n} \right]  \leq  \frac{1}{Z}\left(\frac{p}{1-p}\right)^{n} 
\end{equation} for any initial state $k\in [n]$, with a constant $Z>0$.
\end{lemma} 
\begin{proof} Observe that the stationary distribution $\mu^{\prime}$ of the birth-and-death chain satisfies $\mu^{\prime}(x)= \frac{1}{Z^{\prime}}\left(\frac{1-p}{p}\right)^{x}$ for all $x \in [n]$, with a  normalization constant $Z^{\prime}>0$. Moreover,
\begin{equation*}
\E_{k}\left[ \tau^{+}_{n} \right] \leq \P_{n}\left(\tau^{+}_{k} < \tau^{+}_{n} \right)^{-1} \E_{n}\left[ \tau^{+}_{n} \right] =  \P_{n}\left( \tau^{+}_{k} < \tau^{+}_{n} \right)^{-1}Z^{\prime}\left(\frac{p}{1-p}\right)^{n}
\end{equation*} holds for all $k \in [n-1]$. Observe that $\P_{n}\left( \tau^{+}_{k} < \tau^{+}_{n} \right)$ is bounded from below by some $c>0$ uniformly in $k$ and $n$. We obtain \eqref{eq:ExpectedBirthDeath} for $1/Z=\sup_{n \in \N}c^{-1}Z^{\prime}$.
\end{proof}

We start with the case where particles can enter only at one side of the segment. Without loss of generality, assume that $\delta>0$ and $\gamma=0$ holds. The stationary distribution $\mu$ is then the Dirac measure on the configuration $\mathbf{1}$. Observe that each second class particle moves to the right at least at rate $1-p$,  and to the left at most at rate $p$ independently of the remaining particle configuration. Thus, we have that until time $\tau^{\textup{ex}}_i$, $(X_t^{(i)})_{t \geq 0}$ stochastically dominates the process $(B_t)_{t \geq 0}$ from  Lemma \ref{lem:GamblersRuin} for $n=N$ and started from $B_0=X_0^{(i)}$, i.e.\ we find a coupling such that $X^{(i)}_t \geq B_t$ holds almost surely for all $t<\tau^{\textup{ex}}_i$. Moreover, when a second class particle reaches site $N$ at time $t$, with probability at least $\frac{1-e^{-\delta}}{e}$,  it has exited the segment by time $t+1$. Thus, with respect to the canonical coupling, we conclude that there exists some constant $c>0$ such that
\begin{equation*}
\mathbf{E}\left[ \tau^{\textup{ex}}_i \right] \leq  c \left(\frac{p}{1-p}\right)^{N} \ \text{ for all } i \in [N].
\end{equation*} Moreover, by Markov's inequality, we see that
\begin{equation*}
\mathbf{P}\left( \tau^{\textup{ex}}_i >  cN^2 \left(\frac{p}{1-p}\right)^{N}  \right) \leq \frac{1}{N^2} \ \text{ for all } i \in [N].
\end{equation*} We conclude the upper bound in \eqref{eq:LogCutoffOneSide} of Theorem \ref{thm:asymmetricOneSideReverse} using a union bound for the event that some second class particle has not left the segment by time $cN^2 (p/(1-p))^N$. \\

Now suppose that $\min(\gamma,\delta)>0$ holds. Note that each second class particle has a distance of at most $\lfloor N/2 \rfloor$ to either the site $1$ or the site $N$. Consider the family of processes $(Y^{(i)}_t)_{t \geq 0}$ given by
\begin{equation}
Y^{(i)}_t := \max(X^{(i)}_t-\lceil N/2 \rceil,\lceil N/2 \rceil+1-X^{(i)}_t) 
\end{equation} for all $t\geq 0$ and $i\in [N]$. Note that $(Y^{(i)}_t)_{t \geq 0}$ increases by $1$ at most at rate $p$ and decreases by $1$ at least at rate $1-p$. For all $i \in [N]$, $(Y^{(i)}_t)_{t \geq 0}$ is stochastically dominated by the birth-and-death process in Lemma \ref{lem:GamblersRuin} for $n=\lfloor N/2 \rfloor$ and $B_0=Y^{(i)}_0$. A similar argument as for the one-sided case finishes the proof of Theorem \ref{thm:asymmetricOneSideReverse}.

\section{Mixing times in the high and low density phase} \label{sec:HighLowDensity}

In this section, we prove Theorem \ref{thm:asymmetricTwoSides} for the asymmetric simple exclusion process in the high density and low density phase. We will focus on showing an upper bound of order $N$. The lower bound of order $N$ follows from a comparison to a single particle dynamics using the fact that the invariant measure has a positive density in the bulk, see also Section \ref{sec:AsymmetricLowerBound}. To see this intuitively, fix some $\varepsilon>0$ sufficiently small, and note that with probability tending $1$ as $N \rightarrow \infty$, the segment $[(1-2\varepsilon)N,(1-\varepsilon)N]$ contains at least one particle in the stationary distribution. However, starting from the all empty configuration, we see that with probability tending $1$ as $N \rightarrow \infty$, no particle has reached the segment by time $(1-3\varepsilon)(2p-1)^{-1}N$. Since $\varepsilon>0$ was arbitrary, we conclude.
In the following, we will only consider the high density phase. For the low density phase, similar arguments apply using the particle -- empty site symmetry. 

\subsection{Construction of two disagreement processes}

We assume that we are in the high density phase of the simple exclusion process with parameters $(p,\alpha,\beta,\gamma,\delta)$, i.e.\ we have that $a=a(p,\alpha,\gamma) $ and $b=b(p,\beta,\delta)$ defined in \eqref{def:a} and \eqref{def:b} satisfy $b>\max(a,1)$. We have the following strategy to show the upper bound \eqref{eq:MixingTimeHighDensity} in Theorem \ref{thm:asymmetricTwoSides}. For $j \in [4]$, we study simple exclusion processes $(\eta^{j}_t)_{t \geq 0}$ with open boundaries within the canonical coupling $\mathbf{P}$. The processes $(\eta^{1}_t)_{t \geq 0}$, $(\eta^{2}_t)_{t \geq 0}$ and $(\eta^{3}_t)_{t \geq 0}$ are defined with respect to the parameters $(p,\alpha,\beta,\gamma,\delta)$. They are started at states $\mathbf{1}$, $\mathbf{0}$ and from the stationary distribution $\mu$, respectively. \\
In order to define $(\eta^{4}_t)_{t \geq 0}$, note that $b$ is decreasing and continuous in $\beta$. Thus, we can choose some $\beta^{\prime}>\beta$ such that $b>b^{\prime}> \max(a,1)$ holds for $b^{\prime}:=b(p,\beta^{\prime},\delta)$. We let $(\eta^{4}_t)_{t \geq 0}$ be the simple exclusion process with open boundaries for parameters $(p,\alpha,\beta^{\prime},\gamma,\delta)$ started from its equilibrium. Using Lemma \ref{lem:MonotoneCouplingComponentwise}, note that we can choose the initial configurations in $(\eta^{3}_t)_{t \geq 0}$ and $(\eta^{4}_t)_{t \geq 0}$ such that
\begin{equation}\label{eq:ConstructionZeta}
\mathbf{P}\left( \eta^{3}_t  \succeq_\c  \eta^{4}_t  \ \text{ for all } t \geq 0\right) = 1 \, .
\end{equation} We define $(\xi_t)_{t \geq 0}$ to be the disagreement process between $(\eta^{1}_t)_{t \geq 0}$ and $(\eta^{2}_t)_{t \geq 0}$. Further, we let $(\zeta_t)_{t \geq 0}$ be the disagreement process between $(\eta^{3}_t)_{t \geq 0}$ and $(\eta^{4}_t)_{t \geq 0}$.
Since all simple exclusion processes are within the canonical coupling, note that $(\xi_t)_{t \geq 0}$ and $(\zeta_t)_{t \geq 0}$ can be seen as Markov processes on $\{0,1,2\}^N$, and $(\zeta_t)_{t \geq 0}$ is started from equilibrium. Further, observe that in $(\xi_t)_{t \geq 0}$, no second class particles can enter the segment. In $(\zeta_t)_{t \geq 0}$, second class particles can enter only at site $N$ provided that $N$ is occupied by a first class particle. In Lemma \ref{lem:SecondClassExits}, we will see that if enough second class particles have exited in $(\zeta_t)_{t \geq 0}$, then $(\xi_t)_{t \geq 0}$ has no second class particles with probability tending to $1$. 

For $i \in \{0,1,2\}$, let $(J^{(i)}_t)_{t \geq 0}$ denote the current of objects of type $i$, i.e.\ for a given time $t\geq 0$, $J^{(i)}_t$ denotes the number of objects of type $i$ which have entered by time $t$ minus the number of objects of type $i$ which have exited by time $t$ at the left-hand side boundary in $(\zeta_t)_{t \geq 0}$, see also \eqref{def:current}. The following lemma shows that the current of second class particles in $(\zeta_t)_{t \geq 0}$ is linear when starting from its equilibrium $\mu^{\prime}$.

\begin{lemma} \label{lem:SecondClassCurrent} Let $(\zeta_t)_{t \geq 0}$ have initial distribution $\mu^{\prime}$. There exists some $c=c(b,b^{\prime},p)>0$ such that for all $t=t(N) \geq cN$, we have that
\begin{equation}\label{eq:EstimateSecondClassParticles}
\lim_{N \rightarrow \infty}\mathbf{P}\left(  - J^{(2)}_{t(N)} > 4N \right) = 1 \ . 
\end{equation}
\end{lemma}
\begin{proof} Let $(\zeta^{2\rightarrow 1}_t)_{t \geq 0}$ and $(\zeta^{2\rightarrow 0}_t)_{t \geq 0}$ denote the processes which we obtain from $(\zeta_t)_{t \geq 0}$ by projecting all second class particles to first class particles and empty sites, respectively. Noting that $\zeta^{2\rightarrow 1}_t=\eta_t^3$ and $\zeta^{2\rightarrow 0}_t=\eta_t^4$ for all $t \geq 0$, we see that $(\zeta^{2\rightarrow 1}_t)_{t \geq 0}$ and $(\zeta^{2\rightarrow 0}_t)_{t \geq 0}$ are stationary simple exclusion processes with  parameters $(p, \alpha,\beta, \gamma, \delta)$ and $(p, \alpha,\beta^{\prime}, \gamma, \delta)$, respectively.  Observe that $(J^{(1)}_t+J^{(2)}_t)_{t \geq 0}$ is given by the current of particles in $(\zeta^{2\rightarrow 1}_t)_{t \geq 0}$ and $(J^{(0)}_t+J^{(2)}_t)_{t \geq 0}$ is given by the current of empty sites in $(\zeta^{2\rightarrow 0}_t)_{t \geq 0}$. Since $(\zeta^{2\rightarrow 1}_t)_{t \geq 0}$ and $(\zeta^{2\rightarrow 0}_t)_{t \geq 0}$ are stationary, we get by Lemma  \ref{lem:current} that there exists some $N_0=N_0(p,b,b^{\prime})$ such that for all $N\geq N_0$ and $t>0$
\begin{equation*}
t^{-1}\E[J^{(1)}_t+J^{(2)}_t] + t^{-1}\E[J^{(0)}_t+J^{(2)}_t] < \frac{1}{2}(2p-1)\left( \frac{b}{(1+b)^2}- \frac{b^{\prime}}{(1+b^{\prime})^2}\right) <0 \, . 
\end{equation*}
Then using the ergodic theorem for the current in $(\zeta^{2\rightarrow 1}_t)_{t \geq 0}$ and $(\zeta^{2\rightarrow 0}_t)_{t \geq 0}$, we get
\begin{equation*}
\lim_{N \rightarrow \infty}\mathbf{P}\left( J^{(0)}_{cN}  + J^{(1)}_{cN} +2 J^{(2)}_{cN}< - 4N \right) = 1
\end{equation*} for some constant $c>0$ which does not depend on $N$. Since by construction 
\begin{equation*}
 J^{(0)}_{t}  + J^{(1)}_{t} + J^{(2)}_{t} = 0 \ \text{ for all } t\geq 0 \, ,
\end{equation*} and $(J_t^{(2)})_{t \geq 0}$ is decreasing in $t$, we conclude. 
\end{proof}

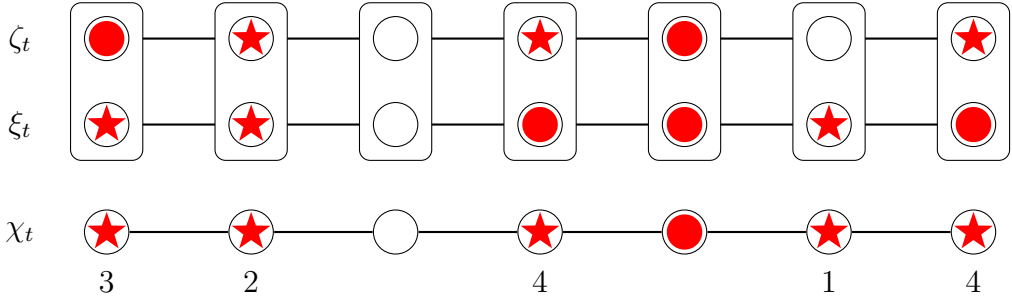
\begin{figure}
\centering
\begin{tikzpicture}[scale=0.95]

\def\x{2};
\def\y{1.2};
\def\z{-1.5};

 	\node[shape=circle,scale=1.5,draw] (A1) at (0,0){} ; 
 	\node[shape=circle,scale=1.5,draw] (A2) at (\x,0){} ; 
 	\node[shape=circle,scale=1.5,draw] (A3) at (2*\x,0){} ; 
 	\node[shape=circle,scale=1.5,draw] (A4) at (3*\x,0){} ; 
 	\node[shape=circle,scale=1.5,draw] (A5) at (4*\x,0){} ; 
 	\node[shape=circle,scale=1.5,draw] (A6) at (5*\x,0){} ; 
 	\node[shape=circle,scale=1.5,draw] (A7) at (6*\x,0){} ; 

 	\node[shape=circle,scale=1.5,draw] (B1) at (0,\y){} ; 
 	\node[shape=circle,scale=1.5,draw] (B2) at (\x,\y){} ; 
 	\node[shape=circle,scale=1.5,draw] (B3) at (2*\x,\y){} ; 
 	\node[shape=circle,scale=1.5,draw] (B4) at (3*\x,\y){} ; 
 	\node[shape=circle,scale=1.5,draw] (B5) at (4*\x,\y){} ; 
 	\node[shape=circle,scale=1.5,draw] (B6) at (5*\x,\y){} ; 
 	\node[shape=circle,scale=1.5,draw] (B7) at (6*\x,\y){} ; 

	\draw[rounded corners] (-0.5+0*\x,-0.5) rectangle (0.5+0*\x,\y+0.5);
	\draw[rounded corners] (-0.5+1*\x,-0.5) rectangle (0.5+1*\x,\y+0.5);
	\draw[rounded corners] (-0.5+2*\x,-0.5) rectangle (0.5+2*\x,\y+0.5);
	\draw[rounded corners] (-0.5+3*\x,-0.5) rectangle (0.5+3*\x,\y+0.5);
	\draw[rounded corners] (-0.5+4*\x,-0.5) rectangle (0.5+4*\x,\y+0.5);
	\draw[rounded corners] (-0.5+5*\x,-0.5) rectangle (0.5+5*\x,\y+0.5);
	\draw[rounded corners] (-0.5+6*\x,-0.5) rectangle (0.5+6*\x,\y+0.5); 	
 	
	\node[shape=circle,scale=1.2,fill=red] (D1) at (B1) {};
	\node[shape=star,star points=5,star point ratio=2.5,fill=red,scale=0.55] (D2) at (B2) {};
	\node[shape=star,star points=5,star point ratio=2.5,fill=red,scale=0.55] (D4) at (B4) {};
	\node[shape=circle,scale=1.2,fill=red] (D5) at (B5) {};
	\node[shape=star,star points=5,star point ratio=2.5,fill=red,scale=0.55] (D7) at (B7) {};
	
	\node[shape=star,star points=5,star point ratio=2.5,fill=red,scale=0.55] (E1) at (A1) {};
	\node[shape=star,star points=5,star point ratio=2.5,fill=red,scale=0.55] (E2) at (A2) {};	
	\node[shape=circle,scale=1.2,fill=red] (E4) at (A4) {};	
	\node[shape=circle,scale=1.2,fill=red] (E5) at (A5) {};		
	\node[shape=star,star points=5,star point ratio=2.5,fill=red,scale=0.55] (E6) at (A6) {};	
	\node[shape=circle,scale=1.2,fill=red] (E7) at (A7) {};		

 	\node[shape=circle,scale=1.5,draw] (C1) at (0,\z){} ; 
 	\node[shape=circle,scale=1.5,draw] (C2) at (\x,\z){} ; 
 	\node[shape=circle,scale=1.5,draw] (C3) at (2*\x,\z){} ; 
 	\node[shape=circle,scale=1.5,draw] (C4) at (3*\x,\z){} ; 
 	\node[shape=circle,scale=1.5,draw] (C5) at (4*\x,\z){} ; 
 	\node[shape=circle,scale=1.5,draw] (C6) at (5*\x,\z){} ; 
 	\node[shape=circle,scale=1.5,draw] (C7) at (6*\x,\z){} ; 
 	
 	\draw[thick] (C1) -- (C2);
 	\draw[thick] (C2) -- (C3);
 	\draw[thick] (C3) -- (C4);
 	\draw[thick] (C4) -- (C5);
 	\draw[thick] (C5) -- (C6);
 	\draw[thick] (C6) -- (C7);

 	\node (F1) at (0,\z-0.7) {$3$};
 	\node (F2) at (\x,\z-0.7) {$2$};
 	\node (F4) at (3*\x,\z-0.7) {$4$};
 	\node (F6) at (5*\x,\z-0.7) {$1$};
 	\node (F7) at (6*\x,\z-0.7) {$4$};
	\node[shape=circle,scale=1.2,fill=red] (F5) at (C5) {};
	
	\node[shape=star,star points=5,star point ratio=2.5,fill=red,scale=0.55] (G1) at (0,\z) {};	
	\node[shape=star,star points=5,star point ratio=2.5,fill=red,scale=0.55] (G2) at (\x,\z) {};	
	\node[shape=star,star points=5,star point ratio=2.5,fill=red,scale=0.55] (G4) at (3*\x,\z) {};	
	\node[shape=star,star points=5,star point ratio=2.5,fill=red,scale=0.55] (G6) at (5*\x,\z) {};	
	\node[shape=star,star points=5,star point ratio=2.5,fill=red,scale=0.55] (G7) at (6*\x,\z) {};		
	
 	\node (H1) at (-1.2,0) {$\xi_t$};
 	\node (H2) at (-1.2,\y) {$\zeta_t$};	 
	\node (H3) at (-1.2,\z) {$\chi_t$};

	\draw[thick] (0.5+0*\x,0) -- (-0.5+1*\x,0);
	\draw[thick] (0.5+0*\x,\y) -- (-0.5+1*\x,\y);
	\draw[thick] (0.5+1*\x,0) -- (-0.5+2*\x,0);
	\draw[thick] (0.5+1*\x,\y) -- (-0.5+2*\x,\y);
	\draw[thick] (0.5+2*\x,0) -- (-0.5+3*\x,0);
	\draw[thick] (0.5+2*\x,\y) -- (-0.5+3*\x,\y);
	\draw[thick] (0.5+3*\x,0) -- (-0.5+4*\x,0);
	\draw[thick] (0.5+3*\x,\y) -- (-0.5+4*\x,\y);
	\draw[thick] (0.5+4*\x,0) -- (-0.5+5*\x,0);
	\draw[thick] (0.5+4*\x,\y) -- (-0.5+5*\x,\y);
	\draw[thick] (0.5+5*\x,0) -- (-0.5+6*\x,0);
	\draw[thick] (0.5+5*\x,\y) -- (-0.5+6*\x,\y);
	
	\end{tikzpicture}	
\caption{\label{fig:CouplingDisagreementProcesses}Coupling $(\chi_t)_{t\geq 0}$ between the processes $(\zeta_t)_{t\geq 0}$ and $(\xi_t)_{t\geq 0}$ for $N=7$.}
 \end{figure}

%

\subsection{Comparison via a multi-species exclusion process}

Next, we relate the current of second class particles in $(\zeta_t)_{t \geq 0}$ to the motion of the second class particles in $(\xi_t)_{t \geq 0}$. 
The following lemma shows that when at least $4N$ second class particles have exited at the left-hand side boundary in $(\zeta_t)_{t \geq 0}$, all second class particles must have left in $(\xi_t)_{t \geq 0}$, with probability tending to $1$.

\begin{lemma}\label{lem:SecondClassExits} For all $N$ large enough and $T=T(N) \leq N^2$, we have 
\begin{equation}
\mathbf{P} \left( \xi_{T}(x) \neq 2 \text{ for all } x\in [N] \  \Big|  \ - J^{(2)}_{T(N)} > 4N \right) \geq 1-\frac{1}{N},
\end{equation}
where $(J^{(2)}_{t})_{t \geq 0}$ is defined with respect to $(\zeta_t)_{t \geq 0}$.
\end{lemma}

\begin{proof}[Proof of Theorem \ref{thm:asymmetricTwoSides}] The upper bound in Theorem \ref{thm:asymmetricTwoSides} follows from Lemma \ref{lem:SecondClassCurrent} 
and Lemma \ref{lem:SecondClassExits} together with Corollary \ref{cor:SecondClassParticles}. 
\end{proof}

In order to show Lemma \ref{lem:SecondClassExits}, we require a bit of setup. Define the process $(\chi_t)_{t\geq 0}=(\zeta_t,\xi_t)_{t \geq 0}$ and note that under the canonical coupling, $(\chi_t)_{t \geq 0}$ is a Markov process with state space $S^N$ where $S:=\{0,1,2\}^2$. In the following, we will use an alternative interpretation of the process $(\chi_t)_{t \geq 0}$ on the state space $\{0,1,2\}^N$. By construction, every site in $(\chi_t)_{t \geq 0}$ which is not occupied by two first class particles or by two empty sites, must be of the form $(0,2),(2,2),(1,2)$ or $(2,1)$ (for example, note that the configuration $(2,0)$ is not attained since whenever a second class particle is created in $(\zeta_t)_{t \geq 0}$ at the boundary, there has to be a first class particle in $(\xi_t)_{t \geq 0}$). We refer to these configurations as second class particles of types $1$ to $4$, respectively, see Figures \hyperref[fig:CouplingDisagreementProcesses]{6} and \hyperref[fig:SecondClassHierachy]{7}. \\

By definition, $\chi_0$ contains only second class particles of types $1,2$ and $3$, while all second class particles which enter at site $N$ must have type $4$. Among each other, the second class particles of types $i$ and $j$ respect the canonical coupling, i.e.\ a particle of type $j$ has a higher priority than a particle of type $i$ if $i<j$, see Remark \ref{rem:MultiSpecies} (for example a second class particle of type $1$ associated to $(0,2)$ has in both components a lower priority than a second class particle of type $4$ which is associated to $(2,1)$). However, there is one exception: When two second class particles of types $3$ and $4$ are updated, they create the configurations $(2,2)$ and $(1,1)$. In this update mechanism, we call $(1,1)$ a second class particle of type $5$, see Figure \hyperref[fig:SecondClassHierachy]{7}. To the other configuration values $(1,1)$ and $(0,0)$ in $(\chi_t)_{t \geq 0}$, we refer as first class particles and empty sites, respectively. Note that when ignoring the labels of the second class particles, the process $(\chi_t)_{t \geq 0}$ has the same transition rates as $(\zeta_t)_{t \geq 0}$. In particular, entering and exiting of first class particles and empty sites in $(\chi_t)_{t \geq 0}$ is not affected by the types of the second class particles. 
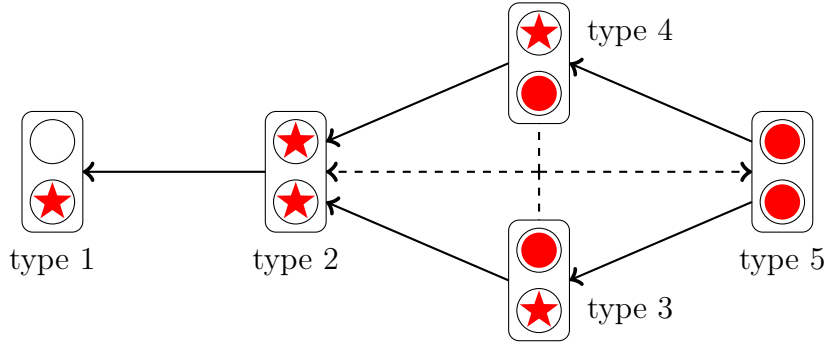
\begin{figure} \label{fig:SecondClassHierachy}
\centering
\begin{tikzpicture}[scale=0.8]

\draw[rounded corners] (0, 0) rectangle (1, 2);

\node (X) at (0.5,-0.5) {type $1$};

\draw[rounded corners] (4, 0) rectangle (5, 2);

\node (X) at (4.5,-0.5) {type $2$};

\draw[rounded corners] (8, -1.8) rectangle (9, 0.2);

\node (X) at (10,-1.3) {type $3$};

\draw[rounded corners] (8, 1.8) rectangle (9, 3.8);

\node (X) at (10,3.3) {type $4$};

\draw[rounded corners] (12, 0) rectangle (13, 2);

\node (X) at (12.5,-0.5) {type $5$};

\draw[thick,arrow] (4, 1) -- (1, 1);
\draw[thick,arrow] (8, -0.8) -> (5, 0.5);
\draw[thick,arrow] (8, 2.8) -> (5, 1.5);
\draw[thick,arrow] (12, 1.5) -> (9, 2.8);
\draw[thick,arrow] (12, 0.5) -> (9,-0.8);

\draw[thick,dashed] (8.5,1) -- (8.5,1.8);
\draw[thick,dashed] (8.5,1) -- (8.5,0.2);
\draw[thick, dashed, arrow] (8.5,1) -- (5, 1);
\draw[thick, dashed, arrow] (8.5,1) -- (12, 1);

\node[shape=circle,scale=1.5,draw] (E) at (0.5,0.5){} ; 
\node[shape=star,star points=5,star point ratio=2.5,fill=red,scale=0.55] (Y1) at (0.5,0.5) {};

\node[shape=circle,scale=1.5,draw] (E) at (0.5,1.5){} ; 

\node[shape=circle,scale=1.5,draw] (E) at (4.5,1.5){} ; 
\node[shape=star,star points=5,star point ratio=2.5,fill=red,scale=0.55] (Y1) at (4.5,1.5) {};

\node[shape=circle,scale=1.5,draw] (E) at (4.5,0.5){} ; 
\node[shape=star,star points=5,star point ratio=2.5,fill=red,scale=0.55] (Y1) at (4.5,0.5) {};

\node[shape=circle,scale=1.5,draw] (E) at (8.5,-1.3){} ; 
\node[shape=star,star points=5,star point ratio=2.5,fill=red,scale=0.55] (Y1) at (8.5,-1.3) {};

\node[shape=circle,scale=1.5,draw] (E) at (8.5,-0.3){} ; 
\node[shape=circle,scale=1.2,fill=red] (Y1) at (8.5,-0.3) {};

\node[shape=circle,scale=1.5,draw] (E) at (8.5,2.3){} ; 
\node[shape=circle,scale=1.2,fill=red] (Y1) at (8.5,2.3) {};

\node[shape=circle,scale=1.5,draw] (E) at (8.5,3.3){} ; 
\node[shape=star,star points=5,star point ratio=2.5,fill=red,scale=0.55] (Y1) at (8.5,3.3) {};

\node[shape=circle,scale=1.5,draw] (E) at (12.5,1.5){} ; 
\node[shape=circle,scale=1.2,fill=red] (Y1) at (12.5,1.5) {};

\node[shape=circle,scale=1.5,draw] (E) at (12.5,0.5){} ; 
\node[shape=circle,scale=1.2,fill=red] (Y1) at (12.5,0.5) {};
	\end{tikzpicture}	
\caption{Visualization of the different types of second class particles. The tip of an arrow between two types indicates which type has the lower priority. The dashed arrows signalize that updating an edge with two second class particles of types $3$ and $4$ creates two second class particles of types $2$ and $5$, respectively.
}
 \end{figure}
We will now investigate the behavior of the different types of second class particles in $(\chi_t)_{t \geq 0}$ among each other using an auxiliary process $(\chi^{\star}_t)_{t \geq 0}$, which will have a similar construction as $(\xi_t^{\ast})_{t\geq0}$ in Section \ref{sec:ShockwaveLemma}. Intuitively, for each $t\geq0$, we obtain $\chi^{\star}_t$ by deleting all sites which are either empty or occupied by a first class particle in $\chi_t$ (here certain edges are merged), and replacing all second class particles of types $1,2,3$ with empty sites, as well as all second class particles of types $4$ or $5$ with first class particles. We then extend $\chi_t$ to a configuration on $\{0,1\}^{\Z}$ by adding particles on the right-hand side, and empty sites (as well as a finite number of particles) on the left-hand side of the segment.  We will see from the formal construction below that $\chi^{\star}_0=\vartheta_0$ (recall \eqref{groundstateconfig}), and that $(\chi^{\star}_t)_{t \geq 0}$ has the law of a simple exclusion process on $\Z$ with censoring, in which the rightmost empty site $R(\chi^{\star}_t)$ is replaced by a first class particle whenever the corresponding second class particle in $(\chi_t)_{t \geq 0}$ exits at site $N$ at time $t$. An edge $e$ is censored for $\chi^{\star}_t$ at time $t$ if and only if it was merged in $\chi_t$ in the deletion step, or if one of its endpoints is occupied by a particle which is not present in $\chi_t$, and thus was only added in the construction when extending the configuration to $\Z$. Note that this censoring scheme does not depend on the different types of the second class particles in $(\chi_t)_{t \geq 0}$. \\

We now give a formal construction of $(\chi^{\star}_t)_{t \geq 0}$. Consider the following procedure which assigns some $\chi^{\star}=\chi^{\star}(v) \in \{0,1\}^\Z$ to every $\chi \in \{ 0,1,2 \}^N$ and every $v=\{ 0,1\}^k$ for $k\in \N \cup \{ 0\}$.
\begin{description}[labelsep=1em]
\item[Step 1] Delete all vertices in $\chi$ which are empty or contain a first class particle.  \vspace{-0.5em} 
\item[Step 2] Concatenate the vector $v$ at the left-hand side of the diminished segment. \vspace{-0.5em}
\item[Step 3] Turn all second class particles to empty sites if they are of type $1,2$ or $3$ and \\ \phantom{abc} turn them into first class particles if they are of type $4$ or $5$.  \vspace{-0.5em}
\item[Step 4] Extend to a configuration $\chi^{\star} \in \{0,1\}^\Z$ by adding empty sites at the left-\\ \phantom{abc} hand side and first class particles at the right-hand side of the segment.
\end{description}
An illustration is given in Figure \hyperref[fig:Projection]{8}. Note that $\chi^{\star}$ in this procedure is only defined up to translations on $\Z$. 
We use this additional degree of freedom when we define the process $(\chi^{\star}_t)_{t \geq 0}$ from $(\chi_t)_{t \geq 0}$. For all $t\geq 0$, let $v=v(t)$ denote the vector of all second class particles which have left the segment at the left-hand side boundary by time $t$. More precisely, we place a $1$ at position $i$ in $v$ if the $i^{\text{th}}$ second class particle exiting is of type $4$ or $5$, and we put a $0$, otherwise. For all $t\geq 0$, we obtain $\chi^{\star}_t$ up to translations by applying the above procedure for $\chi_t$ and $v(t)$. In order to determine the specific translation of $\chi^{\star}_t$ in $(\chi^{\star}_t)_{t \geq 0}$, we proceed as follows. We choose $\chi^{\star}_0 \in A_0$ where $A_0$ is defined in \eqref{blockingsets}. In particular, note that $\chi^{\star}_0=\vartheta_0$ holds. For $t>0$, suppose that $\chi^{\star}_t \in A_n$ holds for some $n\in\Z$. If at time $t$ a second class particle of type $1,2$ or $3$ exits at the right-hand side boundary in $\chi_t$, we choose the updated configuration such that $\chi^{\star}_{t+} \in A_{n-1}$ holds. In all other cases, we choose $\chi^{\star}_{t+} \in A_{n}$. 
\begin{figure} \label{fig:Projection}
\centering
\begin{tikzpicture}[scale=0.59]

\def\x{2.5};
\def\y{-1.7};
\def\z{0.8};
\def\a{0.8};

 	\node[shape=circle,scale=1.5,draw] (C1) at (0,0*\y){} ; 
 	\node[shape=circle,scale=1.5,draw] (C2) at (\x,0*\y){} ; 
 	\node[shape=circle,scale=1.5,draw] (C3) at (2*\x,0*\y){} ; 
 	\node[shape=circle,scale=1.5,draw] (C4) at (3*\x,0*\y){} ; 
 	\node[shape=circle,scale=1.5,draw] (C5) at (4*\x,0*\y){} ; 
 	\node[shape=circle,scale=1.5,draw] (C6) at (5*\x,0*\y){} ; 
 	\node[shape=circle,scale=1.5,draw] (C7) at (6*\x,-0*\y){} ; 
 	
 	\draw[thick] (C1) -- (C2);
 	\draw[thick] (C2) -- (C3);
 	\draw[thick] (C3) -- (C4);
 	\draw[thick] (C4) -- (C5);
 	\draw[thick] (C5) -- (C6);
 	\draw[thick] (C6) -- (C7);

 	\node[scale=\z] (F1) at (0,0*\y-\a) {$3$};
 	\node[scale=\z] (F2) at (\x,0*\y-\a) {$2$};
 	\node[scale=\z] (F4) at (3*\x,0*\y-\a) {$4$};
 	\node[scale=\z] (F6) at (5*\x,0*\y-\a) {$1$};
 	\node[scale=\z] (F7) at (6*\x,0*\y-\a) {$5$};
	\node[shape=circle,scale=1.2,fill=red] (F5) at (C5) {};
	
	\node[shape=star,star points=5,star point ratio=2.5,fill=red,scale=0.55] (G1) at (0,0*\y) {};	
	\node[shape=star,star points=5,star point ratio=2.5,fill=red,scale=0.55] (G2) at (\x,0*\y) {};	
	\node[shape=star,star points=5,star point ratio=2.5,fill=red,scale=0.55] (G4) at (3*\x,0*\y8) {};	
	\node[shape=star,star points=5,star point ratio=2.5,fill=red,scale=0.55] (G6) at (5*\x,0*\y) {};	
	\node[shape=star,star points=5,star point ratio=2.5,fill=red,scale=0.55] (G7) at (6*\x,0*\y) {};		
 
	\node (H3) at (-5,0*\y) {$\chi$};

 	\node[shape=circle,scale=1.5,draw] (C2) at (\x,1*\y){} ; 
 	\node[shape=circle,scale=1.5,draw] (C3) at (2*\x,1*\y){} ; 
 	\node[shape=circle,scale=1.5,draw] (C4) at (3*\x,1*\y){} ; 
 	\node[shape=circle,scale=1.5,draw] (C5) at (4*\x,1*\y){} ; 
 	\node[shape=circle,scale=1.5,draw] (C6) at (5*\x,1*\y){} ;

 	\draw[thick] (C2) -- (C3);
 	\draw[thick, dashed] (C3) -- (C4);
 	\draw[thick, dashed] (C4) -- (C5);
 	\draw[thick] (C5) -- (C6);

 	\node[scale=\z] (F1) at (\x,1*\y-\a) {$3$};
 	\node[scale=\z] (F2) at (2*\x,1*\y-\a) {$2$};
 	\node[scale=\z] (F4) at (3*\x,1*\y-\a) {$4$};
 	\node[scale=\z] (F6) at (4*\x,1*\y-\a) {$1$};
 	\node[scale=\z] (F7) at (5*\x,1*\y-\a) {$5$};

	\node[shape=star,star points=5,star point ratio=2.5,fill=red,scale=0.55] (G1) at (\x,1*\y) {};	
	\node[shape=star,star points=5,star point ratio=2.5,fill=red,scale=0.55] (G2) at (2*\x,1*\y) {};	
	\node[shape=star,star points=5,star point ratio=2.5,fill=red,scale=0.55] (G4) at (3*\x,1*\y) {};	
	\node[shape=star,star points=5,star point ratio=2.5,fill=red,scale=0.55] (G6) at (4*\x,1*\y) {};	
	\node[shape=star,star points=5,star point ratio=2.5,fill=red,scale=0.55] (G7) at (5*\x,1*\y) {};

	\node (H3) at (-5,1*\y) {Step 1};

 	\node[shape=circle,scale=1.5,draw] (C0) at (-\x,2*\y){} ; 
 	\node[shape=circle,scale=1.5,draw] (C1) at (0,2*\y){} ; 
 	\node[shape=circle,scale=1.5,draw] (C2) at (\x,2*\y){} ; 
 	\node[shape=circle,scale=1.5,draw] (C3) at (2*\x,2*\y){} ; 
 	\node[shape=circle,scale=1.5,draw] (C4) at (3*\x,2*\y){} ; 
 	\node[shape=circle,scale=1.5,draw] (C5) at (4*\x,2*\y){} ; 
 	\node[shape=circle,scale=1.5,draw] (C6) at (5*\x,2*\y){} ; 
 
 	\draw[thick, dashed] (C0) -- (C1);	 
 	\draw[thick, dashed] (C1) -- (C2);		
 	\draw[thick] (C2) -- (C3);
 	\draw[thick, dashed] (C3) -- (C4);
 	\draw[thick, dashed] (C4) -- (C5);
 	\draw[thick] (C5) -- (C6);

 	\node[scale=\z] (F1) at (\x,2*\y-\a) {$3$};
 	\node[scale=\z] (F2) at (2*\x,2*\y-\a) {$2$};
 	\node[scale=\z] (F4) at (3*\x,2*\y-\a) {$4$};
 	\node[scale=\z] (F6) at (4*\x,2*\y-\a) {$1$};
 	\node[scale=\z] (F7) at (5*\x,2*\y-\a) {$5$};

	\node[shape=star,star points=5,star point ratio=2.5,fill=red,scale=0.55] (G1) at (\x,2*\y) {};	
	\node[shape=star,star points=5,star point ratio=2.5,fill=red,scale=0.55] (G2) at (2*\x,2*\y) {};	
	\node[shape=star,star points=5,star point ratio=2.5,fill=red,scale=0.55] (G4) at (3*\x,2*\y) {};	
	\node[shape=star,star points=5,star point ratio=2.5,fill=red,scale=0.55] (G6) at (4*\x,2*\y) {};	
	\node[shape=star,star points=5,star point ratio=2.5,fill=red,scale=0.55] (G7) at (5*\x,2*\y) {};		

	\node[shape=circle,scale=1.2,fill=red] (G0) at (0,2*\y) {};

	\node (H3) at (-5,2*\y) {Step 2};

 	\node[shape=circle,scale=1.5,draw] (C0) at (-\x,3*\y){} ; 
 	\node[shape=circle,scale=1.5,draw] (C1) at (0,3*\y){} ; 
 	\node[shape=circle,scale=1.5,draw] (C2) at (\x,3*\y){} ; 
 	\node[shape=circle,scale=1.5,draw] (C3) at (2*\x,3*\y){} ; 
 	\node[shape=circle,scale=1.5,draw] (C4) at (3*\x,3*\y){} ; 
 	\node[shape=circle,scale=1.5,draw] (C5) at (4*\x,3*\y){} ; 
 	\node[shape=circle,scale=1.5,draw] (C6) at (5*\x,3*\y){} ;

 	\draw[thick, dashed] (C0) -- (C1); 	
 	\draw[thick, dashed] (C1) -- (C2); 	
 	\draw[thick] (C2) -- (C3);
 	\draw[thick, dashed] (C3) -- (C4);
 	\draw[thick, dashed] (C4) -- (C5);
 	\draw[thick] (C5) -- (C6);

	\node[shape=circle,scale=1.2,fill=red] (F3) at (3*\x,3*\y) {};
	\node[shape=circle,scale=1.2,fill=red] (F5) at (5*\x,3*\y) {}; 
 
	\node[shape=circle,scale=1.2,fill=red] (G0) at (0,3*\y) {};

	\node (H3) at (-5,3*\y) {Step 3};

 	\node[shape=circle,scale=1.5,draw] (C0) at (-\x,4*\y){} ; 
 	\node[shape=circle,scale=1.5,draw] (C1) at (0,4*\y){} ; 
 	\node[shape=circle,scale=1.5,draw] (C2) at (\x,4*\y){} ; 
 	\node[shape=circle,scale=1.5,draw] (C3) at (2*\x,4*\y){} ; 
 	\node[shape=circle,scale=1.5,draw] (C4) at (3*\x,4*\y){} ; 
 	\node[shape=circle,scale=1.5,draw] (C5) at (4*\x,4*\y){} ; 
 	\node[shape=circle,scale=1.5,draw] (C6) at (5*\x,4*\y){} ; 
 	\node[shape=circle,scale=1.5,draw] (C7) at (6*\x,4*\y){} ;  
 	\node[shape=circle,scale=1.5,draw] (C8) at (7*\x,4*\y){} ;  
 	
  	\draw[thick, dashed] (C0) -- (-0.8-\x,4*\y) ; 	
 	\draw[thick, dashed] (C0) -- (C1) ; 	
 	\draw[thick, dashed] (C1) -- (C2); 	
 	\draw[thick] (C2) -- (C3);
 	\draw[thick, dashed] (C3) -- (C4);
 	\draw[thick, dashed] (C4) -- (C5);
 	\draw[thick] (C5) -- (C6);
 	\draw[thick, dashed] (C6) -- (C7);
 	\draw[thick, dashed] (C7) -- (C8);
 	\draw[thick, dashed] (C8) -- (7*\x+0.8,4*\y);

	\node[shape=circle,scale=1.2,fill=red] (F3) at (3*\x,4*\y) {};
	\node[shape=circle,scale=1.2,fill=red] (F5) at (5*\x,4*\y) {}; 
	\node[shape=circle,scale=1.2,fill=red] (F6) at (6*\x,4*\y) {};  
	\node[shape=circle,scale=1.2,fill=red] (F7) at (7*\x,4*\y) {};  
 
	\node[shape=circle,scale=1.2,fill=red] (G0) at (0,4*\y) {};

	\node (H3) at (-5,4*\y) {Step 4};	 
	
	\end{tikzpicture}
\caption{Construction of $\chi^{\star}$ from $\chi$ for $v=(0,1)$. Censored edges are drawn dashed.}
 \end{figure}
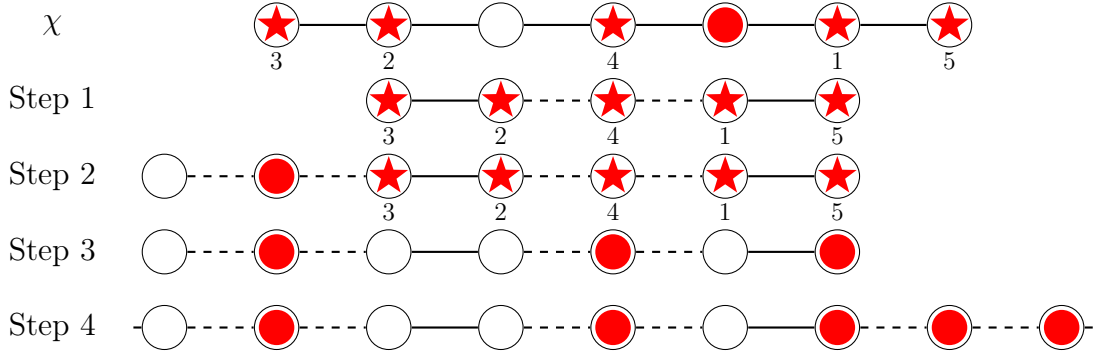 
 The next lemma states that the position of the leftmost particle $(L(\chi^{\star}_t))_{t \geq 0}$ is close to the position of the rightmost empty site $(R(\chi^{\star}_t))_{t \geq 0}$. 
 
 \begin{lemma}\label{lem:RightmostLeftmostStar} There exists a constant $c>0$ such that 
 \begin{equation}\label{eq:RightmostLeftmostEstimate}
 \mathbf{P} \left( \abs{R(\chi^{\star}_T)- L(\chi^{\star}_T)} > c \log N + N \right) \leq \frac{1}{N}
\end{equation} holds for all $N$ sufficiently large and $T\leq N^2$.
 \end{lemma}

 \begin{proof} Let $(\eta^{0}_t)_{t \geq 0}$ and $(\eta^{-N}_t)_{t \geq 0}$ be two simple exclusion processes on $A_{0}$ and $A_{-N}$ with initial states $\vartheta_{0}$ and $\vartheta_{-N}$, respectively. We let $(\eta^{0}_t)_{t \geq 0}$ and $(\eta^{-N}_t)_{t \geq 0}$ be canonically coupled to $(\chi^{\star}_t)_{t \geq 0}$, and apply the same censoring scheme. Since $(\eta^{0}_t)_{t \geq 0}$ and $(\chi^{\star}_t)_{t \geq 0}$ differ only by the fact that in $(\chi^{\star}_t)_{t \geq 0}$ occasionally the right-most empty site is replaced by a particle, we see that $R(\chi^{\star}_t) \leq R(\eta^{0}_t)$ holds almost surely for all $t\geq0$. Further, we claim that $L(\chi^{\star}_t) \geq L(\eta^{-N}_t)$ holds almost surely for all $t\geq0$.
 This can be seen by conditioning on the at most $N$ times at which the rightmost empty site in $(\chi^{\star}_t)_{t \geq 0}$ gets replaced, and then using an induction argument.  Since the above way of prohibiting updates in $(\chi^{\star}_t)_{t \geq 0}$ is indeed a censoring scheme in the sense of Section \ref{sec:Censoring}, we use the censoring inequality from Remark \ref{rem:CensoringBlocking} to see that the laws of $(\eta^{0}_t)_{t \geq 0}$ and $(\eta^{-N}_t)_{t \geq 0}$ are stochastically dominated by the blocking measures on $A_{0}$ and $A_{-N}$, respectively, with respect to the partial order $\succeq_{\h}$. The statement in \eqref{eq:RightmostLeftmostEstimate} now follows from Lemma \ref{lem:HittingBlockingMeasure}. \end{proof}
 \begin{proof}[Proof of Lemma \ref{lem:SecondClassExits}] Note that when the current of second class particles is at most $-4N$ at time $T$, we know that at least $4N$ second class particles are absorbed at the left-hand side boundary in $(\chi_t)_{t \geq 0}$ at time $T$. Note that in this case, at least $2N$ of them must be of type $4$ since all second class particles created at site $N$ are of type $4$, and there are at most $N$ second class particles of types $1,2,3$ initially in the segment. By Lemma \ref{lem:RightmostLeftmostStar}, we see that with probability at least $1-N^{-1}$, each second class particle of type $1$, $2$ or $3$ in $\chi_T$ has at most $c\log N+N$ second class particles of type $4$ or $5$ to its left (counting also particles which have exited at site $1$). Hence for all $N$ large enough, all second class particles in $(\chi_t)_{t \geq 0}$ of type $1$, $2$ or $3$, and thus also all second class particles in $(\xi_t)_{t \geq 0}$, have left the segment by time $T$ with probability at least $1-N^{-1}$.  \end{proof}
 
 \begin{remark}\label{rem:IntuitionMax} For the simple exclusion process in the maximal current phase, we conjecture that a similar analysis of the disagreement process started from $\mathbf{1}$ and $\mathbf{0}$ yields the order of the $\varepsilon$-mixing time. We believe that the typical time for all second class particles to leave the segment is of order $N^{3/2}$, using a comparison to the typical fluctuations of a second class particle on $\Z$ in a Bernoulli-$\frac{1}{2}$-product measure  \cite{BS:OrderCurrent}. Further, note that the exponent $\frac{3}{2}$ is the KPZ relaxation scale which has been proved by Baik and Liu for periodic models as well as by Corwin and Dimitrov for the ASEP on $\Z$, see \cite{BL:TASEPring,CD:ASEPline}, and more broadly is present in all KPZ class models. Moreover, Corwin and Shen, as well as Parekh showed that under a weakly asymmetry scaling, the height function (suitably normalized) of the simple exclusion process with open boundaries in the triple point converges to a solution of the KPZ equation, see \cite{CS:OpenASEPWeakly,P:KPZlimit}. This supports   Conjecture \ref{conj:MaxCurrent} for the maximal current phase of a mixing time of order $N^{3/2}$, and no cutoff.
 \end{remark}

\section{Mixing times for the triple point}\label{sec:MaxCurrent} 

In this section, we prove Theorem \ref{thm:MaxCurrentTriplePoint} for the simple exclusion process $(\eta_t)_{t \geq 0}$ with open boundaries and parameters $(p,\alpha,\beta,\gamma,\delta)$ in the triple point.  We use a symmetrization argument, similar to the one presented in \cite{F:EVBoundsSEP} for the case of the totally asymmetric simple exclusion process on the circle. The main technique used is a Nash inequality as introduced in \cite{DS:GeometricBounds}. We compare the total-variation distance between the law of $(\eta_t)_{t \geq 0}$ and its stationary distribution $\mu$ to the spectral gap of a process $(\zeta_t)_{t \geq 0}$, i.e.\ the absolute value of the largest non-zero eigenvalue of the generator for $(\zeta_t)_{t \geq 0}$.  
We start by defining the \textbf{adjoint} $\mathcal{L}^{\star}$ of the generator $\mathcal{L}$ of the simple exclusion process $(\eta_t)_{t \geq 0}$ with open boundaries. This is the linear operator which satisfies
\begin{equation*}
\sum_{\eta \in \Omega_{N}} f(\eta) (\mathcal{L}g)(\eta) \mu(\eta) =  \sum_{\eta \in \Omega_{N}}  ( \mathcal{L}^{\star} f )(\eta) g(\eta) \mu(\eta)
\end{equation*} for all functions $f,g \colon \Omega_N \rightarrow \mathbb{R}$. In particular, note that for reversible processes, we have that $\mathcal{L}=\mathcal{L}^{\star}$ holds, see \eqref{def:Reversibility}. By Lemma \ref{lem:stationaryDistribution}, we have that the stationary distribution $\mu$ of $(\eta_t)_{t \geq 0}$ is the uniform measure on $\Omega_N$. Hence, observe that 
the simple exclusion process with open boundaries and parameters $(1-p,\gamma,\delta,\alpha,\beta)$ has generator $\mathcal{L}^{\ast}$.  We now consider the additive symmetrization of the simple exclusion process $(\eta_t)_{t \geq 0}$ with open boundaries with generator $\mathcal{L}$ and the simple exclusion process generated by its adjoint $\mathcal{L}^{\ast}$. More precisely, we let $(\zeta_t)_{t \geq 0}$ be the Feller process on $\Omega_N$ generated by $\frac{1}{2}(\mathcal{L}^{\star}+\mathcal{L})$. Observe that $(\zeta_t)_{t \geq 0}$ is reversible with respect to $\mu$.  Moreover, $(\zeta_t)_{t \geq 0}$ has the law of a simple exclusion process with open boundaries for parameters $(p^{\prime},\alpha^{\prime},\beta^{\prime},\gamma^{\prime},\delta^{\prime})$ given by
\begin{equation*}
p^{\prime} = \frac{1}{2} , \quad \alpha^{\prime}=\gamma^{\prime}=\frac{\alpha+\gamma}{2} \quad \text{ and } \quad  \beta^{\prime} = \delta^{\prime} = \frac{\beta+\delta}{2} \ .
\end{equation*}
The next lemma relates the total-variation distance of $(\eta_t)_{t \geq 0}$ to the spectral gap of $(\zeta_t)_{t \geq 0}$. It is an immediate consequence of Theorem 2.14 in \cite{F:EVBoundsSEP}.
\begin{lemma}\label{lem:DiaconisSeries} Let $\lambda$ denote the spectral gap of $(\zeta_t)_{t \geq 0}$. We have that
\begin{equation}\label{eq:DiaconisEstimate}
\TV{\P_{\xi}(\eta_t \in \cdot)- \mu} \leq 2^{N/2+1}\exp(-\lambda t )
\end{equation} holds for all initial states $\xi \in \Omega_N$ and $t \geq 0$.
\end{lemma}
\begin{proof}[Proof of Theorem \ref{thm:MaxCurrentTriplePoint}]
By Remark \ref{rem:relaxSymmetric}, we see that $\lambda^{-1} \leq CN^2$ holds for some constant $C=C(\alpha,\beta,\gamma,\delta)$, and we conclude by applying Lemma \ref{lem:DiaconisSeries}.
\end{proof}

\bibliographystyle{plain}

\begin{footnotesize}
\bibliography{Openboundary}
\end{footnotesize}

\textbf{Acknowledgment} We thank Noam Berger, Ivan Corwin, David Criens, Hubert Lacoin, Allan Sly, Herbert Spohn, and Lauren Williams for helpful discussions and comments. We are indebted to the anonymous referee for a careful reading and pointing out several inaccuracies.
This project was started at Cambridge University and carried out during visits of the authors at Princeton University and Technical University of Munich. We thank all three institutions for their hospitality. The third author acknowledges the TopMath program and the Studienstiftung des deutschen Volkes for financial support.

\appendix

\section{Appendix}

\subsection{Proof of Lemma \ref{lem:HittingBlockingMeasure}}

To show Lemma \ref{lem:HittingBlockingMeasure}, we bound the hitting time $\tau_{0}$ of the ground state $\vartheta_{0}$ in $A_0$.

\begin{lemma}\label{lem:GroundStateHitting} For $x\geq 0$, let $\theta_{x} \in A_0$ with $\theta_{x}(y) := \mathds{1}_{\{-x \leq y < 0\}} + \mathds{1}_{\{y > x\}}$ for all $y \in \Z$ be the initial state for the simple exclusion process $(\eta^{\Z}_t)_{t\geq 0}$ on $A_0$. Then there exists some $c>0$ such that for all $x\geq 0$, we have that 
$\E_{\theta_{x}}[\tau_{0}] \leq c x$ 
holds.
\end{lemma} 

\begin{proof} For all $x\geq 0$, we define $B_x$ to be the set of configurations
\begin{equation}\label{def:SetOfMaximalDistance}
B_x := \left\{ \eta \in A_{0}\colon \max(R(\eta),-L(\eta)) >  x  \right\} 
\end{equation} and denote for all $s \geq 0$ by $
\tau_{B^{\c}_x}^s := \inf \left\{ t \geq s \colon \eta_t \notin B_x \right\}$ the first time after time $s$ when we hit the set $B^{\c}_x$. We claim that there exists some $\tilde{c}>0$ such that for all $x,s \geq 0$
\begin{equation}\label{eq:TauBxCBound}
E_{\theta_{x}}[\tau_{B^{\c}_x}^s] -s \leq \tilde{c} \ .
\end{equation} To see this, let $(\eta^{x}_t)_{t \geq 0}$ and $(\eta^{-x-1}_t)_{t \geq 0}$ be two exclusion processes on $A_{x}$ and $A_{-x-1}$, started from the blocking measure, respectively. Using Remark \ref{rem:PartialOrderonZ}, we note that
\begin{equation}\label{eq:Domination2Blocking}
\mathbf{P}\left( R(\eta^{\Z}_t) \leq  R(\eta^{x}_t) \mbox{ and }  L(\eta^{\Z}_t) \geq L(\eta^{-x-1}_t) \text{ for all } t\geq 0 \right) = 1
\end{equation} holds with respect to the canonical coupling $\mathbf{P}$, see also Figure \hyperref[fig:BlockingDoubleDomination]{9}. Moreover, note that $(\eta^{x}_t,\eta^{-x-1}_t)_{t \geq 0}$ is a stationary and positive recurrent Feller process for which the state $(\vartheta_{x},\vartheta_{-x-1})$ has a strictly positive probability in equilibrium, and that $\tau_{B^{\c}_x}^s \leq T$ whenever $(\eta^{x}_t,\eta^{-x-1}_t)_{t \geq 0}$ is in the state $(\vartheta_{x},\vartheta_{-x-1})$ at time $T\geq s$. We conclude \eqref{eq:TauBxCBound} using  
Kac's lemma for the embedded discrete chain of $(\eta^{x}_t,\eta^{-x-1}_t)_{t \geq 0}$, see Theorem 21.12 in \cite{LPW:markov-mixing},  and a time-change. Next, by Theorem 1.9 in \cite{BBHM:MixingBias},
\begin{equation}\label{eq:FastHittingThetaY}
\P_{\theta_{x}}\left( \tau_{0} \leq  c_1 x \right) \geq c_2
\end{equation} holds for all $x \geq 0$ with constants $c_1,c_2>0$. We claim that together with \eqref{eq:TauBxCBound}, this yields
\begin{equation}\label{eq:RecursionHittingTimeTau0}
\E_{\theta_{x}}[\tau_{0}] \leq c_2 c_1 x + (1-c_2)\left( c_1 x + \tilde{c} + \E_{\theta_{x}}[\tau_{0}]  \right) \ .
\end{equation} To see this, note that with probability at least $c_2$, we hit $\vartheta_{0}$ by time $c_1x$. Suppose that $\vartheta_{0}$ was not hit by time $c_1x$, then we can wait until hitting $B^{\c}_x$ and use \eqref{eq:TauBxCBound}. Since $\eta \preceq_{\h} \theta_x$ holds for all $\eta \in B^{\c}_x$, the hitting time of $\vartheta_{0}$ starting from the configuration at time $\tau_{B^{\c}_x}^{c_1x}$ is stochastically dominated by the hitting time of $\vartheta_{0}$ when starting from $\theta_x$. Now take expectations to get \eqref{eq:RecursionHittingTimeTau0}. Since $\E_{\theta_{x}}[\tau_{0}] < \infty$, we conclude by solving \eqref{eq:RecursionHittingTimeTau0} for $\E_{\theta_{x}}[\tau_{0}] $.
\end{proof}

Next, we study the return time $\tau^{+}_{B_x}:= \inf \left\{ t \geq  \tau_{B^{\c}_x} \colon \eta_t \in B_x \right\}$ to the set ${B_x}$.
\begin{lemma}\label{lem:HittingTimeIntermediate} There exists some $C>0$ such that for all $x\geq 1$
\begin{equation}\label{eq:HittingTimeBound0}
E_{\nu_{(0)}}[\tau^{+}_{B_x}] \geq \nu_{(0)}(\vartheta_0) \E_{\vartheta_0} [\tau_{B_x}]\geq\frac{C}{x}\left(\frac{p}{1-p}\right)^{x} \ .
\end{equation} 
\end{lemma}  
\begin{proof} Observe that an exclusion process in ${B^{\c}_x}$ can change its state if and only if a clock on the sites $[-x,x]$ rings.  Hence, using Kac's lemma for the embedded discrete chain, we see that
\begin{equation}\label{eq:HittingTimeBound1}
\E_{\nu_{(0)}(\ .\  | {B_x})}[\tau^+_{B_x}] \geq \frac{1}{(2x+1)\nu_{(0)}({B_x})} \geq \frac{c_1}{x} \left( \frac{p}{1-p}\right)^x 
\end{equation} 
holds for all $x\geq0$ and some constant $c_1>0$. Since $\vartheta_0 \preceq_{\h} \eta$ for all $\eta \in B^{\c}_x$, we get
\begin{align}\label{eq:HittingTimeBound2}
\E_{\nu_{(0)}(\ .\ | {B_x})}[\tau^+_{B_x}] &=
 \sum_{\zeta \in B^{\c}_x} \left(  \E_{\nu_{(0)}(\ .\  | {B_x})}\left[\tau_\zeta \mid \eta_{\tau_{B^{\c}_x}} = \zeta  \right] + \E_{\zeta}[\tau_{B_x}] \right) \P_{\nu_{(0)}(\ .\  | {B_x})}( \eta_{\tau_{B^{\c}_x}} = \zeta )  \nonumber \\
 &\leq \E_{\nu_{(0)}(\ .\  | {B_x})}[\tau_0] + \E_{\vartheta_0}[\tau_{B_x}] \ .
\end{align} Recall that $\eta \preceq_{\h} \theta_x$ for all $\eta \in B^{\c}_x$. Note that there exists some $c_2>0$ such that 
\begin{equation}\label{eq:HittingTimeBound3}
\E_{\nu_{(0)}(\ .\  | {B_x})}[\tau_0] = \sum_{y \geq x} \sum_{\eta \in B_y\setminus B_{y+1}} \E_{\eta}[\tau_0]   \nu_{(0)}( \eta | {B_x})\leq \sum_{y \geq x} \E_{\theta_{y+1}}[\tau_0] \nu_{(0)}( B_y | {B_x}) \leq c_2 x 
\end{equation} holds for all $x\geq0$, using Lemma \ref{lem:GroundStateHitting} and the fact that $\nu_{(0)}( B_y  | {B_x}) \leq c_3 ((1-p)/p)^{y-x}$ for some $c_3>0$ in the last inequality. Combining \eqref{eq:HittingTimeBound2} and \eqref{eq:HittingTimeBound3}, we see that
\begin{equation*}
\E_{\vartheta_0}[\tau_{B_x}] \geq \E_{\nu_{(0)}(\ .\  | {B_x})}[\tau^+_{B_x}] -   \E_{\nu_{(0)}(\ .\  | {B_x})}[\tau_0]  \geq \E_{\nu_{(0)}(\ .\  | {B_x})}[\tau^+_{B_x}] - c_2 x \, .
\end{equation*}
Together with the lower bound on $\E_{\nu_{(0)}(\ .\  | {B_x})}[\tau^+_{B_x}]$ from \eqref{eq:HittingTimeBound1}, this yields \eqref{eq:HittingTimeBound0}.
\end{proof}
\begin{proof}[Proof of Lemma \ref{lem:HittingBlockingMeasure}]  We will prove Lemma \ref{lem:HittingBlockingMeasure} by contradiction. Take $C>0$ from Lemma \ref{lem:HittingTimeIntermediate} and assume that \eqref{eq:HittingBlocking} is not true. Then,
using the general fact that  for arbitrary events $A$ and $B$, the inequality $\P(A \cap B) \geq \P(A)- \P(B^{\c})$ holds, we have
\begin{equation*}
q:=\P_{\nu_{(0)}}\left( \eta_t \in B_x \text{ for some } t \in \left[ 0, \frac{\varepsilon C}{x} \left(\frac{p}{1-p}\right)^{x}\right] \text{ and } \eta_0 \in B^{\c}_x \right) > 2\varepsilon - \nu_{(0)}(B_x)\, .
\end{equation*} 
A similar argument as for \eqref{eq:RecursionHittingTimeTau0} yields
\begin{equation}\label{eq:Recursion2.0}
\E_{\nu_{(0)}}[\tau^+_{B_x}]  \leq  q \frac{\varepsilon C}{x} \left(\frac{p}{1-p}\right)^{x}  + (1-q) \left(\frac{\varepsilon C}{x} \left(\frac{p}{1-p}\right)^{x} + \E_{\nu_{(0)}}[\tau^+_{B_x}] \right)  \ .
\end{equation}
Solving \eqref{eq:Recursion2.0} for $\E_{\nu_{(0)}}[\tau^+_{B_x}]$, and using the definition of $\nu_{(0)}$ for $q$, we see that for all $x$ large enough $\E_{\nu_{(0)}}[\tau^+_{B_x}] < \varepsilon C{x}^{-1} (p/(1-p))^{x}$ holds. This contradicts Lemma \ref{lem:HittingTimeIntermediate}.
\begin{figure}
\centering
\begin{tikzpicture}[scale=1,rotate=45]
\draw[line width=2pt,red] (0,5) to (0,1.8) to (1.5,1.8) to (1.5,0) to (5,0);
\clip (0,0) to (5,0) to (0,5) to (0,0);
\draw [color=gray!20]  [step=3mm] (0,0) grid (5,5);
\draw[thick] (0,0) to (0,5);
\draw[thick] (0,0) to (5,0);
\draw[dashed,line width=1pt,color=red!60] (1.5,1.8) to (1.5,3.5);
\draw[dashed,line width=1pt,color=red!60] (1.5,1.8) to (3.5,1.8);
\draw[line width=2pt,red] (0,5) to (0,1.8) to (1.5,1.8) to (1.5,0) to (5,0);
\draw[domain=0.1:3.35,smooth,variable=\x, line width=1pt] plot ({\x},{1/(3*\x)+1.8});
\draw[domain=1.6:4.82,smooth,variable=\x, line width=1pt] plot ({\x},{1/(3*(\x-1.5)});
\end{tikzpicture}
\caption{\label{fig:BlockingDoubleDomination}The initial state $\theta_x$ of $(\eta^{\Z}_t)_{t \geq 0}$ is shown in red. The position of the leftmost particle in $(\eta^{\Z}_t)_{t \geq 0}$ is stochastically dominated by the position of the leftmost particle in $(\eta^{-x-1}_t)_{t \geq 0}$.  
A similar statement holds for the rightmost empty site in $(\eta^{\Z}_t)_{t \geq 0}$.}
\end{figure}
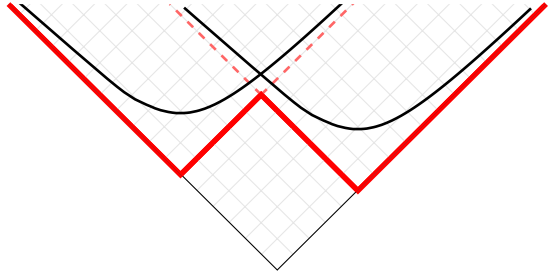
\end{proof}

\subsection{Proof of the generalized version of Wilson's lemma}


\begin{proof}[Proof of Lemma \ref{lem:GeneralizedWilsonContinuousTime}]  

For fixed $X_0= \eta$, let $f(t):= \E[F(X_t)]= \E_\eta[F(X_t)]$ for all $t\geq 0$, and note that 
\begin{equation*}
	f^{\prime}(t) = \E\left[ (\mathcal{A}F)(X_t)\right] \in [-\lambda f(t) - c, -\lambda f(t)+c] \ \text{ for all  } t \geq 0
\end{equation*} by using the martingale property of $(M_t)_{t \geq 0}$ and \eqref{eq:Eigenfunction}. Applying Gronwall's lemma, we get
\begin{equation*}
f(t) \leq f(0)e^{-\lambda t} + \int_{0}^{t} c e^{-\lambda(t-s)}\diff s \leq f(0)e^{-\lambda t} + \frac{c}{\lambda} \ \text{ for all  } t \geq 0, 
\end{equation*}  see Lemma 2.7 in \cite{T:ODEs}. Similarly, apply Gronwall's lemma to $-f$ to conclude that
\begin{equation}\label{eq:FirstOrderEigenfunction}
\abs{f(t) - e^{-\lambda t}f(0)} \leq \frac{c}{\lambda} \ \text{ for all  } t \geq 0.
\end{equation} Next, we define $g(t) := \E[(F(X_t))^2]$. Observe that $(F(X_t))_{t \geq 0}$ is a semimartingale. Thus, we apply It\^o's formula to see that
\begin{align*}
F^2(X_t) - F^2(X_0) &= 2 \int_0^t F(X_s) \diff\big[ F(X_s)-\int_0^s(\mathcal{A}F)(X_r)\diff r \big] \\
&+ 2 \int_0^t F(X_s) \diff \big[ \int_0^s(\mathcal{A}F)(X_r)\diff r \big] + \frac{1}{2} \int_0^t 2\ \diff \langle M \rangle_s
\end{align*} holds, see Theorem 5.33 in \cite{L:Book3}. Taking expectations and changing the order of integration, a calculation yields that
\begin{equation*}
g(t)-g(0) = 2 \int_0^t \E\left[ F(X_s)(\mathcal{A}F)(X_s)\right] \diff s + \E\left[\langle M \rangle_t\right] \ \text{ for all  } t \geq 0.
\end{equation*} Now taking derivatives gives us that
\begin{equation*}
g^{\prime}(t) = 2 \E\left[ F(X_t)(\mathcal{A}F)(X_t)\right]  + \frac{\diff}{\diff t} \E[\langle M \rangle_t]  \ \text{ for all  } t \geq 0.
\end{equation*} Moreover, using \eqref{eq:Eigenfunction}, we obtain that
\begin{equation*}
2\E\left[ F(X_t)(\mathcal{A}F)(X_t)\right]  \leq -2\lambda g(t) + 2c || F ||_{\infty} 
\end{equation*} holds. 
Further, by applying Gronwall's lemma and using \eqref{eq:QVBound}, a calculation shows that
\begin{equation*}
g(t) \leq g(0)e^{-2\lambda t} + \frac{c || F ||_{\infty}}{\lambda} + \int_0^t \left(\frac{\diff}{\diff s} \E[\langle M \rangle_s]\right) e^{-2\lambda(t-s)} \diff s \leq  g(0)e^{-2\lambda t} + \frac{c || F ||_{\infty}+R}{\lambda}
\end{equation*} holds for all $t \geq 0$. Together with \eqref{eq:FirstOrderEigenfunction} and the fact that $g(0)=f(0)^2$, we deduce that
\begin{equation}\label{eq:VarianceBoundWilson}
\V(F(X_t))= \V_{\eta}(F(X_t)) = g(t)-f(t)^2 \leq \frac{3c || F ||_{\infty} +R}{\lambda}
\end{equation} holds for any initial state $\eta \in S$, and all $t \geq 0$. Recall the total-variation distance from \eqref{def:TVDistance} and let $d_\eta(t)$ denote the total-variation distance between the law of $X_t$ started from $\eta$ and its stationary distribution. Note that for all $t \geq 0$ and any initial state $\eta$, we have
\begin{equation*}
 \P\left( F(X_\infty) \geq \frac{1}{2} \E[F(X_t)]\right) \leq  \P\left( F(X_\infty)^2 \geq \frac{1}{4} \E[F(X_t)]^2\right) \leq 4 \frac{\E[F(X_\infty)^2]}{\E[F(X_t)]^2}
\end{equation*} and hence
\begin{align}\label{eq:LowerBoundWilsonDistance}
d_\eta(t) &\geq \P\left( F(X_t) \geq \frac{1}{2} \E[F(X_t)]\right) - \P\left( F(X_\infty) \geq \frac{1}{2} \E[F(X_t)]\right)\cr
&\geq  1 - 4 \frac{\V(F(X_t))}{\E[F(X_t)]^2} - 4 \frac{\V(F(X_\infty)) + \E [F(X_\infty)]^2}
{\E[F(X_t)]^2}
\end{align} where we used Chebyshev's inequality for the second inequality.
Here, $X_{\infty}$ is a random variable whose law is the stationary distribution of $(X_t)_{t \geq 0}$.  The goal is to show that for $t$ equal to the right-hand side of \eqref{eq:LowerBoundWilson}, 
the right-hand side of \eqref{eq:LowerBoundWilsonDistance} is $\geq 1- \varepsilon$, which implies \eqref{eq:LowerBoundWilson}.
Let $\eta$ be such that $\abs{F(\eta)}=|| F ||_{\infty}$ holds. Then, to estimate the denominator of the last term in \eqref{eq:LowerBoundWilsonDistance}
for $t$ equal to the right-hand side of \eqref{eq:LowerBoundWilson}, note that by \eqref{eq:FirstOrderEigenfunction},
$$
\E[F(X_t)] \geq e^{-\lambda t}F(X_0) - \frac{c}{\lambda} = e^{-\lambda t}|| F ||_{\infty}- \frac{c}{\lambda} \geq \frac{1}{2}e^{-\lambda t}|| F ||_{\infty}\, ,
$$
where the last inequality is due to our choice of $t$. To estimate the nominator of the last term in \eqref{eq:LowerBoundWilsonDistance},
taking $t \rightarrow \infty$ in \eqref{eq:FirstOrderEigenfunction} and \eqref{eq:VarianceBoundWilson}, we see that 
$|\E[F(X_\infty)] |\leq c/\lambda$  and $\abs{\V[F(X_\infty)]} \leq (3c || F ||_{\infty}+R)\lambda^{-1}$. Now we see after a calculation that indeed the right-hand side of \eqref{eq:LowerBoundWilsonDistance} is $\geq 1- \varepsilon$.
\end{proof}

\end{document}